\documentclass[11pt]{article}
\usepackage{amsmath}
\usepackage{amsthm}
\usepackage{amsfonts}
\usepackage{amssymb}
\usepackage{amsrefs}
\usepackage{hyperref}

\usepackage{mathtools}
\usepackage{bbm}

\usepackage{cases}
\usepackage{xcolor}

\newcommand{\vech}[1]{\vec{#1}_h}

\newcommand{\dt}{\partial_t}
\newcommand{\dv}{\mathrm{div}\,}
\newcommand{\dvh}{\mathrm{div}_h\,}
\newcommand{\nablah}{\nabla_h}
\newcommand{\deltah}{\Delta_h}
\newcommand{\dz}{\partial_z}
\newcommand{\inth}{\int_{\Omega_h}}
\newcommand{\intw}{\int_{\Omega}}
\newcommand{\idxh}{\,d\vec{x}_h}
\newcommand{\idx}{\,d\vec{x}}
\newcommand{\subeqref}[2]{$\eqref{#1}_{#2}$}
\newcommand{\abs}[2]{\bigl| #1 \bigr|^{#2}}
\newcommand{\norm}[2]{\bigl\Arrowvert #1 \bigr\Arrowvert_{#2}}
\newcommand{\hnorm}[2]{\bigl| #1 \bigr|_{#2}}
\newcommand{\Lnorm}[1]{L^{#1}}
\newcommand{\Hnorm}[1]{H^{#1}}

\theoremstyle{plain}
\newtheorem{proposition}{Proposition}
\newtheorem{lemma}{Lemma}
\newtheorem{theorem}{Theorem}
\theoremstyle{remark}

\numberwithin{equation}{section}
\allowdisplaybreaks
\title{Local Well-posedness of Strong Solutions to the Three-dimensional Compressible Primitive Equations}
\date{June 26, 2018}
\author{Xin Liu\footnote{Department of Mathematics, Texas A{\&}M University, College Station, Texas 77843, USA. Email: stleonliu@gmail.com} \,\, and \,  Edriss S. Titi\footnote{Department of Mathematics, Texas A{\&}M University, College Station, Texas 77843. USA. Also: Department of Computer Science and Applied Mathematics, The Weizmann Institute of Science, Rehovot 76100, Israel. Email: titi@math.tamu.edu \, and \, edriss.titi@weizmann.ac.il}}

\newcommand\blfootnote[1]{%
  \begingroup
  \renewcommand\thefootnote{}\footnote{#1}%
  \addtocounter{footnote}{-1}%
  \endgroup
}

\begin{document}
	\allowdisplaybreaks
	\maketitle	
	\blfootnote{\textup{2010} \textit{Mathematics Subject Classification}: \textup{35Q30}, \textup{35Q35}, \textup{76N10}}
	\begin{abstract}
	This work is devoted to establishing the local-in-time well-posedness of strong solutions to the three-dimensional compressible primitive equations of atmospheric dynamics. It is shown that strong solutions exist, unique, and depend continuously on the initial data, for a short time in two cases: with gravity but without vacuum, and with vacuum but without gravity. We also introduce the free boundary problem for the compressible primitive equations.
	\end{abstract}

	\tableofcontents
	
	\section{Introduction}
	\subsection{The compressible primitive equations}
	The general hydrodynamic and thermodynamic equations (see, e.g.,  \cite{Lions1996}) with Coriolis force and gravity can be used to model the motion and states of the atmosphere, which is a specific compressible fluid. However, such equations are extremely complicated and prohibitively expensive computationally. However, since the vertical scale of the atmosphere is significantly smaller than the planetary horizontal  scale, the authors in \cite{Ersoy2011a} take advantage, as it is commonly done in planetary scale geophysical models,  of the smallness of this aspect ratio  between these two orthogonal directions to formally   derive the compressible primitive equations (CPE) from the compressible Navier-Stokes equations. Specifically,  in the CPE  the vertical component of the momentum in the compressible Navier-Stokes equations is replaced by the hydrostatic balance equation \subeqref{isen-CPE-g}{3}, which is also known as the quasi-static equilibrium equation. It turns out that the hydrostatic approximation equation is  accurate enough for practical applications and   has become a fundamental equation in  atmospheric science. It is the starting point of many large scale models in the theoretical investigations and practical weather predictions (see, e.g., \cite{Lions1992}). This has also been observed by meteorologists (see, e.g., \cite{Richardson1965,Washington2005}). In fact, such an approximation is reliable and useful in the sense that the balance of gravity and pressure dominates the dynamic in the vertical direction and that the vertical velocity is usually hard to observe in reality. In many simplified models, it is assumed that the atmosphere is under adiabatic process and therefore the entropy remains unchanged along the particle path. In particular, if the entropy is constant in the spatial variables initially, it remains so in the following time. On the other hand, instead of the molecular viscosity,  eddy viscosity is used to model the statistical effect of turbulent motion in the atmosphere. The observations above and more perceptions from the meteorological point of view can be found in \cite[Chapter 4]{Richardson1965}.
	Therefore, under the above assumptions, one can write down the isentropic compressible primitive equations as in \eqref{isen-CPE-g}.  Also, we study the problem by further neglecting the gravity in \eqref{isen-CPE}.
	We remark here that, although it does not cause any additional difficulty, we have omitted the Coriolis force in this work for the convenience of presentation. The local well-posedness theorems still work for the systems with the Coriolis force.
	
	The first mathematical treatment of the compressible primitive equations (CPE) can be tracked back to Lions, Temam and Wang \cite{Lions1992}. Actually, the authors formulated the compressible primitive equations in the pressure coordinates ($p$-coordinates) and show that in the new coordinate system, the equations are in the form of classical primitive equations (called primitive equations, or PE hereafter) with the incompressibility condition. In yet another work \cite{JLLions1992}, the authors modeled the nearly incompressible ocean by the PE. It is formulated as the hydrostatic approximation of the Boussinesq equations. The authors show the existence of global weak solutions and therefore indirectly study the CPE (see, e.g., \cite{JLLions1994,Lions2000} for additional work  by the authors).
	Since then, the PE have been the subject of intensive mathematical research.  For instance, Guill{\'e}n-Gonz{\'a}lez, Masmoudi and Rodr{\'\i}guez-Bellido in \cite{guillen2001anisotropic} study the local existence of strong solutions and global existence of strong solutions for small data to the PE. In \cite{HuTemamZiane2003} the authors address the global existence of strong solutions to PE in a domain with small depth. In \cite{Petcu2005}, the authors study the Sobolev and Gevrey regularity of the solutions to PE. The first breakthrough concerning the global well-posedness of PE is obtained by Cao and Titi in \cite{Cao2007}, in which the authors show the existence of unique global strong solutions (see, also, \cite{Cao2003,Kobelkov2006,Kukavica2007a,Kukavica2007,Kukavica2014,Zelati2015,Hieber2016,Li2017a,Ignatova2012} and the references therein for related study). On the other hand, with partial anisotropic diffusion and viscosity, Cao, Li and Titi in \cite{Cao2012,Cao2014b,Cao2014,Cao2017,Cao2016,Cao2016a} establish the global well-posedness of strong solutions to PE. For the inviscid primitive equations, or hydrostatic incompressible Euler equations, in \cite{Brenier1999,Wong2012,Kukavica2011}, the authors show the existence of solutions in the analytic function space and in $H^s $ space. More recently, the authors in \cite{Wong2014,Cao2015} construct finite-time blowup for the inviscid PE in the absence of rotation. Also, in \cite{Gerard-Varet2018}, the authors establish the Gevrey regularity of hydrostatic Navier-Stokes equations with only vertical viscosity.
	
	Despite the fruitful study of the primitive equations, it still remains interesting to study the compressible equations. On the one hand, it is a more direct model to study the atmosphere and perform practical weather predictions. On the other hand, the former deviation of the PE from the CPE in the $ p $-coordinates did not treat the corresponding derivation of the boundary conditions. In fact, due to the change of pressure on the boundary, the appropriate studying domain for the PE should be evolving together with the flows in order to recover the solutions to the CPE. Thus even though the formulation of the PE significantly simplifies the equations of the CPE, the boundary conditions are more complicated than before in order to study the motion of the atmosphere. We believe that this might be one of the reasons responsible for the not-completely successful prediction of the weather by using the PE.
	
	Recently, Gatapov, Kazhikhov, Ersoy, Ngom construct a global weak solution to some variant of two-dimentional compressible primitive equations in \cite{Gatapov2005,Ersoy2012}. Meanwhile, Ersoy, Ngom, Sy, Tang, Gao study the stability of weak solutions to the CPE in \cite{Ersoy2011a,Tang2015} in the sense that a sequence of weak solutions satisfying some entropy conditions  contains a subsequence converging to another weak solution. However, the existence of such sequence of weak solutions to the CPE satisfying these entropy conditions is still open.
	
	In this and subsequent works, we aim to address several problems concerning the compressible primitive equations. In this work, we start by studying the local well-posedness of strong solutions to the CPE. That is, we will establish the local strong solutions to \eqref{isen-CPE-g} and \eqref{isen-CPE}, below, in the domain $ \Omega = \Omega_h \times (0,1) $, with $ \Omega_h = \mathbb{T}^2 = [0,1]^2 \subset \mathbb R^2 $ being the periodic domain. In comparison with the compressible Navier-Stokes equations \cite{Feireisl2004}, the absence of evolutionary equations for the vertical velocity (vertical momentum) causes the main difficulty. This is the same difficulty as in the case of primitive equations. In fact, the procedure of recovering the vertical velocity is a classical one in the modeling of the atmosphere \cite[Chapter 5]{Richardson1965}. This is done with the help of the hydrostatic equation, which causes the stratification of density profiles in the CPE. On the one hand, in \eqref{isen-CPE-g}, as one will see later, the hydrostatic equation implies that if there is vacuum area in the studied domain $ \Omega $, the sound speed will be at most $ 1/2 $-H\"older continuous. Thus the $ H^2 $ estimate of the density is not available in the presence of vacuum. However in \eqref{isen-CPE}, such an obstacle no longer exists. For this reason, the local well-posedness established in this work doesn't allow vacuum in the presence of gravity, but vacuum is allowed in the case without gravity. On the other hand, the hydrostatic equation does have some benefits. Indeed, such a relation yields that the density admits a stratified profile along the vertical direction. This fact will help us recover the vertical velocity from the continuity equation (see \eqref{vertical-isen-g} and \eqref{vertical-isen}, below).
	
	In this work, we will first reformulate the compressible primitive equations \eqref{isen-CPE-g}, \eqref{isen-CPE} by making use of the stratified density profile. Then we will study the local well-posedness of the reformulated systems under the assumption that there is no vacuum initially. This is done via a fixed point argument. Next, in order to obtain the existence of strong solutions to \eqref{isen-CPE} with non-negative density, we establish some uniform estimates independent of the lower bound of the density. We point out that in comparison to the compressible Navier-Stokes equations (see, e.g., \cite{Cho2006a,Cho2006c,Cho2004,Choe2003}), we will require $ H^2 $ estimate of $ \rho^{1/2} $ in order to derive the above mentioned uniform estimates. Such estimates are not available in the case with gravity \eqref{isen-CPE-g}. To this end, continuity arguments are used to establish the solutions with vacuum. We also study the continuous dependence on the initial data and the uniqueness of the strong solutions. Eventually, we derive in the last section  a formulation of a free boundary problem for the upper atmosphere employing  the compressible primitive equations \eqref{isen-CPE-g} (see \eqref{rfeq:isen-CPE-fb}), where the density connects to the vacuum continuously on the evolving interface. As mentioned above, the sound speed on this interface is only $ 1/2$-H\"older continuous. Such a singularity is called physical vacuum in classical literatures (see, e.g.,  \cite{Liu1996,Jang2008a,Jang2009b,Coutand2010,Coutand2011a,Coutand2012,Jang2010,Jang2011,Jang2013a,Jang2014,Jang2015,LuoXinZeng2014,LuoXinZeng2015,LuoXinZeng2016} and the reference therein). We leave such problems for future study.

We further remark here that there are
 many other attempts of  mathematical understanding of the complex dynamics of atmosphere. See, for instance,  the monograph by Majda \cite{majda2003introduction} for a comprehensive survey. In particular, the most widely used Oberbeck-Boussinesq equations of  low stratification of flows are obtained as a limiting system of the low Mach number limit of the compressible Navier-Stokes equations with the Froude number ($ \mathrm{Fr} $) and the Mach number ($ \mathrm{Ma} $) satisfying the relations $ \mathrm{Fr} \sim \sqrt{\mathrm{Ma}} $. On the other hand, with strong stratification, i.e., $ \mathrm{Fr} \sim \mathrm{Ma} $, the compressible Navier-Stokes equations formally converge to the anelastic equations as the Mach number goes to zero. See \cite[Chapter 4,5,6]{feireisl2009singular} and \cite{feireisl2011flows,feireisl2016singular,rajagopal1996oberbeck,masmoudi2007rigorous} and the references therein for introductions  and rigorous derivations. The investigation of the  low Mach number limit of the compressible primitive equations will be  a subject of future study.

Through out this work, we will use
$ \vec{x} := (x,y,z)^\top, \vech{x} := (x,y)^\top $
to represent the coordinates in $ \Omega $ and $ \Omega_h $. In addition, we will use the following notations to denote the differential operators in the horizontal direction,
\begin{gather*}
	\nablah := ( \partial_x, \partial_y)^\top,~ \partial_h \in \lbrace \partial_x, \partial_y \rbrace,\\
	\dvh  :=  \nablah \cdot,~ \deltah  := \dvh \nablah,
\end{gather*}

	As it has been mentioned, we will study the local well-posedness of the isentropic compressible primitive equations in the domain $ \Omega = \Omega_h \times (0,1) $ with $ \Omega_h = \mathbb{T}^2 $. We consider such problems in two cases: with and without gravity. To be more precise,
	the isentropic compressible primitive equations with gravity are governed by the following system:
	\begin{equation}\label{isen-CPE-g}
	\begin{cases}
	\dt \rho + \dvh (\rho v) + \dz (\rho w) = 0 & \text{in} ~ \Omega, \\
	\dt (\rho v) + \dvh (\rho v \otimes v) + \dz (\rho w v) + \nablah P = \mu \deltah v + \mu \partial_{zz} v \\
	~~~~ ~~~~ ~~~~  + (\mu +\lambda) \nablah \dvh v & \text{in} ~ \Omega,\\
	\dz P - \rho g = 0 & \text{in} ~ \Omega,
	\end{cases}
	\end{equation}
	with $ P := \rho^\gamma $. We will study in this work only the case when $ \gamma = 2 $ for \eqref{isen-CPE-g}.
	
	On the other hand, the isentropic compressible primitive equations without gravity are governed by the following system:
	\begin{equation}\label{isen-CPE}
	\begin{cases}
	\dt \rho + \dvh (\rho v) + \dz (\rho w) = 0 & \text{in} ~ \Omega, \\
	\dt (\rho v) + \dvh (\rho v \otimes v) + \dz (\rho w v) + \nablah P = \mu \deltah v + \mu \partial_{zz} v \\
	~~~~ ~~~~ ~~~~  + (\mu +\lambda) \nablah \dvh v & \text{in} ~ \Omega,\\
	\dz P = 0 & \text{in} ~ \Omega,
	\end{cases}
	\end{equation}
	with $ P := \rho^\gamma, \gamma > 1 $.

	The above systems, \eqref{isen-CPE-g} and \eqref{isen-CPE}, are supplemented with the following boundary conditions:
	\begin{equation}\label{bd-cnds}
	w = 0, ~ \dz v = 0 ~~~~ \text{on} ~~ \Omega_h \times\lbrace 0,1 \rbrace.
	\end{equation}

The rest of this work will be organized as follows.
In section \ref{sec:formulation}, we present a reformulation of \eqref{isen-CPE-g} and \eqref{isen-CPE} by making use of the stratified density profiles. Also, we present the formula for recovering the vertical velocity and the main theorems of this work. After listing some useful inequalities and  notation, we study in section \ref{sec:existencetheory} the existence theory. In particular, in section \ref{sec:lin-isen-g} and \ref{sec:lin-isen} we show that the linear equations associated with our reformulation, \eqref{rfeq:isen-CPE-g} and \eqref{rfeq:isen-CPE}, are well-posed together with some a priori estimates. Such estimates will be used in section \ref{sec:exist-thy} to show the existence of solutions with non-vacuum initial density profiles via the Schauder-Tchonoff fixed point theorem. Then in section \ref{sec:exist-vacuum}, we show the existence of strong solutions with vacuum for \eqref{isen-CPE}. Next, in section \ref{sec:cn-ddy-uni} we show the continuous dependence on the initial data and the uniqueness of strong solutions. This in turn concludes the proof of our main theorems. In section \ref{sec:fm-fb}, we show the aforementioned formulation of the free boundary problem of the upper atmosphere.
	
\subsection{Reformulation, analysis and main theorems}\label{sec:formulation}
	In this section, we will reformulate \eqref{isen-CPE-g} and \eqref{isen-CPE} and point out how to recover the vertical velocity from the density and the horizontal velocity.	
	\subsubsection*{The case with gravity and $ \gamma = 2 $:}
	We first consider \eqref{isen-CPE-g}. From \subeqref{isen-CPE-g}{3}, one has
	\begin{align*}
	\dz \rho^{\gamma-1} & = \dfrac{\gamma - 1}{\gamma} g, ~~ \text{or}\\
	\rho^{\gamma-1}(\vec{x},t) & = \dfrac{\gamma-1}{\gamma}gz + \rho^{\gamma-1}(\vec{x}_h,0,t),
	\end{align*}
	Denote $ \xi = \xi(\vec{x}_h,t) : = \rho^{\gamma-1}(\vec{x}_h,0,t) $. Now we derive the equation satisfied by $ \xi $. Notice, the continuity equation \subeqref{isen-CPE-g}{1} implies
	\begin{equation*}
	\dt \rho^{\gamma-1} + v \cdot \nablah \rho^{\gamma-1} + w \dz \rho^{\gamma-1} + (\gamma-1) \rho^{\gamma-1} (\dvh v + \dz w) = 0,
	\end{equation*}
	which, by substituting $ \rho^{\gamma-1}(\vec{x},t) = \xi(\vech{x},t) + \frac{\gamma-1}{\gamma}gz $, yields,
	\begin{equation*}
	\dt \xi + v \cdot \nablah \xi + (\gamma - 1)(\xi + \dfrac{\gamma-1}{\gamma}gz)(\dvh v + \dz w) + \dfrac{\gamma-1}{\gamma} g w = 0 ~ \text{in}~ \Omega.
	\end{equation*}
%
%
	In particular, since $ \gamma = 2 $, $ \rho(\vec{x},t) = \xi(\vech{x},t) + \dfrac{1}{2} gz $
	and
	\eqref{isen-CPE-g} can be written as
	\begin{equation}\label{rfeq:isen-CPE-g}
	\begin{cases}
	\dt \xi + v \cdot \nablah \xi + (\xi + \dfrac{1}{2}gz) \dvh v + \dz (\xi w + \dfrac{1}{2}g z w ) = 0 & \text{in} ~ \Omega,\\
	(\xi + \dfrac{1}{2}gz) ( \dt v + v \cdot\nablah v + w \dz v) + (2\xi +gz) \nablah \xi \\
	~~~~ ~~~~ ~~~~ = \mu \deltah v + \mu \partial_{zz} v + (\mu +\lambda) \nablah \dvh v & \text{in} ~ \Omega,\\
	\dz \xi = 0 & \text{in} ~ \Omega.
	\end{cases}
	\end{equation}
	Notice that the condition \subeqref{rfeq:isen-CPE-g}{3} implies that $ \xi $ is independent of the vertical variable.
	Hereafter, we denote for any $ f :\Omega \mapsto \mathbb{R} $,
	\begin{equation}\label{averaging-vert}
	\overline{f}: = \int_0^1 f \,dz , ~ \widetilde{f} := f - \overline{f}.
	\end{equation}
	Then averaging over the vertical variable in \subeqref{rfeq:isen-CPE-g}{1} yields, thanks to  \eqref{bd-cnds},
	\begin{equation}\label{conti-isen-g}
	\dt \xi + \overline{v} \cdot \nablah \xi + \xi \overline{\dvh v} + \dfrac{1}{2} g \overline{z \dvh v} = 0.
	\end{equation}
	Then comparing \eqref{conti-isen-g} with \subeqref{rfeq:isen-CPE-g}{1} implies
	\begin{equation*}
	\dz(\rho w) = \dz(\xi w + \dfrac{1}{2}gz w) = - \widetilde{v}\cdot\nablah \xi - \xi \dvh \widetilde{v} - \dfrac{g}{2} \widetilde{z \dvh v} ~~ \text{in}~ \Omega.
	\end{equation*}
	Therefore, the vertical velocity $ w $ is determined, thanks to the boundary condition \eqref{bd-cnds}, by the relation
	\begin{equation}\label{vertical-isen-g}
	\rho w = (\xi + \dfrac{1}{2}gz) w = - \int_0^z \bigl( \dvh (\xi \widetilde{v}) + \dfrac{g}{2} \widetilde{z \dvh v} \bigr) \,dz.
	\end{equation}
	System \eqref{rfeq:isen-CPE-g} is complemented with the initial data
	\begin{equation}\label{isen-initial-g}
	(\xi, v)|_{t=0} = (\xi_0, v_0),
	\end{equation}
	with $ \xi_0, v_0 \in H^2(\Omega) $.
	Also the following compatible conditions are imposed:
	\begin{equation}\label{isen-cptbl-conds-g}
	\begin{gathered}
		\rho_0 = \xi_0 + \dfrac{1}{2} gz \geq \underline{\rho} > 0 ~ \text{in} ~ \Omega, ~ \text{and} ~ \dz v_0|_{z=0,1} = 0,\\
		\mu \deltah v_0 + \mu \partial_{zz} v_0 + (\mu+\lambda) \nablah \dvh v_0 - (2\xi_0 + gz) \nablah \xi_0  \\
		- \rho_0 v_0 \cdot \nablah v_0 - \rho_0 w_0 \dz v_0 =: \rho_0 V_1, ~~~~ \text{with} ~ V_1 \in L^{2}(\Omega),\\
		 \text{and} ~~ \rho_0 w_0 = - \int_0^z \bigl( \dvh (\xi_0 \widetilde{v_0}) + \dfrac{g}{2} \widetilde{z \dvh v_0} \bigr) \,dz.
	\end{gathered}
	\end{equation}
	Also, we will denote the bounds
	\begin{equation}\label{bound-of-initial-g}
		\norm{\xi_0}{\Hnorm{2}}^2 \leq B_{g,1}, ~~ \norm{v_0}{\Hnorm{2}}^2 + \norm{V_1}{\Lnorm{2}}^2 \leq B_{g,2}.
	\end{equation}
	Our main theorem concerning the short time well-posedness of strong solutions to \eqref{isen-CPE-g} is the following:
	\begin{theorem}\label{thm:gravity}
	Suppose the initial data $ (\rho_0, v_0) = (\xi_0 + \frac{1}{2}gz, v_0) $ satisfy \eqref{bound-of-initial-g} and the compatible conditions \eqref{isen-cptbl-conds-g}. Then there is a unique strong solution $ (\rho,v) $ to system \eqref{isen-CPE-g}, with the boundary condition \eqref{bd-cnds}, in $ \Omega \times (0,T) $, for some positive constant $ T = T(B_{g,1},B_{g,2},\underline\rho) > 0 $. Also, the solution satisfies
	\begin{gather*}
		\rho \in L^\infty(0,T;H^2(\Omega)), \dt \rho \in L^\infty(0,T;H^1(\Omega)), \\
	v \in L^\infty(0,T;H^2(\Omega))\cap L^2(0,T;H^3(\Omega)), \\
	\dt v \in L^\infty(0,T;L^2(\Omega))\cap L^2(0,T;H^1(\Omega)).
	\end{gather*}
	Furthermore, for some positive constant $ \mathcal C(B_{g,1},B_{g,2},\underline\rho) $,
	\begin{gather*}
		\inf_{(\vec{x},t)\in \Omega\times (0,T)} \rho(\vec{x},t )  \geq \dfrac{1}{2} \underline{\rho} > 0,  \\
		\sup_{0\leq t\leq T} \bigl( \norm{\rho(t)}{\Hnorm{2}}^2 + \norm{\dt \rho(t)}{\Hnorm{1}}^2 + \norm{v(t)}{\Hnorm{2}}^2 + \norm{\dt v(t)}{\Lnorm{2}}^2 \bigr) \\
		 + \int_0^T \bigl( \norm{v(t)}{\Hnorm{3}}^2
		+ \norm{\dt v(t)}{\Hnorm{1}}^2 \bigr) \,dt \leq \mathcal C(B_{g,1},B_{g,2},\underline\rho).
	\end{gather*}	
	Moreover, for any two solutions $ (\rho_i, v_i), i = 1,2 $ with initial data $ (\rho_{i,0},v_{i,0}) , i = 1,2 $ satisfying the conditions mentioned above, we have the following inequality
	\begin{equation*}
	\begin{aligned}
		& \norm{\rho_1 - \rho_2}{L^\infty(0,T;L^2(\Omega))} + \norm{v_1- v_2}{L^\infty(0,T;L^2(\Omega))}\\
		& ~~~~ + \norm{\nabla (v_1 - v_2)}{L^2(0,T;L^2(\Omega))} \leq C_{\mu, \lambda, B_{g,1}, B_{g,2},\underline\rho, T} \\
		& ~~~~ ~~~~ \times \bigl(\norm{\rho_{1,0}-\rho_{2,0}}{L^2(\Omega))} + \norm{v_{1,0}-v_{2,0}}{L^2(\Omega))}\bigr),
	\end{aligned}
	\end{equation*}
	for some positive constant $ C_{\mu, \lambda, B_{g,1}, B_{g,2},\underline\rho, T} $.
	\end{theorem}

	\subsubsection*{The case without gravity and $ \gamma > 1 $:}
	Concerning system \eqref{isen-CPE}, since \subeqref{isen-CPE}{3} already yields the independence of the density of the vertical variable, the vertical velocity is determined through \subeqref{isen-CPE}{1}. 
	In fact, after taking the vertical average of \subeqref{isen-CPE}{1}, as before, one has
	\begin{equation}\label{conti-isen}
	\dt \rho + \dvh(\rho \overline v) = 0.
	\end{equation}
	Comparing \eqref{conti-isen} with \subeqref{isen-CPE}{1} yields, thanks to the boundary condition \eqref{bd-cnds}, that the vertical velocity $ w $ is determined by the relation
	\begin{equation}\label{vertical-isen}
		\rho w = - \int_0^z \dvh(\rho \widetilde v )\,dz.
	\end{equation}
	In particular, by denoting $ \sigma := \rho^{1/2} $, from \eqref{conti-isen} and \eqref{vertical-isen}, one has either $ \sigma = 0 $ or
	\begin{gather}
		\dt \sigma + \overline{v} \cdot \nablah \sigma + \frac{1}{2} \sigma \overline{\dvh v}  = 0, \label{conti-isen-02} \\
		\sigma w = - \int_0^z \bigl( \sigma {\widetilde{\dvh v}} + 2 \widetilde{v} \cdot \nablah \sigma \bigr) \,dz. \label{vertical-isen-02}
	\end{gather}
	In fact, for $ (\sigma, v) $ regular enough, \eqref{conti-isen-02}, \eqref{vertical-isen-02} hold regardless of whether $ \sigma = 0 $ or not. See also, the justification in the beginning of section \ref{sec:continuous-dependence}.
	
	System \eqref{isen-CPE} is complemented with the initial data
	\begin{equation}\label{isen-initial}
		(\rho, v)|_{t=0} = (\rho_0, v_0) ~~ \text{or equivalently} ~ (\sigma, v)|_{t=0} = (\sigma_0, v_0),
	\end{equation}
	with $ \sigma_0 = \rho_0^{1/2}, v_0 \in
		H^2(\Omega) $ and the initial total mass and physical energy satisfy
	\begin{equation}\label{bound-of-initial-energy}
	\begin{gathered}
		0 < \intw \rho_0 \idx = \intw \sigma_0^2 \idx = M < \infty, \\
		0 < \intw \rho_0 \abs{v_0}{2} \idx + \dfrac{1}{\gamma-1} \intw \rho_0^\gamma \idx = \intw \sigma_0^2 \abs{v_0}{2} \idx\\
		 + \dfrac{1}{\gamma-1} \intw \sigma_0^{2\gamma} \idx = E_0 < \infty.
	\end{gathered}
	\end{equation}
	Also the following compatible conditions are imposed:
	\begin{equation}\label{isen-cptbl-conds}
		\begin{gathered}
		\rho_0 \geq 0, ~~ \dz v_0|_{z=0,1} = 0, \\
			\mu \deltah v_0 + \mu \partial_{zz} v_0 + (\mu+\lambda) \nablah \dvh v_0 - \nablah \rho_0^\gamma - \rho_0 v_0 \cdot \nablah v_0 \\
			 - \rho_0 w_0 \dz v_0 =: \rho_0^{1/2} h_1, ~~~~ \text{with} ~ h_1 \in L^{2}(\Omega),\\
			\text{and} ~~ \rho_0 w_0 = - \int_0^z \dvh (\rho_0 \widetilde{v_0}) \,dz.
		\end{gathered}
	\end{equation}
	Also, we will denote the bounds
	\begin{equation}\label{bound-of-initial}
		\norm{\sigma_0}{\Hnorm{2}}^2 = \norm{\rho_0^{1/2}}{\Hnorm{2}}^2 \leq B_1, ~~ \norm{v_0}{\Hnorm{2}}^2 + \norm{h_1}{\Lnorm{2}}^2 \leq B_2.
	\end{equation}
	Moreover, if $ \rho = \sigma^2 > 0 $, \eqref{isen-CPE} can be written as
	\begin{equation}\label{rfeq:isen-CPE}
	\begin{cases}
	\dt \sigma + v \cdot \nablah \sigma + w \dz \sigma + \dfrac{1}{2} \sigma (\dvh v + \dz w) = 0 & \text{in} ~ \Omega,\\
	\sigma^2 (\dt v + v \cdot \nablah v + w \dz v) + \nablah \sigma^{2\gamma} \\
	~~~~ ~~~~ = \mu \deltah v + \mu \partial_{zz} v + (\mu + \lambda) \nablah \dvh v & \text{in} ~ \Omega,\\
	\dz \sigma = 0 & \text{in} ~ \Omega.
	\end{cases}
	\end{equation}
	Our main theorem concerning the short time well-posedness of strong solutions of system \eqref{isen-CPE} is stated in the following:
	\begin{theorem}\label{thm:vacuum}
	Suppose the initial data $ (\rho_0,v_0) = (\sigma_0^2, v_0) $ satisfy \eqref{bound-of-initial-energy}, \eqref{bound-of-initial} and the compatible conditions \eqref{isen-cptbl-conds}. Then there is a unique strong solution $ (\rho, v) $ to system \eqref{isen-CPE}, with the boundary condition \eqref{bd-cnds}, in $ \Omega \times (0,T^*) $ for some positive constant $ T^* = T^*(B_1,B_2) > 0 $. Also, the solution satisfies
	\begin{gather*}
		\rho^{1/2} \in L^\infty(0,T^*;H^2(\Omega)), ~~ \dt \rho^{1/2} \in L^\infty(0,T^*; H^1(\Omega)), \\
	v \in L^\infty(0,T^*;H^2(\Omega))\cap L^2(0,T^*;H^3(\Omega)), ~~ \dt v \in L^2(0,T^*;H^1(\Omega))\\
	\rho^{1/2} \dt v \in L^\infty(0,T^*;L^2(\Omega)).
	\end{gather*}
	Furthermore, for some positive constant $ \mathcal C(B_1,B_2) $,
	\begin{gather*}
		 \inf_{(\vec{x},t) \in \Omega\times (0,T^*)} \rho(\vec{x},t) \geq 0,\\
		\sup_{0\leq t\leq T^*} \bigl( \norm{\rho^{1/2}(t)}{\Hnorm{2}}^2+\norm{\dt \rho^{1/2}(t)}{\Hnorm{1}}^2 +   \norm{v(t)}{\Hnorm{2}}^2 + \norm{(\rho^{1/2} v_{t})(t)}{\Lnorm{2}}^2 \bigr) \\
		 + \int_0^{T^*} \bigl( \norm{v(t)}{\Hnorm{3}}^2 + \norm{v_{t}(t)}{\Hnorm{1}}^2 \bigr) \,dt  \leq \mathcal C(B_1,B_2).
	\end{gather*}
	Moreover, for any two strong solutions $ (\rho_i, v_i), i = 1,2 $, with initial data $ (\rho_{i,0},v_{i,0}) , i = 1,2 $, satisfying the conditions mentioned above, we have the following inequality
	\begin{equation*}
	\begin{aligned}
		& \norm{\rho_1^{1/2} - \rho_2^{1/2}}{L^\infty(0,T^*;L^2(\Omega))} + \norm{\rho_1^{1/2}(v_1- v_2)}{L^\infty(0,T^*;L^2(\Omega))} \\
		& ~~~~ + \norm{\rho_2^{1/2}(v_1- v_2)}{L^\infty(0,T^*;L^2(\Omega))} + \norm{v_1 - v_2}{L^2(0,T^*;L^2(\Omega))} \\
		& ~~~~ + \norm{\nabla (v_1 - v_2)}{L^2(0,T^*;L^2(\Omega))}\\
		& ~~ \leq C_{\mu, \lambda, B_{1}, B_{2}, T^*}\bigl( \norm{\rho_{1,0}^{1/2}-\rho_{2,0}^{1/2}}{L^2(\Omega))} + \norm{v_{1,0}-v_{2,0}}{L^2(\Omega))}\bigr),
	\end{aligned}
	\end{equation*}
	for some positive constant $ C_{\mu, \lambda, B_{1}, B_{2}, T^*} $.
	\end{theorem}
	
	 \subsection{Preliminaries}
	We will use $ \hnorm{\cdot}{}, \norm{\cdot}{} $ to denote norms in $ \Omega_h \subset \mathbb R^2 $ and $ \Omega \subset \mathbb R^3 $, respectively.
	After applying the Ladyzhenskaya's and Agmon's inequalities
	in $ \Omega_h $ and $ \Omega $, directly we have
	\begin{equation}\label{ineq-supnorm}
	\begin{gathered}
		\hnorm{f}{\Lnorm{4}} \leq C \hnorm{f}{\Lnorm{2}}^{1/2} \hnorm{f}{\Hnorm{1}}^{1/2}, ~ 
		\hnorm{f}{\Lnorm\infty}
		\leq C \hnorm{f}{\Lnorm{2}}^{1/2} \hnorm{f}{\Hnorm{2}}^{1/2}, \\
		\norm{f}{\Lnorm{3}} \leq C \norm{f}{\Lnorm{2}}^{1/2} \norm{f}{\Hnorm{1}}^{1/2},
	\end{gathered}
	\end{equation}
	for any function $ f $ with bounded right-hand sides.
	Also, $ \hnorm{\overline f}{\Lnorm{p}}, \norm{\widetilde f}{\Lnorm{p}} \leq C \norm{f}{\Lnorm{p}} $, for every $ p \geq 1 $.
	Considering any quantities $ A, B $, we use the notation $ A \lesssim B $ to denote $ A \leq CB $ for some generic positive constant $ C $ which may be different from line to line. In what follows $ \delta, \omega > 0 $ are arbitrary constants which will be chosen later in the relevant paragraphs to be adequately small. $ C_q $ represents a positive constant depending on the quantity $ q $.
	We will also need the following classical inequality.
	\begin{lemma} Let $ 2\leq p \leq 6  $, and $ \rho \geq 0 $ such that $ 0< \intw \rho \idx = M <\infty $, and $ \intw \rho^\gamma \idx \leq E_0 $, for some $ \gamma \in (1, \infty) $. Then one has
		\begin{equation}\label{ineq:embedding-weighted}
			\norm{f}{\Lnorm{p}} \leq C \norm{\nabla f}{\Lnorm{2}} + C \norm{\rho^{1/2} f}{\Lnorm{2}},
		\end{equation}
		for some constant $ C = C(M, E_0) $,
		provided the right-hand side is finite.
	\end{lemma}
	\begin{proof}
		This is standard. See, e.g., \cite[Lemma 3.2]{Feireisl2004}.
	\end{proof}
	In this work, we will apply \eqref{ineq:embedding-weighted} frequently. In fact, as we will see below, the conservations of energy and mass \eqref{conservation-energy}, \eqref{conservation-mass} satisfied by the strong solutions imply that one can apply this inequality with the constant $ C $ depending only on the initial physical energy and the total mass.
	
	\section{Associated linear systems and existence theory}\label{sec:existencetheory}
	In this section, we will establish the local existence theory of \eqref{isen-CPE-g} and \eqref{isen-CPE}. To do this, we will first study the local existence of solutions to \eqref{rfeq:isen-CPE-g} and \eqref{rfeq:isen-CPE} via the Schauder-Tchonoff fixed point theorem
	capitalizing on
	 some a priori estimates.
	In fact, under the assumption that \begin{equation}\label{lower-bound-initial} \rho_0 = \begin{cases}
	\xi_0 + \dfrac{1}{2} gz &\text{in the case with gravity}\\
	(\sigma_0)^2 & \text{in the case without gravity}
	\end{cases}  > \underline{\rho} > 0,
	\end{equation}
	we will first introduce linear systems associated with \eqref{rfeq:isen-CPE-g} and \eqref{rfeq:isen-CPE} with 
	some given input states $ (\xi^o,v^o) $ and $(\sigma^o,v^o)$, respectively. Then we will construct the maps $ (\xi^o, v^o) \leadsto (\xi,v) $ and $ (\sigma^o, v^o) \leadsto (\sigma,v) $, respectively, by means of the unique solution operators of
	 the associated linear systems,
	 and establish corresponding regularity estimates. Then we 
	 show that these maps have fixed points in some function spaces, which are the solutions to the reformulated compressible primitive equations \eqref{rfeq:isen-CPE-g} and \eqref{rfeq:isen-CPE}. This will yield the existence of strong solutions to \eqref{isen-CPE-g} and \eqref{isen-CPE} under the assumption \eqref{lower-bound-initial}. Furthermore,
	  in the case when there is no gravity, we will also be able to derive some uniform estimates independent of the lower bound $ \underline \rho $ of the initial density profile. Such a fact enables us to establish the existence of strong solutions with vacuum to system \eqref{isen-CPE}.
	
	Specifically, in section \ref{sec:lin-isen-g} and \ref{sec:lin-isen} we will introduce the associated linear systems and the function spaces $ \mathfrak{Y} $ for \eqref{rfeq:isen-CPE-g} and \eqref{rfeq:isen-CPE}, respectively, where $ \mathfrak Y $ are compactly embedded in some corresponding spaces $ \mathfrak V $. Also, we will show that the maps $ \mathcal{T}: \mathfrak X \mapsto \mathfrak X $, for some convex bounded subsets $ \mathfrak X $ of $ \mathfrak Y $, given by
	\begin{align*}
	(\xi^o, v^o) \leadsto (\xi,v) ~~~~ & \text{in the case with gravity, and}\\
	(\sigma^o, v^o) \leadsto (\sigma,v) ~~~~ & \text{in the case without gravity,}
	\end{align*}
	are well-defined; observing that $ \mathfrak X $ are convex subsets of $\mathfrak Y $ and hence compact in $ \mathfrak  V $. We will use the same notations $ \mathfrak X, \mathfrak Y, \mathfrak V, \mathcal{T} $ to denote the convex bounded sets, the compact function spaces, the embedded function spaces and the constructed maps in both cases. We summarize the relevant regularity estimates in section \ref{sec:exist-thy} and show that the Schauder-Tchonoff fixed point theorem will yield the existence of solutions to \eqref{rfeq:isen-CPE-g} and \eqref{rfeq:isen-CPE} in the corresponding set. Recall that the Schauder-Tchonoff fixed point theorem states that for a Banach space $ V $ with a convex compact subset $ X \subset V $, if $F : X \mapsto X$ is continuous, then $F$ has at least one fixed point in $ X $. In our case, we will take $ X = \mathfrak X $ and
	$ V = \mathfrak V := \lbrace (\xi,v)| \xi , v \in L^\infty(0,T;L^2(\Omega)), \nabla v\in L^2(0,T;L^2(\Omega)) \rbrace
	$ in the case with gravity, or $ V= \mathfrak V := \lbrace (\sigma,v)| \sigma , v \in L^\infty(0,T;L^2(\Omega)), \nabla v\in L^2(0,T;L^2(\Omega)) \rbrace $ in the case without gravity, with the corresponding norms.
	
Notice that, under assumption \eqref{lower-bound-initial},  systems \eqref{rfeq:isen-CPE-g} and \eqref{rfeq:isen-CPE} are equivalent to \eqref{isen-CPE-g} and \eqref{isen-CPE}, respectively. In particular, the fixed point arguments establish the existence of strong solutions to \eqref{isen-CPE-g} and \eqref{isen-CPE} with strictly positive initial density.
	On the other hand, in the case without gravity, in section \ref{sec:exist-vacuum}, we give some estimates independent of the lower bound $ \underline{\rho} $. Therefore, by taking an approximating sequence, we will eventually get the existence of strong solutions to \eqref{isen-CPE} with vacuum. The reason for this is that we have the $ H^2(\Omega) $ estimate for $ \sigma = \rho^{1/2} $, in the case without gravity, and hence $ \rho^{1/2} w $ is sufficiently regular (see \eqref{vertical-isen-02}). This will enable us to recover the spatial derivative estimates of $ v $ (see \eqref{Impt-H2-sigma}, below) . We emphasize that similar estimates are not available in the case with gravity. In fact, the hydrostatic equation \subeqref{isen-CPE-g}{3} implies that the sound speed ($ \sqrt{2}\rho^{1/2} $ in our setting) is only $ 1/2 $-H\"older continuous across the gas-vacuum interface if there is any vacuum. Thus $ \rho^{1/2} $ can not belong to the space $ H^2(\Omega) $. Therefore, in our setting, no vacuum is allowed in the case with gravity.
		
	\subsection{The case with gravity and $ \gamma = 2 $}\label{sec:lin-isen-g}
	
	\subsubsection{Associated linear inhomogeneous system}
	Consider a finite positive time $ T $, which will be determined later.
	Let $ \mathfrak{Y} = \mathfrak Y_T $ be the function space
	defined by
	\begin{equation}\label{fnc-space-isen-g-02}
	\begin{aligned}
	\mathfrak Y = \mathfrak Y_T: = & \lbrace (\xi, v)| \xi \in L^\infty(0,T;H^2(\Omega)), \dt \xi \in L^\infty(0,T;H^1(\Omega)), \\
	& v \in L^\infty(0,T;H^2(\Omega))\cap L^2(0,T;H^3(\Omega)), \\
	& \dt v \in L^\infty(0,T;L^2(\Omega))\cap L^2(0,T;H^1(\Omega)) \rbrace,
	\end{aligned}
	\end{equation}
	with the norm
	\begin{align*}
		& \norm{(\xi,v)}{\mathfrak Y} := \norm{\xi}{L^\infty(0,T;\Hnorm{2}(\Omega))} + \norm{\dt\xi}{L^\infty(0,T;\Hnorm{1}(\Omega))} + \norm{v}{L^\infty(0,T;\Hnorm{2}(\Omega))}\\
		& ~~~~ + \norm{v}{L^2(0,T;\Hnorm{3}(\Omega))} + \norm{\dt v}{L^\infty(0,T;\Lnorm{2}(\Omega))} + \norm{\dt v}{L^2(0,T;\Hnorm{1}(\Omega))}.
	\end{align*}
	Notice that for this space, one can make sense of the initial value $ \xi_0 $ and $ v_0 $ for $ \xi $ and $ v $, respectively.
	Hereafter, the notation $ \mathfrak S_T $ for the function space $ \mathfrak S \in \lbrace \mathfrak Y, \mathfrak V, \mathfrak X \rbrace $ is used to emphasized the time dependence of the function space on $ T > 0 $, while the notation $ \mathfrak S $ and $ \mathfrak S_T $ will be used alternatively.
	Notice that, thanks to Aubin compactness theorem (see, e.g., \cite[Theorem 2.1]{temam1977} and \cite{Chen2012,Simon1986}),  every bounded subset of $ \mathfrak Y $ is a compact subset of the space
	\begin{equation}\label{embedd-space-g}
		\mathfrak V = \mathfrak V_T := \lbrace (\xi,v)| \xi , v \in L^\infty(0,T;L^2(\Omega)), \nabla v\in L^2(0,T;L^2(\Omega)) \rbrace,
	\end{equation}
	with the norm
	\begin{equation}\label{norm-bspace-g}
		\begin{aligned}
		\norm{(\xi,v)}{\mathfrak V} := & \norm{\xi}{L^\infty(0,T;\Lnorm{2}(\Omega))} + \norm{v}{L^\infty(0,T;\Lnorm{2}(\Omega))} \\
		& ~~~~ + \norm{v}{L^2(0,T;\Hnorm{1}(\Omega))}.
		\end{aligned}
	\end{equation}
	Let $ \mathfrak X = \mathfrak X_{T} $ be a bounded subset of $ \mathfrak Y $ defined by
	\begin{equation}\label{fnc-space-isen-g}
	\begin{aligned}
		\mathfrak X = & \mathfrak X_T:= \bigl\lbrace (\xi, v)\in \mathfrak Y| (\xi, v)|_{t=0} = (\xi_0,v_0), \dz v|_{z=0,1} = 0, \dz \xi = 0,  \\
		& \xi + \dfrac{1}{2}gz \geq \dfrac{1}{2} \underline{\rho} > 0,\sup_{0\leq t\leq T} \norm{\xi(t)}{\Hnorm{2}}^2 \leq 2 M_0,  \sup_{0\leq t\leq T} \norm{\dt \xi(t)}{\Hnorm{1}}^2 \leq C_2, \\
		& \sup_{0\leq t\leq T} \lbrace \norm{v(t)}{\Hnorm{2}}^2 + \norm{\dt v(t)}{\Lnorm{2}}^2 \rbrace + \int_0^T \biggl( \norm{v(t)}{\Hnorm{3}}^2 \\
		& + \norm{\dt v(t)}{\Hnorm{1}}^2 \biggr) \,dt \leq C_1 M_1   \bigr\rbrace,
	\end{aligned}
	\end{equation}
	where $ M_0, M_1 $ are the bounds of initial data in \eqref{bound-of-initial-data-linear-g} and $C_1 = C_1(M_0,\mu,\lambda,\underline\rho)$, $ C_2 = C_2(M_0,C_1M_1) $ are given below in \eqref{def:C_1}, \eqref{def:C2}, respectively.
	Notice, for $ (\xi, v) \in \mathfrak X $,
	$$\int_0^1\bigl( \dvh (\xi \widetilde{v}) + \dfrac{g}{2} \widetilde{z \dvh v} \bigr) \,dz = 0.$$
	
	Let $ (\xi^o,v^o) \in \mathfrak X $. The following inhomogeneous linear system is inferred from \eqref{rfeq:isen-CPE-g} using $(\xi^o,v^o) $ as an input:	
	\begin{equation}\label{eq:lin-isen-g}
	\begin{cases}
		\dt \xi + \overline{v^o} \cdot \nablah \xi + \xi \overline{\dvh v^o} + \dfrac{g}{2} \overline{z \dvh v^o} = 0 & \text{in} ~ \Omega,\\
		(\xi^o + \dfrac{1}{2}gz) ( \dt v + v^o \cdot\nablah v^o + w^o \dz v^o) + (2\xi^o +gz) \nablah \xi^o \\
		~~~~ ~~~~ ~~~~ = \mu \deltah v + \mu \partial_{zz} v + (\mu +\lambda) \nablah \dvh v & \text{in} ~ \Omega,\\
		\dz \xi = 0 & \text{in} ~ \Omega,
	\end{cases}
	\end{equation}
	where $ w^o $ is given by 
	\eqref{vertical-isen-g} with $ (\xi^o, v^o) $ instead of $ (\xi, v) $, i.e.,
	\begin{equation}\label{vertical-lin-isen-g}
		\rho^o w^o = (\xi^o + \dfrac{1}{2}gz) w^o := - \int_0^z \bigl( \dvh (\xi^o \widetilde{v^o}) + \dfrac{g}{2} \widetilde{z \dvh v^o} \bigr) \,dz.
	\end{equation}
	Notice that \subeqref{eq:lin-isen-g}{1} is inferred from \eqref{conti-isen-g}. For details, see the deviation from \eqref{rfeq:isen-CPE-g} to \eqref{vertical-isen-g}.
	Hereafter, denote by
	$ \rho^o := \xi^o + \dfrac{1}{2}gz $
	. The initial and boundary conditions for the linear system \eqref{eq:lin-isen-g} are given by
	\begin{equation}\label{bd-conds-lin-g}
	(\xi,v)|_{t=0} = (\xi_0,v_0),~ \dz v|_{z=0,1} = 0.
	\end{equation}
	The compatible conditions in \eqref{isen-cptbl-conds-g} are still imposed and we require \begin{equation}\label{bound-of-initial-data-linear-g} \norm{\xi_0}{\Hnorm{2}}^2 \leq M_0, \norm{v_0}{\Hnorm{2}}^2 + \norm{V_1}{\Lnorm{2}}^2 \leq M_1. \end{equation}
	Recall that $ V_1 $ is given in \eqref{isen-cptbl-conds-g}, essentially, $ V_1 = v_t|_{t=0} $.

	Then the map $ \mathcal{T} $, in this case, is defined as
	\begin{equation}\label{def:map-isen-g}
	\mathcal T: (\xi^o,v^o) \leadsto (\xi, v),
	\end{equation}
	where $ (\xi, v) $ is the unique solution to the linear system \eqref{eq:lin-isen-g} with $ (\xi^o, v^o) \in \mathfrak X $, to be established below. The rest of this subsection is devoted to show that $ \mathcal{T} $ is a well defined map from $ \mathfrak{X} $ to $ \mathfrak{X} $.
	
	\subsubsection{Analysing the linear system}\label{sec:solving-linear-g}
	In this section we will show the existence of solutions to the linear system \eqref{eq:lin-isen-g} with 
	$ (\xi^o, v^o) $ in $ C^{\infty}(\overline\Omega \times [0,T])\cap \mathfrak X_T $ 
	by the Galerkin method.
	
	 Equations \subeqref{eq:lin-isen-g}{1} and \subeqref{eq:lin-isen-g}{2} are linear hyperbolic and parabolic equations, respectively. The standard existence theory for the corresponding problem will yield the existence of solutions to our linear system \eqref{eq:lin-isen-g}. We will only sketch the proof. In this subsection, the positive constant $ C $, which may be different from line to line, depends on $ \underline\rho $, $ T ,\norm{\xi _0}{\Hnorm{2}} , \norm{v_0}{\Hnorm{2}} $, and the $ \Lnorm{\infty} $-norm of $ \xi^o, v^o $ and their derivatives which are finite.
	
	Equation \subeqref{eq:lin-isen-g}{2} can be solved by standard Galerkin approximation. For the sake of simplicity, equation \subeqref{eq:lin-isen-g}{2} is written as
	\begin{equation}{\label{eq:linear-parabolic-g}}
		(\xi^o + \dfrac{1}{2}gz) \dt v = \mu \deltah v + \mu \partial_{zz} v + (\mu + \lambda)\nablah \dvh v + \mathcal F(\xi^o, v^o),
	\end{equation}
	with temporary assumption that $ \mathcal F(\xi^o, v^o) \in C^\infty(\overline{\Omega}\times[0,T]) $.
	Consider the orthonormal basis $ \lbrace \vec{e}_i \in H^{20}(\Omega) \rbrace_{i = 1, 2 \cdots} $ of the Hilbert space $ H^1(\Omega) $.
	Denote by 
	$ v_m(t):= \sum_{i=1}^{m} \beta_i(t) \vec{e}_i (\vec{x}) $, which is the solution to the following ODE system:
	for $ i=1,2\cdots m $,
	\begin{equation}\label{ODE-g-2}
		\begin{aligned}
			& \sum_{j=1}^m \beta'_j(t) \intw (\xi^o + \dfrac{1}{2} gz) \vec{e}_j \cdot \vec{e}_i \idx + \sum_{j=1}^{m}\intw \Big\lbrack \beta_j(t) (\mu \nablah \vec e_j : \nablah \vec e_i \\
			& ~~~~ + \mu \partial_{z} \vec{e}_j \cdot \partial_z \vec{e}_i + (\mu+\lambda) \bigl(\dvh \vec e_j \bigr) \bigl( \dvh \vec e_i \bigr) \Big\rbrack \idx = \intw \mathcal F(\xi^o, v^o) \cdot e_i\idx,
		\end{aligned}
	\end{equation}
%
	where
	the above system is supplemented by
	the initial data $\lbrace \beta_i(0) \rbrace_{i=1,2\cdots m} $ satisfying
	\begin{equation*}
		\beta_i(0) = \intw \vec{e}_i \cdot v_0 \idx ~~~~ i \in \lbrace 1,2\cdots m \rbrace.
	\end{equation*}
	Notice that $ \rho^o = \xi^o + \dfrac{1}{2} gz \geq \dfrac{1}{2} \underline\rho > 0 $. Therefore the matrix $ (a_{ij})_{m \times m} $ defined by
	\begin{equation*}
		a_{ij} = \intw (\xi^o + \dfrac{1}{2} gz) \vec{e}_j \cdot \vec{e}_i \idx
	\end{equation*}
	is invertible.
	
	This can be justified by contradiction as follows. Suppose $ (a_{ij})_{m\times m} $ is not invertible. Consider the vectors $ b_i := (a_{i1}, a_{i2} \cdots a_{i m}) , i = 1,2\cdots m $. Then there exists $ (\lambda_i)_{ i=1,2\cdots m} \neq 0 $ such that $ \sum_{i=1}^m \lambda_i b_i = 0 $. This is equivalent to say
	\begin{equation*}
		0 = \intw (\xi^o + \dfrac{1}{2} gz) \vec{e}_j \cdot \sum_{i=1}^m \lambda_i \vec{e}_i \idx, ~~ \forall ~ j = 1,2\cdots m.
	\end{equation*}
	Multiply the above equation with $ \lambda_j $ and sum over $ j = 1,2\cdots m $. We have
	\begin{equation*}
		\intw (\xi^o + \dfrac{1}{2} gz) \abs{ \sum_{i=1}^m \lambda_i \vec{e}_i}{2} \idx = 0,
	\end{equation*}
	which yields  $ \sum_{i=1}^m \lambda_i \vec{e}_i = 0 $, since $ \xi^o + \dfrac{1}{2} gz \geq \dfrac{1}{2} \underline\rho > 0 $. But the linear independence of $ \lbrace \vec{e}_i \rbrace_{i = 1,2\cdots m} $ in $ H^1(\Omega) $ implies that $ (\lambda_i)_{i=1,2\cdots m} = 0 $, which is contradictory to the previous assumption.
		
	Then the standard ODE theory yields that there exist unique solutions $ \lbrace \beta_i \rbrace_{i=1,2\cdots m} $ for $ t \in [0,T_m') $, with maximal existence time $  T_{m}' \in (0,T\rbrack $. Now we show that the maximal existence time $ T_m' = T $. Indeed, let us assume $ T_m' < T $. Multiply \eqref{ODE-g-2} with $ \beta_i $ and sum up the results over $ i \in \lbrace 1 , 2 \cdots m \rbrace $. Then in terms of $ v_m $, one will have
	\begin{align*}
		& \dfrac{1}{2} \dfrac{d}{dt} \intw (\xi^o + \dfrac{1}{2} gz) \abs{v_m}{2}\idx + \intw \biggl( \mu \abs{\nablah v_m}{2} + \mu \abs{\dz v_m}{2} \\
		& ~~~~ ~~~~  + (\mu+\lambda) \abs{\dvh v_m}{2} \biggr) \idx  = \dfrac{1}{2} \intw \dt \xi^o \abs{v_m}{2} \idx \\
		& ~~~~  + \intw  \mathcal F(\xi^o, v^o) \cdot  v_m \idx \leq C \intw (\xi^o + \dfrac{1}{2} gz) \abs{v_m}{2}\idx + C,
	\end{align*}
	for some positive constant $ C $ depending on $ \underline \rho $ and \begin{equation*}
		 \sup_{0\leq t \leq T}\bigl\lbrace \norm{\xi^o(t)}{\Lnorm{\infty}} , \norm{\dt \xi^o(t)}{\Lnorm{\infty}}, \norm{\mathcal F(\xi^o, v^o)(t)}{\Lnorm{\infty}}  \bigr\rbrace.
	\end{equation*}
	Then the Gr\"onwall's inequality yields
	\begin{align*}
		& \sup_{0\leq t\leq T_m'} \norm{v_m(t)}{\Lnorm{2}}^2 + \int_0^{T_m'} \norm{\nabla v_m(t)}{\Lnorm{2}}^2 \,dt \leq C,
	\end{align*}
	for some positive constant $ C $ independent of $ m $.
	This implies that one can extend the solution such that $ T_m' = T $. On the other hand, multiply \eqref{ODE-g-2} with $ \beta'_i(t) $ and sum up the results over $ i\in \lbrace 1,2\cdots m\rbrace $. Then one has
	\begin{align*}
		& \norm{\sqrt{\rho^o} \dt v_m}{\Lnorm{2}}^2
		 + \dfrac{1}{2} \dfrac{d}{dt} \biggl( \mu \norm{\nablah v_m}{\Lnorm{2}}^2 + \mu \norm{\dz v_m}{\Lnorm{2}}^2 \\
		& ~~~~ + (\mu+\lambda) \norm{\dvh v_m}{\Lnorm{2}}^2 \biggr) = \int_\Omega \mathcal F(\xi^o, v^o) \cdot \dt v_m \idx \\
		& ~~~~ \leq \dfrac{1}{2} \int \rho^o \abs{\dt v_m}{2} \idx + C,
 	\end{align*}
 	for some positive constant $ C $, observing that $ \rho^o \geq \dfrac{1}{2} \underline \rho > 0 $. Then
	integrating the above with respect to $ t \in [0,T] $ yields
	\begin{equation*}
		 \sup_{0\leq t\leq T}\norm{\nabla v_m(t)}{\Lnorm{2}}^2 + \int_0^{T} \norm{\dt v_m(t)}{\Lnorm{2}}^2\,dt \leq C,
	\end{equation*}
	for some positive constant $ C $.
	Therefore, we have the following $m$-independent bounds for $ v_m $:
	\begin{equation*}
		\norm{v_m}{L^\infty(0,T;H^1(\Omega))} + \norm{\dt v_m}{L^2(0,T;L^2(\Omega))} < C.
	\end{equation*}
	Then after taking a subsequence if necessary, there is $ v \in L^\infty(0,T;H^1(\Omega))
	$, with $ \dt v\in L^2(0,T;L^2(\Omega)) $ such that
	\begin{gather*}
		v_m \buildrel\ast\over\rightharpoonup v ~~ \text{in} ~ L^\infty(0,T;H^1(\Omega)), \\
		\dt v_m \rightharpoonup \dt v ~~ \text{in} ~ L^2(0,T;L^2(\Omega)).
	\end{gather*}
	By Aubin compactness theorem (\cite[Theorem 2.1]{temam1977} and \cite{Chen2012,Simon1986}) we have also
	$$ v_m \rightarrow v ~~ \text{in} ~ L^2(0,T;L^2(\Omega)) ~ \text{strongly}. $$
	With the above convergences, one can derive, from \eqref{ODE-g-2}, that the following equality holds for $ a.e., s_1, s_2 \in [0,T] $ and $ s_1 < s_2 $:
	\begin{equation}
		\begin{aligned}\label{weak-sol-g-2}
			& \intw (\rho^o v \cdot \vec{\psi})|_{t=s_2} \idx -  \intw (\rho^o v \cdot \vec{\psi})|_{t=s_1} \idx + \int_{s_1}^{s_2} \intw \biggl( - v \cdot \dt (\rho^o \vec{\psi})  \\
			& ~~~~ + \mu \nablah v : \nablah \vec{\psi}
			 + \mu \dz v \cdot \dz \vec{\psi} + (\mu+\lambda) \dvh v \dvh \vec{\psi} \biggr) \idx\,dt \\
			 & =  \int_{s_1}^{s_2} \intw \mathcal F(\xi^o, v^o) \cdot \vec{\psi} \,dt,
		\end{aligned}
	\end{equation}
	for $ \forall \vec{\psi} \in C^\infty(\overline{\Omega} \times [0,T]) $.
	In addition, since $ \dt v \in L^2(0,T;L^2(\Omega)) $ and $ v \in L^\infty(0,T;H^1(\Omega)) $, $ v $ is continuous as a $ L^2 $ function; moreover, \eqref{weak-sol-g-2} holds for any $ s_1, s_2 \in [0,T] $. In particular, it holds for $ s_1 = 0, s_2 = T $. Also, $ v $ is the weak solution to \eqref{eq:linear-parabolic-g} (equivalently \subeqref{eq:lin-isen-g}{2}). 
We remark that, argued by taking a convergence sequence of $ C^1(\overline\Omega\times[0,T]) $ functions in $ C^\infty(\overline\Omega\times[0,T]) $,
$ \vec{\psi} $ in \eqref{weak-sol-g-2} needs not to be $ C^\infty $ as above. In particular, $ \vec{\psi} \in C^1(\overline\Omega\times[0,T]) $ is sufficient to apply the arguments in the following paragraphs.

Now we employ standard difference quotient method in the tangential direction by choosing $ \vec{\psi} = D^{-h}(D^h v_m )) $ in \eqref{weak-sol-g-2} for $ h \neq 0 $ where $ v_m $ is as before and $ D^h $ is the different quotient in the tangential direction:
$$ D^h (f)(\cdot) := \dfrac{f(\cdot + h \vech{e}) - f(\cdot) }{h}, ~ h \neq 0, $$
with $ \vech{e} \in \lbrace \vec{e}_x, \vec{e}_y \rbrace $ being the horizontal unit vector. Then it implies,
\begin{align*}
	& \intw ( D^h( \rho^o v ) \cdot D^h v_m)|_{t=T} \idx -  \intw (D^h(\rho^o v_0) \cdot D^h v_m)|_{t=0} \idx \\
	& ~~~~ + \int_0^T \intw \biggl( D^{-h} (D^h v) \cdot \dt (\rho^o v_m) + D^{-h} v \cdot \dt ((D^h\rho^o) v_m) \\
	& ~~~~ + D^{h} v \cdot \dt ((D^{-h}\rho^o) v_m) + v \cdot \dt (D^{-h}(D^h\rho^o) v_m)   \\
	& ~~~~ + \mu \nablah D^h v : \nablah D^h v_m
		 + \mu \dz D^h v \cdot \dz D^h v_m \\
	& ~~~~ + (\mu+\lambda) \dvh D^h v \dvh D^h v_m \biggr) \idx\,dt\\
	& ~~~~ =  - \int_0^T \intw D^{-h} (D^h \mathcal F(\xi^o, v^o)) \cdot v_m \,dt.
\end{align*}
After expanding the above expression, taking $ m \rightarrow \infty $ and then $ h \rightarrow 0 $, we have
\begin{equation*}
	\nabla \nablah v \in L^2(0,T; L^2(\Omega)).
\end{equation*}
Moreover, \eqref{weak-sol-g-2} can be written as, after applying integration by parts in time, thanks to the regularity we have obtained so far, i.e., $ \dt v \in L^2(0,T;L^2(\Omega)), \nabla \nablah v \in L^2(0,T;L^2(\Omega)), v \in L^\infty(0,T;H^1(\Omega)) $,
\begin{equation*}{\tag{\ref{weak-sol-g-2}'}}
		\begin{aligned}
			& \int_0^T \intw \mu \dz v \cdot \dz \vec{\psi} \idx\,dt = - \int_0^T \intw \biggl( \rho^o \dt v
			-  \mu \deltah v  \\
			& ~~~~ - (\mu+\lambda) \nablah \dvh v - \mathcal F(\xi^o, v^o) \biggr) \cdot \vec{\psi} \idx\,dt,
		\end{aligned}
\end{equation*}
where the right-hand side is finite for any $ \vec{\psi} \in C^\infty(\overline\Omega\times[0,T]) $ with $ \norm{\vec{\psi}}{\Lnorm{2}(0,T;L^2(\Omega))} < \infty $. In particular, by taking $ \vec{\psi} \in C_0^{\infty}(\overline\Omega\times[0,T]) $, this identity implies that the distributional derivative $ \partial_{zz} v = \partial_z (\partial_z v) $ is a function in $ L^2(0,T;L^2(\Omega)) $, thus $ v \in L^2(0,T;H^2(\Omega)) $. Therefore, \eqref{eq:linear-parabolic-g} holds in $ L^2(0,T;L^2(\Omega)) $ and a.e. in $ \Omega \times (0,T) $. Consequently,
\begin{equation}\label{21Jun001}
\begin{aligned}
	& \int_0^T \intw \biggl( \rho^o \dt v
			-  \mu \deltah v  - (\mu+\lambda) \nablah \dvh v - \mu \partial_{zz} v\\
	& ~~~~ - \mathcal F(\xi^o, v^o) \biggr) \cdot \vec{\psi} \idx\,dt = 0,
\end{aligned}
\end{equation}
for any $ \vec{\psi} \in C^\infty(\overline{\Omega}\times [0,T]) $. Now, thanks to the regularity we have obtained, from \eqref{weak-sol-g-2}, after applying integration by parts, for every $ \vec{\psi} \in C^\infty(\overline{\Omega}\times[0,T])  $, it holds
	\begin{align*}
		& \int_0^T \inth \mu \dz v \cdot \vec{\psi} \idxh\,dt|_{z=1} - \int_0^T \inth \mu \dz v \cdot \vec{\psi} \idxh\,dt|_{z=0} \\
		& ~~~~ = -  \int_0^T \intw \biggl( \rho^o v_t - \mu \deltah v - \mu \partial_{zz} v - (\mu+\lambda) \nablah \dvh v \\
		& ~~~~ - \mathcal F(\xi^o, v^o) \biggr) \cdot \vec{\psi} \idx \,dt= 0,
	\end{align*}
due to \eqref{21Jun001}.
Therefore, by taking $ \vec{\psi} $ with support near the boundaries $ \lbrace z=0 \rbrace $ and $ \lbrace z = 1 \rbrace $, one can check $ \dz v \bigr|_{z=0,1} = 0 $ in the distribution sense and also in $ L^2(\Omega_h) $ thanks to the trace theorem. This verifies the boundary condition \eqref{bd-conds-lin-g} and $ v $ is the solution to \subeqref{eq:lin-isen-g}{2} (equivalently \eqref{eq:linear-parabolic-g}) with $ \partial_z v \bigr|_{z=0,1} = 0 $ and
	\begin{equation}\label{09Jun}
	\begin{gathered}
		v \in L^\infty(0,T;H^1(\Omega)) \cap L^2(0,T;H^2(\Omega)), \\
		\dt v \in L^2(0,T;L^2(\Omega)).
	\end{gathered}
	\end{equation}	
	Observing that the ODE theory we apply before in \eqref{ODE-g-2} implies that $ \lbrace \beta_j' \rbrace_{j=1,2 \cdots m} $  are differentiable in time, and hence $ \lbrace \beta_j'' \rbrace_{j=1,2 \cdots m}  $ exist.
	To derive higher order regularity, consider the time derivative of the Galerkin approximation system \eqref{ODE-g-2}: for $ i \in \lbrace 1 ,2 \cdots m \rbrace $,
	\begin{align*}
			& \sum_{j=1}^m \beta''_j(t) \intw (\xi^o + \dfrac{1}{2} gz) \vec{e}_j \cdot \vec{e}_i \idx + \sum_{j=1}^m \beta'_j(t) \intw \dt\xi^o \vec{e}_j \cdot \vec{e}_i \idx \\
			& ~~~~  + \sum_{j=1}^{m}\intw \Big\lbrack \beta_j'(t) (\mu \nablah \vec e_j : \nablah \vec e_i + \mu \partial_{z} \vec{e}_j \cdot \partial_z \vec{e}_i \\
			& ~~~~ + (\mu+\lambda) \bigl(\dvh \vec e_j \bigr) \bigl( \dvh \vec e_i \bigr) \Big\rbrack \idx = \intw \dt (\mathcal F(\xi^o, v^o)) \cdot e_i\idx.
	\end{align*}

	Then multiply the above with $ \beta_i'(t) $ and sum the results over $ i \in \lbrace 1,2 \cdots m \rbrace $. Integration in time yields $ \dt v_m \in L^\infty(0,T;L^2(\Omega)) \cap L^2(0,T;H^1(\Omega)) $, and the norm is bounded independent of $ m $. By taking $ m \rightarrow\infty $, one concludes that $$ \dt v \in L^\infty(0,T;L^2(\Omega)) \cap L^2(0,T;H^1(\Omega)) .$$ Now, since $ v $ is the solution to \eqref{eq:linear-parabolic-g} (equivalently \subeqref{eq:lin-isen-g}{2}), consider the elliptic problem,
	\begin{equation}\label{09Jun02}
	\begin{gathered}
		- \mu \deltah v - \mu \partial_{zz} v - (\mu + \lambda)\nablah \dvh v = - (\xi^o + \dfrac{1}{2}gz) \dt v  + \mathcal F(\xi^o, v^o)\\
		\text{in} ~ \Omega, ~~~~ \text{with} ~~ \dz v\bigr|_{z=0,1}, ~ \text{periodic in} ~ \vech{x}.
	\end{gathered}
	\end{equation}
	Then the regularity estimate of such an elliptic problem implies that $$ v \in L^\infty(0,T;H^2(\Omega)) \cap L^2(0,T;H^3(\Omega)) . $$
	We remark here about how to get the $ H^3 $ regularity:  first using difference quotient method to obtain the estimates of the horizontal derivatives; and then using the equation to represent the vertical derivative in terms of others and thus obtaining the regularity of the vertical derivative. Similar arguments also applied to the $ H^4 $ regularity below. See, for instance, \cite[Proposition 7.5]{taylorPDE01}

	
	Additionally, supposed $ v_0 \in H^3(\Omega) $, from the equation \eqref{eq:linear-parabolic-g}, one has $ \dt v\bigr|_{t=0} \in H^1(\Omega) $. Repeating the above arguments from \eqref{ODE-g-2} to \eqref{09Jun}
	 with $ v $ replaced by $ \dt v $ yields $ \dt v \in L^\infty(0,T;H^1(\Omega)) \cap L^2(0,T;H^2(\Omega)) $ and the elliptic estimate of elliptic problem \eqref{09Jun02} implies $ v \in L^\infty(0,T;H^3(\Omega)) \cap L^2(0,T;H^4(\Omega)) $.
	
	In order to solve the hyperbolic equation \subeqref{eq:lin-isen-g}{1}, we employ a parabolic  regularizing procedure. Consider the regularization of \subeqref{eq:lin-isen-g}{1}: for any $ \iota > 0 $,
	\begin{equation}\label{eq:regularized-hyper-g}
		\dt \xi + \overline{v^o} \cdot \nablah \xi + \xi \overline{\dvh v^o} + \dfrac{g}{2} \overline{z \dvh v^o} = \iota \deltah \xi  ~~~ \text{in} ~ \Omega_h,
	\end{equation}
	subject to periodic boundary condition.
The existence and the regularity of  solutions to \eqref{eq:regularized-hyper-g} follow with similar arguments as above. In fact, one can repeat similar steps  to show that for $ \xi_0 \in H^3(\Omega_h) $, the solution to \eqref{eq:regularized-hyper-g} satisfies
\begin{gather*}
	\xi \in L^\infty(0,T;H^3(\Omega)) \cap L^2(0,T;H^4(\Omega)),\\
	\dt \xi \in L^\infty(0,T; H^1(\Omega)) \cap L^2(0,T;H^2(\Omega)).
\end{gather*}

Now we derive the estimate independent of $ \iota > 0 $. In fact, taking $ \partial_{hh} = \partial_h^2 $, $ \partial_h \in \lbrace \partial_x, \partial_y \rbrace $, in \eqref{eq:regularized-hyper-g}, and taking $ L^2 $-inner product of the resultant with $ \partial_{hh} \xi $ yield, after applying integration by parts and H\"older inequality,
\begin{align*}
	\dfrac{1}{2} \dfrac{d}{dt}\hnorm{\partial_{hh} \xi}{\Lnorm{2}}^2 + \iota \hnorm{\nablah \partial_{hh} \xi }{\Lnorm{2}}^2 \leq C + C \hnorm{\partial_{hh}\xi }{\Lnorm{2}}^2.
\end{align*}
Then the Gr\"onwall's inequality yields $ \hnorm{\xi}{\Hnorm{2}} < C\hnorm{\xi_0}{\Hnorm{2}} + C $ for some positive constant $ C $, independent of $ \iota $ and $ \hnorm{\xi_0}{\Hnorm{3}} $. Also, taking $ L^2 $-inner product of \eqref{eq:regularized-hyper-g} with $ \dt \xi $ and applying integration by parts in the resultants yield,
\begin{equation*}
	 \dfrac{\iota}{2} \dfrac{d}{dt} \hnorm{\nablah \xi}{\Lnorm{2}}^2 + \hnorm{\dt \xi}{\Lnorm{2}}^2 \leq \dfrac{1}{2} \hnorm{\dt \xi}{\Lnorm{2}}^2 + C \hnorm{\xi}{\Hnorm{2}} + C.
\end{equation*}
Thus $ \hnorm{\dt \xi}{L^2(0,T;\Lnorm{2}(\Omega))} < C \hnorm{\xi_0}{\Hnorm{2}} + C  $ for some positive constant $ C $, independent of $ \iota $ and $ \hnorm{\xi_0}{\Hnorm{3}} $.
 Therefore by taking $ \iota \rightarrow 0^+ $ in \eqref{eq:regularized-hyper-g}, we obtain a solution $ \xi $ to the hyperbolic equation \subeqref{eq:lin-isen-g}{1} with $$ \xi \in L^\infty(0,T;H^2(\Omega)), ~ \dt \xi \in L^2(0,T;L^2(\Omega)). $$
From \subeqref{eq:lin-isen-g}{1}, one can infer that $ \dt \xi \in L^\infty(0,T;H^1(\Omega)) $. In particular, the bounds depend only on $ \norm{\xi_0}{\Hnorm{2}}$. Then taking an approximating sequence of $ \xi_0 \in H^2 $ in $ H^3 $ space concludes the existence of solutions to the hyperbolic equation \subeqref{eq:lin-isen-g}{1} for $ \xi_0 \in H^2 $.

In addition, suppose $ \xi_0 \in H^3(\Omega) $. By approximating it in $ H^4 $ space, similar arguments will establish a solution $ \xi $ to the hyperbolic equation \eqref{eq:lin-isen-g} with
$$ \xi \in L^\infty(0,T;H^3(\Omega)), ~ \dt \xi \in L^\infty(0,T;H^2(\Omega)). $$

	We summarize the above discussion in the following:
	\begin{proposition}\label{prop:exist-lin-isen-g}
		For given $ (\xi^o,v^o) \in C^\infty(\overline{\Omega}\times[0,T]) \cap \mathfrak X_T $, there is a unique strong solution $ (\xi, v) \in \mathfrak Y_T $ of system \eqref{eq:lin-isen-g} with the initial and boundary conditions \eqref{bd-conds-lin-g}.
		
		Supposed, in addition, $ \xi_0, v_0 \in H^3(\Omega) $, 
		one will have the following regularity of the unique solution $ (\xi, v) $ of system \eqref{eq:lin-isen-g}:
		\begin{equation}\label{rg-aprox}
		\begin{gathered}
			\xi \in L^\infty(0,T;H^3(\Omega)), ~ \dt \xi \in L^\infty(0,T; H^2(\Omega)), \\
			v \in L^\infty(0,T;H^3(\Omega))\cap L^2 (0,T;H^4(\Omega)), \\
			\dt v \in L^\infty(0,T;H^1(\Omega)) \cap L^2(0,T;H^2(\Omega)).
		\end{gathered}
		\end{equation}
	\end{proposition}
	\begin{proof}
		With the discussion above, what is left is to prove the uniqueness. It is sufficient to consider the following homogeneous system with vanishing initial data:
	\begin{equation*}
	\begin{cases}
		\dt \xi + \overline{v^o} \cdot \nablah \xi + \xi \overline{\dvh v^o} = 0 & \text{in} ~ \Omega,\\
		(\xi^o + \dfrac{1}{2}gz) \dt v = \mu \deltah v + \mu \partial_{zz} v + (\mu +\lambda) \nablah \dvh v & \text{in} ~ \Omega,\\
		\dz \xi = 0 & \text{in} ~ \Omega,\\
		\dz v\bigr|_{z=0,1} = 0, ~~ (\xi, v)_{t=0} = 0.
	\end{cases}
	\end{equation*}
	However, the standard $ L^2 $ estimate yields
	\begin{align*}
		& \dfrac{d}{dt}\norm{\xi}{\Lnorm{2}}^2 \leq C \norm{v^o}{\Hnorm{3}} \norm{\xi }{\Lnorm{2}}^2, \\
		&  \dfrac{d}{dt}\norm{(\xi^o + \dfrac{1}{2} gz)^{1/2} v}{\Lnorm{2}}^2 + 2\mu \norm{\nabla v}{\Lnorm{2}}^2 + 2(\mu+\lambda) \norm{\dvh v}{\Lnorm{2}}^2\\
		& ~~~~ 
		\leq \norm{\dt \xi^o}{\Lnorm{2}} \norm{v}{\Lnorm{3}} \norm{v}{\Lnorm{6}} \leq C \norm{\dt \xi^o}{\Lnorm{2}} \norm{v}{\Lnorm{2}}^{1/2} \norm{v}{\Hnorm{1}}^{3/2}\\
		& ~~~~ \leq \mu \norm{\nabla v}{\Lnorm{2}}^2 + C \bigl( \norm{\dt \xi^o}{\Lnorm{2}}^4 + 1\bigr) \norm{v}{\Lnorm{2}}^2.
	\end{align*}
	 Thus applying Gr\"onwall's inequality to above implies $ (\xi, v) \equiv 0 $ for every $ t \in [0,T] $ thanks to the fact $\xi^o + \dfrac{1}{2} gz \geq \dfrac{1}{2} \underline\rho > 0 $. This finishes the proof.
	\end{proof}
	
	
	\subsubsection{A priori estimates for the inhomogeneous linear system}\label{sec:apriori-linear-g}
	In this subsection, we show that the map $ \mathcal{T} $ defined in \eqref{def:map-isen-g} is a well defined map from $ \mathfrak X $ into $ \mathfrak X $.
	Indeed, we will show the following proposition:
	\begin{proposition}\label{prop:maps-set2set-g}
	Consider the initial data with the bounds $ M_0, M_1 $ in \eqref{bound-of-initial-data-linear-g} and $ (\xi^o, v^o) \in \mathfrak X = \mathfrak X_T $.
		There is a $ T_g = T_g(M_0,M_1,\mu,\lambda,\underline\rho) > 0 $ sufficiently small such that for any $ T \in (0, T_g) $,
		there exists a unique solution to \eqref{eq:lin-isen-g}. Moreover, the solution		
		belongs to $ \mathfrak X = \mathfrak X_T $. Therefore, for any such $ T $, the map $ \mathcal T $ in \eqref{def:map-isen-g} is a well defined map from $ \mathfrak X $ into $ \mathfrak X $.
	\end{proposition}
	To show this proposition, we choose approximating sequences of $ (\xi^o, v^o) $ (respectively, $ (\xi_0, v_0) $) in $ C^\infty(\overline\Omega\times[0,T]) \cap \mathfrak X_T $ (respectively, $ H^3(\Omega) $). Then Proposition \ref{prop:exist-lin-isen-g} guarantees there exists a unique solution to system \eqref{eq:lin-isen-g} with the regularity stated in \eqref{rg-aprox}. In particular,  the estimates below on the equations are allowed and rigorous. Also, as one will see, the estimates depend only on $ M_0, M_1 $. Then after taking a subsequence if necessary, the estimates below hold for $ (\xi^o,v^o) $ (respectively, $ (\xi_0, v_0) $) in $ \mathfrak X_T $ (respectively, $ H^2(\Omega) $). Therefore we conclude the existence of solutions to system \eqref{eq:lin-isen-g} with $ (\xi^o, v^o) \in \mathfrak X $ and initial data in \eqref{bound-of-initial-data-linear-g}.
	 \begin{proof}[Proof of Proposition \ref{prop:maps-set2set-g}]
		The existence of solutions is a direct consequence of Proposition \ref{prop:estimates-xi-g} and Proposition \ref{prop:estimates-v-g}, below, and the approximating arguments discussed above. The uniqueness of solutions follows from the same arguments in the proof of Proposition \ref{prop:exist-lin-isen-g}.
	\end{proof}
%
%
%
Hereafter, to simplify the arguments, we assume that the solution $ (\xi, v) $  to the linear system \eqref{eq:lin-isen-g} is smooth enough so that the regularity in \eqref{rg-aprox} holds and the following estimates are rigorous. In particular, all the operations applied on the equations below are allowed.

We start by establishing some estimates for the solutions of \subeqref{eq:lin-isen-g}{1}.
	In particular, we will establish the following estimates on $ \xi $:
	\begin{itemize}
		\item The lower bound for $ \xi $;
		\item The $ H^2(\Omega) $ norm for $ \xi $;
		\item The $ \Hnorm{1} $ norm for $ \dt \xi $.
	\end{itemize}	
	Indeed, we will show the following:
	\begin{proposition}\label{prop:estimates-xi-g}
	There exists a
		$  T'= T'(M_0,C_1 M_1,\underline\rho) > 0 $ sufficiently small such that for any $ T \in (0,T'] $, the solution $ \xi $ to \subeqref{eq:lin-isen-g}{1} 
		satisfies  the following,
		\begin{itemize}
			\item $ \xi + \dfrac{1}{2} gz \geq \dfrac{1}{2} \underline{\rho} $;
			\item $ \sup_{0\leq t\leq T} \norm{\xi(t)}{\Hnorm{2}}^2 \leq 2 M_0 $;
			\item $ \sup_{0\leq t\leq T} \norm{\dt \xi(t)}{\Hnorm{1}}^2 \leq C_2 $,
		\end{itemize}
		where $ M_0 $ is as in \eqref{bound-of-initial-data-linear-g}.
	\end{proposition}
	
	{\par\hfill\par\noindent\bf The lower bound for $ \xi $ \par}
	
	In order to derive the lower bound of $ \xi $, we employ the following Stampaccia-like argument. Let $ M = M(t) > 0 $ be a nonnegative integrable function to be determined later. Consider $ \eta = \eta(x,y,t) := \xi - \underline\rho + \int_0^t M(s) \,ds  $. Then according to \subeqref{eq:lin-isen-g}{1}, $ \eta $ satisfies the equation
	\begin{align*}
		& \dt \eta + \overline{v^o} \cdot \nablah \eta + \eta \overline{\dvh v^o} = -(\underline\rho - \int_0^t M(s) \,ds ) \overline{\dvh v^o} \\
		& ~~~~ - \dfrac{g}{2} \overline{z \dvh v^o} + M(t).
	\end{align*}
	Let $$ \mathbbm{1}_{\lbrace \eta < 0 \rbrace} = \begin{cases}
		1 & \text{in} ~ \lbrace \eta < 0 \rbrace,\\
		0 & \text{otherwise},
	\end{cases} $$ be the characteristic function of the set $ \lbrace \eta < 0 \rbrace $ and denote by $ \eta_- := - \eta \mathbbm{1}_{\lbrace \eta < 0 \rbrace} \geq 0  $. Observe that since $ \xi \in H^1(\Omega\times[0,T]) $, so it $ \eta_- $.
	Then multiplying the above equation with $ - \mathbbm{1}_{\lbrace \eta < 0 \rbrace} $ and integrating the resultant in $ \Omega_h $ yield
	\begin{equation*}
		 \dfrac{d}{dt} \inth \eta_- \idxh = \int_{\lbrace \eta < 0\rbrace}\biggl( (\underline\rho - \int_0^t M(s)\,ds) \overline{\dvh v^o} + \dfrac{g}{2} \overline{z \dvh v^o} - M(t) \biggr) \idxh.
	 \end{equation*}
	 Now, let $ 0 < M(t) := C \max\lbrace \hnorm{\overline{\dvh v^o}}{\Lnorm\infty}, \hnorm{\overline{z \dvh v^o}}{\Lnorm{\infty}} \rbrace \leq C \norm{v^o}{\Hnorm{3}} < \infty, ~ a.e., $ for some constant $ C>0 $. Then the integrand on the right-hand side of the above equation satisfies
	 \begin{align*}
	 	& (\underline\rho - \int_0^t M(s)\,ds) \overline{\dvh v^o} + \dfrac{g}{2} \overline{z \dvh v^o} - M(t) \\
	 	& ~~~~ \leq \dfrac{1}{C} \bigl( \underline\rho + C \int_0^T \norm{v^o}{\Hnorm{3}}(s) \,ds + \dfrac{g}{2} \bigr)M(t) - M(t) \\
	 	& ~~~~ \leq \dfrac{1}{C} \bigl( \underline\rho + 1 + \dfrac{g}{2} \bigr)M(t) - M(t)  < 0,
	 \end{align*}
	 provided $ C $ is large enough and $ T $ is small enough such that
	 \begin{align*}
		  \dfrac{1}{C} \bigl( \underline\rho + 1 + \dfrac{g}{2} \bigr) & < 1, ~ \text{and}  \\
		 C \int_0^T \norm{v^o}{\Hnorm{3}}(s) \,ds & \leq C T^{1/2} \bigl( \int_0^T \norm{v^o}{\Hnorm{3}}^2(s)\,ds\bigr)^{1/2} \\
		 & \leq C C_1^{1/2}M_1^{1/2}T^{1/2} < 1 .
	 \end{align*}
	 Therefore we have
	 $$ \dfrac{d}{dt} \inth \eta_- \idxh \leq 0 ~~ a.e. $$
	 which, after
	 integrating over $ [0,t_0] $ for any $ t_0 \in [0,T] $, thanks to the fact $ \eta_-(0) \equiv 0 $, yields
	\begin{equation}\label{lowerbounds-g}
		\inth \eta_-(t_0) \idxh \leq 0.
	\end{equation}
	Hence $ \eta_- = 0 ~ \text{in} ~ \Omega_h \times[0,T] $. That is,
	$ \eta(t) = \xi(t) - \underline\rho + \int_0^t M(s) \,ds \geq 0 $ and
	\begin{equation}\label{lower-bound-conti-g}
	\begin{aligned}
		\xi(t) + \dfrac{1}{2} gz & > \xi(t) \geq \underline\rho - C \int_0^T \norm{v^o}{\Hnorm{3}}(s) \,ds \\
		&  \geq \underline\rho - C C_1^{1/2} M_1^{1/2} T^{1/2} \geq \dfrac{1}{2} \underline\rho,
	\end{aligned}	
	\end{equation}
	for $ t \in [0,T], ~ T \leq T_1 $, with $ T_1= T_1(C_1M_1, \underline\rho) $ sufficiently small.

{\par \hfill \par\noindent\bf The $ H^2(\Omega) $ norm for $\xi $\par}
	Since $ \xi $ is independent of the vertical variable, it is sufficient to estimate the horizontal derivatives. Denote, hereafter, $ \partial_h \in \lbrace \partial_{x} , \partial_y \rbrace $ and $ \partial_{hh} = \partial_h^2 $, etc. Applying $ \partial_{hh} $ to \subeqref{eq:lin-isen-g}{1} will give us the following,
	\begin{equation}\label{eq:horizontal-conti-g}
		\begin{aligned}
			& \dt \partial_{hh} \xi + \overline{v^o} \cdot \nablah \partial_{hh} \xi + 2 \overline{\partial_h v^o} \cdot\nablah \partial_h \xi + \partial_{hh} \xi \overline{\dvh v^o} + \overline{\partial_{hh} v^o} \cdot \nablah \xi \\
			& ~~~~ + 2 \partial_h \xi \overline{\dvh \partial_h v^o} + \xi \overline{\dvh \partial_{hh} v^o} + \dfrac{g}{2} \overline{z \dvh \partial_{hh} v^o} = 0.
		\end{aligned}
	\end{equation}
	Multiply \eqref{eq:horizontal-conti-g} with $ \partial_{hh}\xi $ and integrate the resultant in $ \Omega_h $. It holds
	\begin{align*}
		& \dfrac{1}{2} \dfrac{d}{dt}\hnorm{\partial_{hh}\xi}{\Lnorm{2}}^2 + \inth \biggl( \dfrac{1}{2} \overline{\dvh v^o} \abs{\partial_{hh}\xi}{2} + ( 2 \overline{\partial_h v^o} \cdot \nablah \partial_h \xi ) \partial_{hh}\xi \biggr) \idxh \\
		& ~~~~ + \inth \biggl( (\overline{\partial_{hh} v^o} \cdot \nablah \xi) \partial_{hh} \xi + 2 \partial_h \xi \overline{\dvh \partial_h v^o}  \partial_{hh} \xi \biggr) \idxh \\
		& ~~~~ + \inth \xi \overline{\dvh \partial_{hh} v^o} \partial_{hh}\xi \idxh
		 + \dfrac g 2 \inth \overline{z \dvh \partial_{hh} v^o} \partial_{hh} \xi \idxh = 0.
	\end{align*}
%
	Then after applying H\"older inequality and Sobolev embedding inequalities, we have
	\begin{equation*}
		\dfrac{d}{dt} \hnorm{\partial_{hh}\xi}{\Lnorm{2}}^2 \lesssim \norm{v^o}{\Hnorm{3}} \hnorm{\xi}{\Hnorm{2}}^2 + \norm{v^o}{\Hnorm{3}} \hnorm{\xi}{\Hnorm{2}}.
	\end{equation*}
	Similar arguments also hold for the lower order derivatives. Therefore, one has, since $ \hnorm{\xi }{\Hnorm{2}} = \norm{\xi}{\Hnorm{2}} $,
	\begin{equation*}
		\dfrac{d}{dt} \norm{\xi}{\Hnorm{2}}^2 \leq C \norm{v^o}{\Hnorm{3}}\norm{\xi}{\Hnorm{2}}^2 + C \norm{v^o}{\Hnorm{3}}\norm{\xi}{\Hnorm{2}}.
	\end{equation*}	
	This will imply after applying the Gr\"onwall's inequality
	\begin{equation}\label{001-g}
		\begin{aligned}
			& \sup_{0\leq t\leq T}\norm{\xi(t)}{\Hnorm{2}}^2 \leq e^{C C_1^{1/2}M_1^{1/2}T^{1/2}} \norm{\xi_0}{\Hnorm{2}}^2 \\
			& ~~~~ ~~~~ + C C_1^{1/2}M_1^{1/2}T^{1/2} e^{C C_1^{1/2}M_1^{1/2}T^{1/2}} \sup_{0<t<T} \norm{\xi }{\Hnorm{2}}\\
			& ~~ \leq \dfrac{1}{4} \sup_{0\leq t\leq T} \norm{\xi}{\Hnorm{2}}^2  + C^2 C_1M_1 e^{2C C_1^{1/2}M_1^{1/2}T^{1/2}}T \\
			& ~~~~ ~~~~ + e^{C C_1^{1/2}M_1^{1/2}T^{1/2}} M_0,
		\end{aligned}
	\end{equation}
	where we have plugged in above the estimate
	\begin{equation*}
		C \int_0^T \norm{v^o}{\Hnorm{3}}\,dt \leq C \bigl( \int_0^T \norm{v^o}{\Hnorm{3}}^2\,dt\bigr)^{1/2} T^{1/2} \leq C C_1^{1/2}M_1^{1/2}T^{1/2}.
	\end{equation*}
	Then for $ T \in (0, T_2] $ with $ T_2(M_0,C_1M_1) $ sufficiently small, \eqref{001-g} yields
	\begin{equation}\label{H2-conti-g}
		\sup_{0\leq t\leq T} \norm{\xi(t)}{\Hnorm{2}}^2 \leq 2 M_0.
	\end{equation}
	{\par\hfill\par\noindent\bf The $ H^1(\Omega) $ norm for $ \dt \xi $ \par}
	Applying $ \partial_h $ to \subeqref{eq:lin-isen-g}{1} gives us the following,
	\begin{align*}
		& \dt \partial_h \xi = - \overline{v^o} \cdot \nablah \partial_h \xi - \overline{\partial_h v^o} \cdot \nablah \xi - \partial_h \xi \overline{\dvh v^o} \\
		& ~~~~ - \xi \overline{\dvh \partial_h v^o} - \dfrac g 2 \overline{z \dvh \partial_h v^o}.
	\end{align*}
	Therefore, direct estimates imply
	\begin{equation*}
		\hnorm{\dt \partial_h \xi }{\Lnorm{2}} \lesssim \hnorm{\overline{v^o}}{\Hnorm{2}}\hnorm{\xi}{\Hnorm{2}} + \hnorm{\overline{v^o}}{\Hnorm{2}}.
	\end{equation*}
	Similar estimates also hold for $ \norm{\dt \xi }{\Lnorm{2}} $. Hence
	\begin{equation}\label{H1-t-conti-g}
		\norm{\dt \xi }{\Hnorm{1}}^2 \leq C \norm{v^o}{\Hnorm{2}}^2\norm{\xi}{\Hnorm{2}}^2  + C \norm{v^o}{\Hnorm{2}}^2 \leq C_2,
	\end{equation}
	where $ C_2 = C_2(M_0,C_1M_1) $ is given by
	\begin{equation}\label{def:C2}
		C (1 + 2M_0)C_1M_1 =: C_2.
	\end{equation}
	\begin{proof}[Proof of Proposition \ref{prop:estimates-xi-g}]
		By choosing $ T' = \min\lbrace T_1,T_2 \rbrace $, the proof of the proposition is the direct consequence of \eqref{lower-bound-conti-g}, \eqref{H2-conti-g} and \eqref{H1-t-conti-g}.
	\end{proof}
	
	Next, we shall perform some estimates for the solutions of \subeqref{eq:lin-isen-g}{2}. In particular, we will establish the following estimates:
	\begin{itemize}
		\item Horizontal spatial derivative estimates for $ v $;
		\item Time derivative estimates for $ v $;
		\item Vertical derivative estimates for $ v $.
	\end{itemize}
	Then we will use the estimates mentioned above to show the following:
	\begin{proposition}\label{prop:estimates-v-g}
	There exists a
		$  T''=T''(M_0,M_1,C_1,C_2,\underline\rho) \in(0,\infty) $ sufficiently small such that for every $ T \in (0, T''] $, the solution $ v $ to \subeqref{eq:lin-isen-g}{2} 
		satisfies
		\begin{equation*}
			\sup_{0\leq t\leq T} (\norm{v(t)}{\Hnorm{2}}^2 + \norm{v_t(t)}{\Lnorm{2}}^2 ) + \int_0^T \biggl( \norm{v(t)}{\Hnorm{3}}^2 + \norm{v_t(t)}{\Hnorm{1}}^2\biggr) \,dt 
			\leq C_1M_1.
		\end{equation*}
	\end{proposition}
	
	{\par\hfill \par\noindent\bf Horizontal spatial derivative estimates for $ v $\par}
	Applying $ \partial_{hh} $ to \subeqref{eq:lin-isen-g}{2} will yield the following,
	\begin{equation}\label{eq:horizontal-v-g}
		\begin{aligned}
			& \rho^o \dt \partial_{hh} v - \mu \deltah \partial_{hh} v - \mu \partial_{zz} \partial_{hh} v - (\mu + \lambda) \nablah \dvh \partial_{hh} v \\
			& ~~  = - 2 \partial_h \rho^o \dt \partial_h v - \partial_{hh} \rho^o \dt v - \partial_{hh} (\rho^o v^o \cdot \nablah v^o) - \partial_{hh} (\rho^o w^o \dz v^o) \\
			& ~~~~ - \partial_{hh}((2\xi^o + gz) \nablah \xi^o),
		\end{aligned}
	\end{equation}
	where $ \rho^o  = \xi^o + \frac{1}{2} gz $. Notice that the boundary condition in \eqref{bd-conds-lin-g} will ensure that after integration by parts there are has no boundary terms in the following estimates. After taking the inner product of \eqref{eq:horizontal-v-g} with $ \partial_{hh} v $ and integrating by parts, one has
	\begin{align}\label{002-g}
			& \dfrac{d}{dt} \biggl\lbrace \dfrac{1}{2} \intw \rho^o \abs{\partial_{hh}v}{2} \idx \biggr\rbrace + \intw \biggl( \mu \abs{\nablah \partial_{hh}v}{2} + \mu \abs{\partial_{hhz} v}{2} \nonumber
			 + (\mu+\lambda) \\
			 & ~~~~ \times \abs{\dvh \partial_{hh} v}{2} \biggr) \idx {\nonumber}
			 = - \intw \biggl( 2 \partial_h \rho^o \dt \partial_h v \cdot \partial_{hh} v + \partial_{hh} \rho^o \dt v \cdot \partial_{hh} v \biggr) \idx \\
			& ~~~~ +  \dfrac{1}{2} \intw \dt \rho^o \abs{\partial_{hh}v^o}{2} \idx{\nonumber}
			 + \intw \partial_h ( \rho^o v^o \cdot\nablah v^o) \cdot \partial_{hhh} v \idx \nonumber\\
			& ~~~~ + \intw \partial_h(\rho^o w^o \dz v^o) \cdot \partial_{hhh} v \idx {\nonumber}
			 + \intw \partial_h((2\xi^o + gz) \nablah \xi^o) \cdot \partial_{hhh} v \idx\\
			 & ~~~~  =: \sum_{i=1}^{5}I_{i}.
	\end{align}
	Notice that the regularity in \eqref{rg-aprox} guarantees the inner product above is allowed. For instance, $ \rho^o \dt \partial_{hh} v \in L^2(0,T;L^2(\Omega)) $ and $ \partial_{hh}v \in L^2(0,T;H^2(\Omega)) $, and therefore $ \int \rho^o \dt \partial_{hh} v \cdot \partial_{hh} v \idx $ is allowed as an integrable quantity and the fundamental theorem of calculus can be applied.
	We list below the estimates of the $ I_i$ terms. We will use the fact $ \norm{\rho^o}{\Hnorm{2}}^2 \leq C \norm{\xi^o}{\Hnorm{2}}^2 + C g^2 \leq C M_0 + C $, $ \norm{\dt \rho^o}{\Lnorm{2}}^2 = \norm{\dt \xi^o}{\Lnorm{2}}^2 \leq C_2$. Also, hereafter the following estimates hold for every $ \delta, \omega >0 $ which will be chosen later to be adequately small. Correspondingly, $ C_{\delta}, C_{\omega}, C_{\delta,\omega} $ are some positive constants depending on $ \delta, \omega $.
	\begin{align*}
		& I_1
		\lesssim \norm{\partial_h \rho^o}{\Lnorm{6}} \norm{\nablah v_t}{\Lnorm{2}} \norm{\partial_{hh}v}{\Lnorm{3}} \\
		& ~~~~ ~~~~  + \norm{\partial_{hh}\rho^o}{\Lnorm{2}} \norm{\dt v}{\Lnorm{3}} \norm{\partial_{hh} v}{\Lnorm{6}} \\
		& ~~~~  \lesssim \norm{\rho^o}{\Hnorm{2}} \norm{\nablah v_t}{\Lnorm{2}} \norm{v}{\Hnorm{2}}^{1/2} \norm{\partial_{hh}v}{\Hnorm{1}}^{1/2}\\
		& ~~~~ ~~~~ + \norm{\rho^o}{\Hnorm{2}}\norm{\dt v}{\Lnorm{2}}^{1/2} \norm{\dt v}{\Hnorm{1}}^{1/2} \norm{\partial_{hh}v}{\Hnorm{1}}  \\
		& ~~~~ \lesssim \delta \norm{\partial_{hh}v}{\Hnorm{1}}^2 + \omega \norm{v_t}{\Hnorm{1}}^2 + C_{\delta,\omega}
		(M_0^2 + 1)  (\norm{v}{\Hnorm{2}}^2 \\
		& ~~~~ ~~~~ ~~~~ + \norm{\dt v}{\Lnorm{2}}^2). \\
		& I_2
		\lesssim \norm{\dt \rho^o}{\Lnorm{2}} \norm{\partial_{hh} v^o}{\Lnorm{3}} \norm{\partial_{hh} v^o}{\Lnorm{6}} \lesssim \norm{\dt \rho^o}{\Lnorm{2}} \norm{v^o}{\Hnorm{2}}^{1/2} \\
		& ~~~~ ~~~~ \times \norm{v^o}{\Hnorm{3}}^{3/2} \lesssim \omega \norm{v^o}{\Hnorm{3}}^2 + C_\omega
		C_2^2 C_1M_1. \\
		& I_3 \lesssim ( \norm{\partial_h \rho^o}{\Lnorm{6}} \norm{v^o}{\Lnorm{\infty}} \norm{\nablah v^o}{\Lnorm{3}} + \norm{\rho^o}{\Lnorm{\infty}}\norm{\partial_h v^o}{\Lnorm{6}} \norm{\nablah v^o}{\Lnorm{3}} \\
		& ~~~~ ~~~~ ~~~~ + \norm{\rho^o}{\Lnorm{\infty}} \norm{v^o}{\Lnorm{\infty}} \norm{\nablah^2 v^o}{\Lnorm{2}}    ) \norm{\partial_{hhh}v}{\Lnorm{2}} \\
		& ~~~~ \lesssim \norm{\rho^o}{\Hnorm{2}} \norm{v^o}{\Hnorm{2}}^2 \norm{\partial_{hh} v}{\Hnorm{1}} \lesssim \delta \norm{\partial_{hh} v}{\Hnorm{1}}^2 + C_\delta (M_0 + 1) C_1^2 M_1^2.  \\
		& I_5 \lesssim (\norm{\nablah \xi^o}{\Lnorm{4}}^2 + \norm{\xi^o}{\Lnorm{\infty}} \norm{\nablah^2 \xi^o}{\Lnorm{2}} + \norm{\nablah^2 \xi^o }{\Lnorm{2}} ) \norm{\partial_{hhh}v}{\Lnorm{2}}\\
		& ~~~~ \lesssim (\norm{\xi^o}{\Hnorm{2}}^2 + 1) \norm{\partial_{hh} v}{\Hnorm{1}} \lesssim \delta \norm{\partial_{hh} v}{\Hnorm{1}}^2 + C_\delta( M_0^2+ 1).
	\end{align*}
	In order to estimate $ I_4 $, we shall plug in \eqref{vertical-lin-isen-g}. One has
	\begin{align*}
		& I_4 = \intw \partial_h (\rho^o w^o) \dz v^o \cdot \partial_{hhh} v\idx + \intw \rho^o w^o \partial_{hz} v^o \cdot \partial_{hhh} v\idx \\
		& ~~~~ = - \int_0^1 \inth \biggl\lbrack \int_0^z \bigl\lbrack \partial_h\dvh(\xi^o \widetilde{v^o}) + \dfrac{g}{2} \widetilde{z \dvh \partial_h v^o} \bigr\rbrack  \,dz'  \dz v^o \cdot \partial_{hhh} v  \biggr\rbrack \idxh \,dz \\
		& ~~~~ ~~~~ - \int_0^1  \inth \biggl\lbrack \int_0^z \bigl\lbrack \dvh(\xi^o \widetilde{v^o}) + \dfrac{g}{2} \widetilde{z \dvh  v^o} \bigr\rbrack  \,dz'  \partial_{hz} v^o \cdot \partial_{hhh} v \biggr\rbrack  \idxh \,dz \\
		& ~~~~ = :I_4' + I_4''.
	\end{align*}
	Then applying the Minkowski's and the Sobolev embedding inequalities yields
	\begin{align*}
		& I_4'' = - \int_0^1  \int_0^z \inth \bigl\lbrack \dvh(\xi^o \widetilde{v^o}) + \dfrac{g}{2} \widetilde{z \dvh  v^o} \bigr\rbrack{(\vech{x},z',t)} \\
		& ~~~~ ~~~~ ~~~~ \times \bigl\lbrack \partial_{hz} v^o \cdot \partial_{hhh} v \bigr\rbrack(\vech{x},z,t)  \idxh \,dz' \,dz \\
		&  \lesssim \int_0^1 \biggl( \hnorm{\nablah \xi^o}{\Lnorm{4}} \hnorm{\widetilde{v^o}}{\Lnorm{\infty}} + \hnorm{\xi^o}{\Lnorm{\infty}} \hnorm{\widetilde{\nablah v^o}}{\Lnorm{4}} + \hnorm{\widetilde{\nablah v^o}}{\Lnorm{4}} \biggr) \,dz'\\
		& ~~~~  \times \int_0^1 \hnorm{\partial_{hz} v^o}{\Lnorm{4}} \hnorm{\partial_{hhh} v}{\Lnorm{2}} \,dz \lesssim (\norm{\xi^o}{\Hnorm{2}} + 1)\norm{v^o}{\Hnorm{2}}^{3/2}\norm{v^o}{\Hnorm{3}}^{1/2}\\
		& ~~~~ \times \norm{\partial_{hh} v}{\Hnorm{1}} \lesssim \delta \norm{\partial_{hh} v}{\Hnorm{1}}^2 +\omega\norm{v^o}{\Hnorm{3}}^2 + C_{\delta,\omega} ( M_0^2 + 1) C_1^3M_1^{3},\\
		& I_4' = - \int_0^1 \int_0^z \inth \bigl\lbrack \partial_h\dvh(\xi^o \widetilde{v^o}) + \dfrac{g}{2} \widetilde{z \dvh \partial_h v^o} \bigr\rbrack(\vech{x},z',t)  \\
		& ~~~~ ~~~~ ~~~~~ \times \bigl\lbrack \dz v^o \cdot \partial_{hhh} v  \bigr\rbrack(\vech{x},z,t) \idxh \,dz' \,dz \\
		& ~~~~ \lesssim \int_0^1 \biggl( \hnorm{\nablah^2 \xi^o}{\Lnorm{2}} \hnorm{\widetilde{v^o}}{\Lnorm{\infty}} + \hnorm{\nablah \xi^o}{\Lnorm{4}} \hnorm{\widetilde{\nablah v^o}}{\Lnorm{4}} + \hnorm{\xi^o}{\Lnorm{\infty}} \hnorm{\widetilde{\nablah^2 v^o}}{\Lnorm{2}}\\
		& ~~~~ ~~~~ + \hnorm{\widetilde{\nablah^2 v^o}}{\Lnorm{2}} \biggr)  \,dz'  \times \int_0^1 \hnorm{\dz v^o}{\Lnorm{\infty}} \hnorm{\partial_{hhh} v}{\Lnorm{2}} \,dz \\
		& ~~~~ \lesssim \int_0^1 (\hnorm{\xi^o}{\Hnorm{2}} + 1) \hnorm{\widetilde{v^o}}{\Hnorm{2}} \,dz' \times \int_0^1 \hnorm{\dz v^o}{\Hnorm{1}}^{1/2} \hnorm{\dz v^o}{\Hnorm{2}}^{1/2} \hnorm{\partial_{hhh} v}{\Lnorm{2}} \,dz\\
		& ~~~~ \lesssim (\norm{\xi^o}{\Hnorm{2}} + 1) \norm{v^o}{\Hnorm{2}}^{3/2} \norm{v^o}{\Hnorm{3}}^{1/2} \norm{\partial_{hh} v}{\Hnorm{1}} \lesssim \delta \norm{\partial_{hh}v}{\Hnorm{1}}^2 + \omega \norm{v^o}{\Hnorm{3}}^2 \\
		& ~~~~ ~~~~ + C_{\delta,\omega} (M_0^2 + 1) C_1^3M_1^3,
	\end{align*}
	where we have employed \eqref{ineq-supnorm}. Summing the above inequalities with $ \delta $ small enough yields the following estimate
	\begin{equation*}
		\begin{aligned}
			& \dfrac{d}{dt} \norm{\sqrt{\rho^o}\partial_{hh} v}{\Lnorm{2}}^2 + c_{\mu, \lambda} \norm{\partial_{hh} v}{\Hnorm{1}}^2 \lesssim \omega ( \norm{\dt v}{\Hnorm{1}}^2 + \norm{v^o}{\Hnorm{3}}^2 ) \\
			& ~~~~ + C_\omega \mathcal H(M_0,C_1M_1,C_2) (\norm{v}{\Hnorm{2}}^2 + \norm{\dt v}{\Lnorm{2}}^2 + 1).
		\end{aligned}
	\end{equation*}
	Hereafter, $ \mathcal H $ will be used to denote a polynomial quantity of its arguments (i.e., the norms of the initial data and $\xi^o, v^o$) which may be different from line to line. Also $ c_{\mu,\lambda}, C_\omega $ denote positive constants depending on $ \mu, \lambda $ and $ \omega $, respectively. Similar arguments also hold for the lower order derivatives. Then after suitable choice of $ \omega $, one has
	\begin{equation}\label{Horiz-v-g}
		\begin{aligned}
			& \dfrac{d}{dt}\bigl(  \norm{\sqrt{\rho^o}v}{\Lnorm{2}}^2 +\norm{\sqrt{\rho^o}\nablah v}{\Lnorm{2}}^2 + \norm{\sqrt{\rho^o}\nablah^2 v}{\Lnorm{2}}^2 \bigr) + c_{\mu, \lambda}\bigl( \norm{v}{\Hnorm{1}}^2 \\
			& ~~~~ ~~~~ + \norm{\nablah v}{\Hnorm{1}}^2 + \norm{\nablah^2 v}{\Hnorm{1}}^2 \bigr) \leq \omega \bigl( \norm{\dt v}{\Hnorm{1}}^2 + \norm{v^o}{\Hnorm{3}}^2 \bigr) \\
			& ~~~~ + C_\omega \mathcal H(M_0,C_1M_1,C_2) \bigl(\norm{v}{\Hnorm{2}}^2 + \norm{\dt v}{\Lnorm{2}}^2 + 1\bigr).
		\end{aligned}
	\end{equation}
	{\par\hfill\par\noindent\bf Time derivative estimates for $ v $\par}
	Observe, from \subeqref{eq:lin-isen-g}{2}, we have
	\begin{align*}
		& \dt v = - v^o \cdot\nablah v^o - w^o \dz v^o - 2 \nablah \xi^o \\
		& ~~~~ + (\rho^o)^{-1} \bigl( \mu \deltah v + \mu \partial_{zz} v + (\mu +\lambda) \nablah \dvh v \bigr).
	\end{align*}
	with the right-hand side which is differentiable in time with value in $ L^2(0,T;L^2(\Omega)) $. Hence $ \dt v_t \in L^2(0,T;L^2(\Omega)) $, and  we can apply
	$ \dt $ to \subeqref{eq:lin-isen-g}{2}. This yields
	\begin{equation}\label{eq:temporal-v-g}
		\begin{aligned}
			& \rho^o \partial_t v_t - \mu \deltah v_t - \mu \partial_{zz} v_t - (\mu + \lambda) \nablah \dvh v_t = - \dt \rho^o \dt v \\
			& ~~~~ - \dt (\rho^o v^o \cdot \nablah v^o) - \dt (\rho^o w^o \dz v^o) - \dt ((2\xi^o + gz) \nablah \xi^o ).
		\end{aligned}
	\end{equation}
	Taking the inner product of \eqref{eq:temporal-v-g} with $ v_t $ and integrating the resultant yield
	\begin{align}\label{003-g}
		& \dfrac{d}{dt} \dfrac{1}{2} \intw \rho^o \abs{\dt v}{2} \idx + \intw \biggl( \mu \abs{\nablah v_t}{2} + \mu \abs{\dz v_t}{2} + (\mu+\lambda) \abs{\dvh v_t}{2} \biggr) \idx {\nonumber} \\
		& ~~ = - \dfrac{1}{2} \intw \dt \rho^o \abs{v_t}{2} \idx - \intw \dt (\rho^o v^o \cdot\nablah v^o) \cdot \dt v \idx \nonumber \\
		& ~~~~ - \intw \dt(\rho^o w^o \dz v^o) \cdot \dt v\idx - \intw \dt ((2\xi^o + gz) \nablah \xi^o ) \cdot \dt v \idx \nonumber\\
		& ~~~~ =: \sum_{i=6}^{9}I_i.
	\end{align}
	We now provide estimates for the right-hand side terms of \eqref{003-g}.
	\begin{align*}
		& I_6 \lesssim \norm{\dt \rho^o}{\Lnorm{2}} \norm{\dt v}{\Lnorm{3}} \norm{\dt v}{\Lnorm{6}} \lesssim \norm{\dt \rho^o}{\Lnorm{2}} \norm{\dt v}{\Lnorm{2}}^{1/2} \norm{\dt v}{\Hnorm{1}}^{3/2} \\
		& ~~~~ \lesssim \delta \norm{\dt v}{\Hnorm{1}}^2 + C_\delta C_2^2 \norm{\dt v}{\Lnorm{2}}^2. \\
		& I_7 \lesssim (\norm{\dt \rho^o}{\Lnorm{2}}\norm{v^o}{\Lnorm{\infty}} \norm{\nablah v^o}{\Lnorm{6}}+ \norm{\rho^o}{\Lnorm{\infty}} \norm{\dt v^o}{\Lnorm{2}}\norm{\nablah v^o}{\Lnorm{6}} \\
		& ~~~~ ~~~~ + \norm{\rho^o}{\Lnorm{\infty}}\norm{v^o}{\Lnorm{6}} \norm{\nablah v^o_t}{\Lnorm{2}} ) \norm{\dt v}{\Lnorm{3}}\lesssim (\norm{\dt \rho^o}{\Lnorm{2}} \norm{v^o}{\Hnorm{2}}^2 \\
		& ~~~~ ~~~~ + \norm{\rho^o}{\Hnorm{2}}\norm{\dt v^o}{\Lnorm{2}}\norm{v^o}{\Hnorm{2}} + \norm{\rho^o}{\Hnorm{2}}\norm{v^o}{\Hnorm{1}}\norm{v_t^o}{\Hnorm{1}} ) \\
		& ~~~~ \times \norm{\dt v}{\Lnorm{2}}^{1/2} \norm{\dt v}{\Hnorm{1}}^{1/2} \lesssim \delta \norm{\dt v}{\Hnorm{1}}^2 +\omega (\norm{v^o}{\Hnorm{3}}^2 + \norm{\dt v^o}{\Hnorm{1}}^2) \\
		& ~~~~ + C_{\delta,\omega} (C_2^2 + M_0^2 + 1) C_1^2M_1^2  \norm{\dt v}{\Lnorm{2}}^2. \\
		& I_{9} = \intw (2\xi^o + gz) \dt \xi^o \dvh v_t \idx \lesssim (\norm{\xi^o}{\Lnorm{\infty}} + 1) \norm{\dt \xi^o}{\Lnorm{2}} \norm{\dvh \dt v}{\Lnorm{2}} \\
		& ~~~~ \lesssim \delta \norm{\dt v}{\Hnorm{1}}^2 + C_\delta (M_0 + 1) C_2.
	\end{align*}
	In order to estimate $ I_8 $, we first substitute \eqref{vertical-lin-isen-g} and thus we have,
	\begin{align*}
		& I_8 = - \intw \dt (\rho^o w^o) \dz v^o \cdot \dt v\idx - \intw \rho^o w^o \dz v_t^o \cdot \dt v\idx \\
		& ~~~~ = \int_0^1 \inth \biggl\lbrack \int_0^z \bigl\lbrack \dvh(\xi^o_t \widetilde{v^o}) + \dvh (\xi^o \widetilde{v^o_t}) + \dfrac{g}{2}\widetilde{z\dvh v^o_t} \bigr\rbrack \,dz' \\
		& ~~~~ ~~~~ \times \bigl( \dz v^o \cdot \dt v \bigr) \biggr\rbrack \idxh \,dz
		 + \int_0^1 \inth \biggl\lbrack \int_0^z \bigl\lbrack \dvh(\xi^o \widetilde{v^o}) + \dfrac{g}{2}\widetilde{z\dvh v^o} \bigr\rbrack \,dz' \\
		 & ~~~~ ~~~~ \times \bigl( \dz v_t^o \cdot \dt v \bigr) \biggr\rbrack \idxh \,dz =: I_8' + I_8''.
	\end{align*}
	Then we apply \eqref{ineq-supnorm}, the Minkowski's and the Sobolev embedding inequalities as follows,
	\begin{align*}
		& I_8' = - \int_0^1  \int_0^z \inth \bigl\lbrack \xi^o_t \widetilde{v^o} + \xi^o \widetilde{v_t^o} + \dfrac g 2 \widetilde{z v_t^o} \bigr\rbrack(\vech{x},z',t) \\
		& ~~~~ ~~~~ ~~~~ \cdot \bigl\lbrack\nablah (\dz v^o \cdot \dt v)\bigr\rbrack(\vech{x},z,t) \idxh \,dz'  \,dz  \\
		& ~~~~ \lesssim \int_0^1 \biggl( \hnorm{\xi_t^o}{\Lnorm{2}}\hnorm{\widetilde{v^o}}{\Lnorm{\infty}} +  \hnorm{\xi^o}{\Lnorm{\infty}}\hnorm{\widetilde{v^o_t}}{\Lnorm{2}} + \hnorm{\widetilde{v^o_t}}{\Lnorm{2}} \biggr) \,dz' \times \int_0^1 \biggl( \hnorm{\nablah \dz v^o}{\Lnorm{4}} \hnorm{\dt v}{\Lnorm{4}} \\
		& ~~~~ ~~~~ ~~~~ + \hnorm{\dz v^o}{\Lnorm{\infty}} \hnorm{\nablah v_t}{\Lnorm{2}} \biggr) \,dz \lesssim (\norm{\xi^o_t}{\Lnorm{2}}\norm{v^o}{\Hnorm{2}} + \norm{\xi^o}{\Hnorm{2}}\norm{v^o_t}{\Lnorm{2}}\\
		& ~~~~ ~~~~ + \norm{v^o_t}{\Lnorm{2}})\norm{v^o}{\Hnorm{2}}^{1/2} \norm{v^o}{\Hnorm{3}}^{1/2} \norm{\dt v}{\Hnorm{1}} \lesssim \delta \norm{\dt v}{\Hnorm{1}}^2 + \omega \norm{v^o}{\Hnorm{3}}^2 \\
		& ~~~~ ~~~~ + C_{\delta,\omega} (C_2^2 + M_0^2 + 1 ) C_1^3M_1^3,\\
		& I_8'' \lesssim \int_0^1 \biggl( \hnorm{\nablah \xi^o}{\Lnorm{4}}\hnorm{\widetilde{v^o}}{\Lnorm{\infty}} + \hnorm{\xi^o}{\Lnorm{\infty}}\hnorm{\widetilde{\nablah v^o}}{\Lnorm{4}} + \hnorm{\widetilde{\nablah v^o}}{\Lnorm{4}} \biggr)  \,dz' \\
		& ~~~~ ~~~~ \times \int_0^1 \hnorm{\dz v_t^o}{\Lnorm{2}}\hnorm{\dt v}{\Lnorm{4}} \,dz \lesssim (\norm{\xi^o}{\Hnorm{2}} \norm{v^o}{\Hnorm{2}} + \norm{v^o}{\Hnorm{2}})\norm{v_t^o }{\Hnorm{1}}\\
		& ~~~~ ~~~~ \times \norm{\dt v}{\Lnorm{2}}^{1/2} \norm{\dt v}{\Hnorm{1}}^{1/2} \lesssim \delta \norm{\dt v}{\Hnorm{1}}^2 + \omega \norm{\dt v^o}{\Hnorm{1}}^2 \\
		& ~~~~ ~~~~ + C_{\delta,\omega} ( M_0^2 + 1) C_1^2M_1^2 \norm{\dt v}{\Lnorm{2}}^2.
	\end{align*}
	Summing up the above inequalities with small enough $ \delta $ and $ \omega $ yields
	\begin{equation}\label{Temp-v-g}
		\begin{aligned}
		& \dfrac{d}{dt} \norm{\sqrt{\rho^o}v_t}{\Lnorm{2}}^2 + c_{\mu,\lambda} \norm{v_t}{\Hnorm{1}}^2 \leq \omega \bigl( \norm{v^o}{\Hnorm{3}}^2 + \norm{\dt v^o}{\Hnorm{1}}^2 \bigr) \\
		& ~~~~  + C_\omega \mathcal H (M_0,C_1M_1,C_2) \bigl( \norm{\dt v}{\Lnorm{2}}^2  + 1 \bigr).
		\end{aligned}
	\end{equation}

	{\par\hfill\par\noindent\bf Vertical derivative estimates for $ v $\par}
	Taking the inner product of \subeqref{eq:lin-isen-g}{2} with $ v_t $ and integrating the resultant will yield,
	\begin{align}\label{004-g}
		& \dfrac{1}{2} \dfrac{d}{dt}\bigl( \mu \norm{\nablah v}{\Lnorm{2}}^2 + \mu\norm{\dz v}{\Lnorm{2}}^2 + (\mu+\lambda) \norm{\dvh v}{\Lnorm{2}}^2 \bigr) + \norm{\sqrt{\rho^o}v_t}{\Lnorm{2}}^2 \nonumber \\
		& ~~ =  - \intw ( \rho^o v^o \cdot \nablah v^o ) \cdot \dt v \idx - \intw \rho^o w^o \dz v^o \cdot \dt v \idx \nonumber \\
		& ~~~~ - \intw (2\xi^o + gz) \nablah \xi^o \cdot \dt v\idx =: \sum_{i=10}^{12} I_i.
	\end{align}
	Then the H\"older and the Sobolev embedding inequalities yield
	\begin{align*}
		& I_{10} \lesssim \norm{\rho^o}{\Lnorm{\infty}} \norm{v^o}{\Lnorm{\infty}} \norm{\nablah v^o}{\Lnorm{2}} \norm{\dt v}{\Lnorm{2}} \lesssim \delta \norm{\sqrt{\rho^o} v_t}{\Lnorm{2}}^2 \\
		& ~~~~ ~~~~ ~~~~ + C_{\delta,\underline\rho} (M_0+1)C_1^2M_1^2, \\
		& I_{12} \lesssim \delta \norm{\sqrt{\rho^o}v_t}{\Lnorm{2}}^2 + C_{\delta,\underline\rho} (M_0^2 + 1),
	\end{align*}
	where we have used the fact that $ \rho^o = \xi^o + \frac{1}{2} gz \geq \frac{1}{2} \underline \rho > 0 $.
	As before, plugging in \eqref{vertical-lin-isen-g} yields
	\begin{align*}
		& I_{11} = \int_0^1 \inth \biggl\lbrack  \int_0^z \bigl\lbrack \dvh (\xi^o \widetilde{v^o}) + \dfrac{g}{2} \widetilde{z \dvh v^o} \bigr\rbrack \,dz'  \times \bigl( \dz v^o \cdot \dt v \bigr) \biggr\rbrack  \idxh \,dz\\
		& ~~~~ \lesssim \int_0^1 \biggl( \hnorm{\nablah \xi^o }{\Lnorm{4}} \hnorm{\widetilde{v^o}}{\Lnorm{\infty}} + \hnorm{\xi^o}{\Lnorm{\infty}} \hnorm{\widetilde{\nablah v^o}}{\Lnorm{4}} + \hnorm{\widetilde{\nablah v^o}}{\Lnorm{4}} \biggr) \,dz' \\
		& ~~~~ ~~~~ \times \int_0^1 \hnorm{\dz v^o}{\Lnorm{4}} \hnorm{\dt v}{\Lnorm{2}} \,dz \lesssim (\norm{\xi^o}{\Hnorm{2}} + 1) \norm{v^o}{\Hnorm{2}}^2 \norm{\dt v}{\Lnorm{2}}\\
		& ~~~~ \lesssim \delta \norm{\sqrt{\rho^o}v_t}{\Lnorm{2}}^2 + C_{\delta,\underline\rho} (M_0+1)C_1^2M_1^2.
	\end{align*}
	Therefore, \eqref{004-g} implies
	\begin{equation}\label{Vert-v-g-01}
		\begin{aligned}
			& \dfrac{d}{dt} \bigl( \mu \norm{\nablah v}{\Lnorm{2}}^2 + \mu\norm{\dz v}{\Lnorm{2}}^2 + (\mu+\lambda) \norm{\dvh v}{\Lnorm{2}}^2 \bigr) + \norm{\sqrt{\rho^o}v_t}{\Lnorm{2}}^2 \\
			& ~~~~ ~~~~ \leq \mathcal H(M_0,C_1M_1,C_2,\underline\rho).
		\end{aligned}
	\end{equation}
	\par On the other hand, \subeqref{eq:lin-isen-g}{2} can be written as
	\begin{equation}\label{eq:vertical-v-g-01}
		\begin{aligned}
			& \mu \partial_{zz} v - \rho^o \dt v = - \mu \deltah v - (\mu+\lambda) \nablah \dvh v + \rho^o v^o \cdot \nablah v^o \\
			& ~~~~ ~~~~ + \rho^o w^o \dz v^o + (2\xi^o + gz) \nablah \xi^o.
		\end{aligned}
	\end{equation}
	Taking the inner product of \eqref{eq:vertical-v-g-01} with $ \dt  \partial_{zz}v $ and integrating the resultant will yield, after using the boundary condition \eqref{bd-conds-lin-g} and integrating by parts,
	\begin{align}\label{005-g}
		& \dfrac{1}{2} \dfrac{d}{dt} ( \mu \norm{\nablah \partial_z v}{\Lnorm{2}}^2 + \mu \norm{\partial_{zz} v}{\Lnorm{2}}^2 + (\mu+\lambda)\norm{\dvh \partial_z v}{\Lnorm{2}}^2  )  \nonumber \\
		& ~~~~ + \norm{\sqrt{\rho^o}\partial_z v_t}{\Lnorm{2}}^2 = -  \intw \partial_z \rho^o v_t \cdot \dz v_t \idx - \intw \dz (\rho^o v^o \cdot \nablah v^o) \cdot \dz v_t \idx \nonumber \\
		& ~~~~ ~~~~ - \intw \dz (\rho^o w^o \dz v^o) \cdot \dz v_t \idx - \intw g \nablah \xi^o \cdot \dz v_t \idx =: \sum_{i=13}^{16} I_i.
	\end{align}
	Then we have, noticing that $ \dz \rho^o = \frac{1}{2} g $,
	\begin{align*}
		& I_{13} = - \intw \dfrac{g}{2} v_t \cdot \dz v_t \idx \lesssim \norm{v_t}{\Lnorm{2}} \norm{\dz v_t}{\Lnorm{2}} \lesssim \delta \norm{\sqrt{\rho^o}\dz v_t}{\Lnorm{2}}^2  \\
		& ~~~~ ~~~~ + C_{\delta,\underline\rho} \norm{v_t}{\Lnorm{2}}^2, \\
		&I _{14} \lesssim ( \norm{v^o}{\Lnorm{3}} \norm{\nablah v^o}{\Lnorm{6}} + \norm{\rho^o}{\Lnorm{\infty}} \norm{\dz v^o}{\Lnorm{6}} \norm{\nablah v^o}{\Lnorm{3}} \\
		& ~~~~ ~~~~ + \norm{\rho^o}{\Lnorm{\infty}} \norm{v^o}{\Lnorm{\infty}} \norm{\nablah \dz v^o}{\Lnorm{2}} )\norm{\dz v_t}{\Lnorm{2}}  \lesssim \delta \norm{\sqrt{\rho^o}v_t}{\Lnorm{2}}^2 \\
		& ~~~~ ~~~~ + C_{\delta,\underline\rho} (M_0 + 1) C_1^2M_1^2, \\
		& I_{16} \lesssim \norm{\xi^o}{\Hnorm{1}}\norm{\dz v_t}{\Lnorm{2}} \lesssim \delta \norm{\sqrt{\rho^o}v_t}{\Lnorm{2}}^2 + C_{\delta,\underline\rho} M_0.
	\end{align*}
	Next, we observe that
	\begin{align*}
		& I_{15} = - \intw \dz(\rho^o w^o) \dz v^o \cdot \dz v_t \idx - \intw \rho^o w^o \partial_{zz} v^o \cdot \dz v_t \idx \\
		& ~~ = \intw (\dvh(\xi^o \widetilde{v^o}) + \dfrac{g}{2} \widetilde{z \dvh v^o}) \times (\dz v^o \cdot \dz v_t) \idx \\
		& ~~~~  + \int_0^1 \inth \biggl\lbrack \int_0^z \bigl\lbrack \dvh(\xi^o \widetilde{v^o}) + \dfrac{g}{2} \widetilde{z\dvh v^o} \bigr\rbrack \,dz'  \times \bigl( \partial_{zz} v^o \cdot \dz v_t \bigr) \biggr\rbrack \idxh   \,dz \\
		& ~~~~ =: I_{15}' + I_{15}''.
	\end{align*}
	Then one has
	\begin{align*}
		& I_{15}' \lesssim ( \norm{\nablah \xi^o}{\Lnorm{3}}\norm{\widetilde{v^o}}{\Lnorm{\infty}} + \norm{\xi^o}{\Lnorm{\infty}} \norm{\widetilde{\nablah v^o}}{\Lnorm{3}} + \norm{\widetilde{\nablah v^o}}{\Lnorm{3}}) \\
		& ~~~~ ~~~~ \times \norm{\dz v^o}{\Lnorm{6}} \norm{\dz v_t}{\Lnorm{2}} \lesssim \delta \norm{\sqrt{\rho^o}\dz v_t}{\Lnorm{2}}^2 + C_{\delta,\underline\rho}(M_0 + 1)C_1^2M_1^2,\\
		& I_{15}'' \lesssim \int_0^1 \biggl( \hnorm{\nablah \xi^o}{\Lnorm{4}} \hnorm{\widetilde{v^o}}{\Lnorm{\infty}} + \hnorm{\xi^o}{\Lnorm{\infty}}\hnorm{\widetilde{\nablah v^o}}{\Lnorm{4}} + \hnorm{\widetilde{\nablah v^o}}{\Lnorm{4}} \biggr) \,dz' \\
		& ~~~~ ~~~~ \times  \int_0^1 \hnorm{\partial_{zz}v^o}{\Lnorm{4}} \hnorm{\dz v_t}{\Lnorm{2}} \,dz \lesssim (\norm{\xi^o}{\Hnorm{2}} + 1 )\norm{v^o}{\Hnorm{2}}^{3/2} \norm{v^o}{\Hnorm{3}}^{1/2}\\
		& ~~~~ ~~~~ \times \norm{\dz v_t}{\Lnorm{2}} \lesssim \delta \norm{\sqrt{\rho^o}\dz v_t}{\Lnorm{2}}^2 + \omega \norm{v^o}{\Hnorm{3}}^2 + C_{\delta,\omega,\underline\rho} (M_0^2 + 1) C_1^3M_1^3.
	\end{align*}
	Therefore \eqref{005-g} yields
	\begin{equation}\label{Vert-v-g-02}
		\begin{aligned}
			& \dfrac{d}{dt} ( \mu \norm{\nablah \partial_z v}{\Lnorm{2}}^2 + \mu \norm{\partial_{zz} v}{\Lnorm{2}}^2 + (\mu+\lambda)\norm{\dvh \partial_z v}{\Lnorm{2}}^2  )  \\
			& ~~~~ + c_{\mu,\lambda} \norm{\sqrt{\rho^o}\dz v_t}{\Lnorm{2}}^2 \leq \omega \norm{v^o}{\Hnorm{3}}^2 \\
			& ~~~~ ~~~~ + C_\omega \mathcal H(M_0,C_1M_1,C_2,\underline\rho) (\norm{v_t}{\Lnorm{2}}^2 +1).
		\end{aligned}
	\end{equation}
	\par Next, applying $ \partial \in \lbrace \partial_x, \partial_y, \partial_z \rbrace $ to \eqref{eq:vertical-v-g-01} yields
	\begin{equation}\label{eq:vertical-v-g-03}
		\begin{aligned}
			& \mu \partial \partial_{zz} v - \rho^o \partial v_t = \partial \rho^o \dt v - \mu \deltah \partial v - (\mu + \lambda) \nablah\dvh \partial v \\
			& ~~~~ + \partial (\rho^o v^o \cdot \nablah v^o) + \partial (\rho^o w^o \dz v^o) + \partial ((2\xi^o + gz) \nablah \xi^o) .
		\end{aligned}
	\end{equation}
	This implies
	\begin{align*}
		& \mu \norm{\partial \partial_{zz} v}{\Lnorm{2}} \lesssim \norm{\rho^o}{\Lnorm{\infty}} \norm{\partial v_t}{\Lnorm{2}} + \norm{\partial \rho^o}{\Lnorm{6}} \norm{\dt v}{\Lnorm{3}} \\
		& ~~~~ + \norm{\partial \nablah^2 v}{\Lnorm{2}} + \norm{\partial \rho^o}{\Lnorm{6}} \norm{v^o}{\Lnorm{\infty}} \norm{\nablah v^o}{\Lnorm{3}} \\
		& ~~~~ + \norm{\rho^o}{\Lnorm{\infty}} \norm{\partial v^o}{\Lnorm{6}} \norm{\nablah v^o}{\Lnorm{3}} + \norm{ \rho^o}{\Lnorm{\infty}} \norm{v^o}{\Lnorm{\infty}} \norm{\nablah \partial v^o}{\Lnorm{2}} \\
		& ~~~~ + \norm{\partial (2\xi^o + gz) }{\Lnorm{3}} \norm{\nablah \xi^o}{\Lnorm{6}} + \norm{2\xi^o + gz }{\Lnorm{\infty}} \norm{\nablah \partial \xi^o}{\Lnorm{2}} \\
		& ~~~~ + \norm{\partial(\rho^o w^o) \dz v^o}{\Lnorm{2}} + \norm{\rho^o w^o  \partial \dz v^o}{\Lnorm{2}} \\
		& \lesssim \norm{\nablah^2 v}{\Hnorm{1}} + (M_0^{1/2} + 1) (\norm{v_t}{\Hnorm{1}} + M_0^{1/2} + C_1 M_1 )\\
		& ~~~~   + \norm{\partial(\rho^o w^o) \dz v^o}{\Lnorm{2}} + \norm{\rho^o w^o  \partial \dz v^o}{\Lnorm{2}}.
	\end{align*}
	Also, by employing the Minkowski's inequality, one obtains
	\begin{align*}
		& \norm{\partial_h (\rho^o w^o) \dz v^o}{\Lnorm{2}}^2 \lesssim \int_0^1 \hnorm{\partial_h(\rho^o w^o)}{\Lnorm{2}}^2 \hnorm{\dz v^o}{\Lnorm{\infty}}^2   \,dz \\
		& ~~~~ \lesssim \biggl( \int_0^1 \bigl( \hnorm{\nablah^2 (\xi^o \widetilde{v^o})}{\Lnorm{2}} + \hnorm{\widetilde{\nablah^2 v^o}}{\Lnorm{2}} \bigr) \,dz' \biggr)^2 \times \int_0^1 \hnorm{\dz v^o}{\Hnorm{1}}\hnorm{\dz v^o}{\Hnorm{2}}\,dz \\
		& ~~~~ \lesssim (\norm{\xi^o}{\Hnorm{2}}^2 + 1) \norm{v^o}{\Hnorm{2}}^3 \norm{v^o}{\Hnorm{3}} \lesssim \omega \norm{v^o}{\Hnorm{3}}^2 + C_\omega (M_0^2 + 1) C_1^3 M_1^3, \\
		& \norm{\rho^o w^o  \partial \dz v^o}{\Lnorm{2}}^2 \lesssim \biggl( \int_0^1 \bigl( \hnorm{\nablah (\xi^o \widetilde{v^o})}{\Lnorm{4}} + \hnorm{\widetilde{\nablah v^o} }{\Lnorm{4}} \bigr)  \,dz'\biggr)^2 \times \int_0^1 \hnorm{\partial \partial_z v^o}{\Lnorm{4}}^2 \,dz\\
		& ~~~~ \lesssim (\norm{\xi^o}{\Hnorm{2}}^2 + 1 ) \norm{v^o}{\Hnorm{2}}^3 \norm{v^o}{\Hnorm{3}} \lesssim \omega \norm{v^o}{\Hnorm{3}}^2 + C_\omega (M_0^2 + 1) C_1^3 M_1^3, \\
		& \norm{\dz (\rho^o w^o) \dz v^o}{\Lnorm{2}}^2 \lesssim (\norm{\nablah (\xi^o \widetilde{v^o})}{\Lnorm{3}}^2 + \norm{\widetilde{\nablah v^o} }{\Lnorm{3}}^2) \norm{\dz v^o}{\Lnorm{6}}^2 \\
		& ~~~~ \lesssim ( M_0 + 1) C_1^2M_1^2.
	\end{align*}
	Therefore, we have
	\begin{equation} \label{Vert-v-g-04}
		\begin{aligned}
			& \norm{\partial_{zz} v}{\Hnorm{1}}^2 \leq C\norm{\nablah^2 v}{\Hnorm{1}}^2 + C ( M_0 + 1) \norm{v_t}{\Hnorm{1}}^2 + \omega \norm{v^o}{\Hnorm{3}}^2\\
			& ~~~~ + C_\omega \mathcal H (M_0,C_1M_1,C_2).
		\end{aligned}
	\end{equation}
	
	Now we have the required estimates to prove Proposition \ref{prop:estimates-v-g}.
	\begin{proof}[Proof of Proposition \ref{prop:estimates-v-g}]
		From \eqref{Horiz-v-g}, \eqref{Temp-v-g}, \eqref{Vert-v-g-01} and \eqref{Vert-v-g-02}, there is a constant $ c_{\mu, \lambda, \underline\rho} $ such that
		\begin{equation}\label{energy-ineq-g}
			\begin{aligned}
				& \dfrac{d}{dt} \mathcal E_g(t) + c_{\mu,\lambda,\underline \rho}\bigl( \norm{v}{\Hnorm{1}}^2 + \norm{\nablah v}{\Hnorm{1}}^2 + \norm{\nablah^2 v}{\Hnorm{1}}^2 + \norm{v_t}{\Hnorm{1}}^2 \bigr) \\
				& ~~~~ ~~~~ \leq \omega \norm{v_t}{\Hnorm{1}}^2 + \omega \bigl( \norm{v^o}{\Hnorm{3}}^2 + \norm{v_t^o}{\Hnorm{1}}^2 \bigr)\\
				& ~~~~ ~~~~ ~~~~  + C_\omega \mathcal H (M_0,C_1M_1,C_2,\underline\rho ) ( \mathcal E_g(t) + 1),
			\end{aligned}
		\end{equation}
		where
		\begin{equation}\label{total-energy-g}
			\begin{aligned}
				& \mathcal E_g(t) : = \norm{\sqrt{\rho^o} v}{\Lnorm{2}}^2 + \norm{\sqrt{\rho^o} \nablah v}{\Lnorm{2}}^2 + \norm{\sqrt{\rho^o} \nablah^2  v}{\Lnorm{2}}^2 \\
				& ~~~~ + \norm{\sqrt{\rho^o} v_t}{\Lnorm{2}}^2 + \mu \norm{\nabla v}{\Lnorm{2}}^2 + (\mu+\lambda) \norm{\dvh v}{\Lnorm{2}}^2 \\
				& ~~~~ + \mu \norm{\nabla \partial_{z} v}{\Lnorm{2}}^2 + (\mu+\lambda) \norm{\dvh\partial_{z}v}{\Lnorm{2}}^2.
			\end{aligned}
		\end{equation}
		Notice that for some positive constants $ C_{i,\mu, \lambda, \underline \rho, M_0}, i = 1,2 $, depending on $ \mu, \lambda, \underline \rho, M_0 $, we have
		\begin{equation}
			C_{1,\mu, \lambda, \underline \rho,M_0} (\norm{v}{\Hnorm{2}}^2 + \norm{v_t}{\Lnorm{2}}^2) \leq \mathcal E_g(t) \leq C_{2,\mu, \lambda, \underline \rho,M_0} (\norm{v}{\Hnorm{2}}^2 + \norm{v_t}{\Lnorm{2}}^2).
		\end{equation}
		For $ 0 < \omega \leq \frac{c_{\mu,\lambda,\underline \rho}}{2}$, one infers from \eqref{energy-ineq-g},
		\begin{equation*}
			\dfrac{d}{dt} \mathcal E_g(t) \leq \omega  ( \norm{v^o}{\Hnorm{3}}^2 + \norm{v_t^o}{\Hnorm{1}}^2 )  + C_\omega \mathcal H (M_0,C_1M_1,C_2,\underline\rho ) ( \mathcal E_g(t) + 1 ).
		\end{equation*}
		Therefore, applying the Gr\"onwall's inequality yields
		\begin{align*}
			& \sup_{0\leq t\leq T} \mathcal E_g(t) \leq e^{C_\omega \mathcal H (M_0,C_1M_1,C_2,\underline \rho) T } \biggl( \mathcal E_g(0) + \omega \int_0^T ( \norm{v^o}{\Hnorm{3}}^2 + \norm{v_t^o}{\Hnorm{1}}^2) \,dt \\
			& ~~~~ + \int_0^T C_\omega \mathcal H( M_0,C_1M_1,C_2,\underline\rho)\,dt  \biggr)  \leq e^{C_\omega \mathcal H (M_0,C_1M_1,C_2,\underline \rho) T} \biggl( C_{2,\mu, \lambda, \underline \rho,M_0} M_1 \\
			& ~~~~  + \omega C_1M_1 + C_\omega \mathcal H( M_0,C_1M_1,C_2,\underline\rho) T \biggr), ~~ \text{where $ \omega$ is as above}.
		\end{align*}
		Now, we integrate with respect to the time variable inequality \eqref{energy-ineq-g}. It follows, since $ 0< \omega < \frac{c_{\mu,\lambda,\underline\rho}}{2} $, that
		\begin{align*}
			& \dfrac{c_{\mu,\lambda,\underline\rho}}{2} \int_0^T  \biggl( \norm{v}{\Hnorm{1}}^2 + \norm{\nablah v}{\Hnorm{1}}^2 + \norm{\nablah^2 v}{\Hnorm{1}}^2 + \norm{v_t}{\Hnorm{1}}^2\biggr) \,dt \leq \mathcal E_g(0) + \mathcal E_g(t) \\
			& ~~~~ + \omega \int_0^T \biggl( \norm{v^o}{\Hnorm{3}}^2 + \norm{v_t^o}{\Hnorm{1}}^2 \biggr) \,dt + C_\omega \mathcal H (M_0,C_1M_1,C_2,\underline\rho ) \int_0^T \bigl( \mathcal E_g(t) + 1 \bigr) \,dt \\
			& ~~~~ \leq (2 + TC_\omega \mathcal H (M_0,C_1M_1,C_2,\underline\rho )) e^{C_\omega \mathcal H (M_0,C_1M_1,C_2,\underline \rho) T} \\
			& ~~~~ ~~~~  \times ( C_{2,\mu, \lambda, \underline \rho,M_0} M_1 + \omega C_1M_1 + C_\omega \mathcal H( M_0,C_1M_1,C_2,\underline\rho) T ).
		\end{align*}
		Additionally, from \eqref{Vert-v-g-04}, we have
		\begin{align*}
			& \int_0^T \norm{\partial_{zz}v}{\Hnorm{1}}^2 \,dt \leq C \int_0^T \norm{\nablah^2 v}{\Hnorm{1}}^2 + (M_0+1) \norm{v_t}{\Hnorm{1}}^2 \,dt \\
			& ~~~~ + \omega C_1M_1 + C_\omega \mathcal H(M_0,C_1M_1,C_2) T.
		\end{align*}
	Therefore, we conclude that
	\begin{equation}\label{ttl-en-ineq-final-g-1}
		\begin{aligned}
			& \sup_{0\leq t\leq T} (\norm{v(t)}{\Hnorm{2}}^2 + \norm{v_t(t)}{\Lnorm{2}}^2 ) + \int_0^T \biggl( \norm{v}{\Hnorm{3}}^2 + \norm{v_t}{\Hnorm{1}}^2\biggr) \,dt \\
			&~~~~ \leq C_{1,\mu\lambda,\underline\rho, M_0}^{-1} \sup_{0\leq t\leq T} \mathcal E_g(t)
			 + \int_0^T \biggl( \norm{v}{\Hnorm{1}}^2 + \norm{\nablah v}{\Hnorm{1}}^2 + \norm{\nablah^2 v}{\Hnorm{1}}^2\\
			 & ~~~~ + \norm{v_t}{\Hnorm{1}}^2 \biggr) \,dt
			 + \int_0^T  \norm{\partial_{zz}v}{\Hnorm{1}}^2 \leq (M_0+1) \bigl(C_{3,\mu,\lambda,\underline\rho,M_0} \\
			 & ~~~~ + C_\omega \mathcal H (M_0,C_1M_1,C_2,\underline\rho )T \bigr)
			\times  e^{C_\omega \mathcal H (M_0,C_1M_1,C_2,\underline \rho) T}\bigl( C_{2,\mu, \lambda, \underline \rho,M_0} M_1 \\
			& ~~~~ + \omega C_1M_1
			 + C_\omega \mathcal H( M_0,C_1M_1,C_2,\underline\rho) T \bigr) + \omega C_1M_1  \\
			 & ~~~~ + C_\omega \mathcal H(M_0,C_1M_1,C_2,\underline\rho )T,
		\end{aligned}
	\end{equation}
	for some positive constant $ C_{3,\mu,\lambda,\underline\rho,M_0} $ depending on $ \mu, \lambda, \underline\rho,M_0 $. Now fix  $ \omega = \frac{1}{2} \min\lbrace \frac{c_{\mu,\lambda,\underline\rho}}{2}, \frac{1}{C_1} \rbrace $ and let $ T \in (0,T''] $, where
	$ T'' = T''(M_0,M_1,C_1,C_2,\underline\rho) $ is small enough and satisfying $$ C_\omega \mathcal H(M_0,C_1M_1,C_2,\underline\rho )T'' \leq \min \lbrace 1,M_1\rbrace. $$
	Then \eqref{ttl-en-ineq-final-g-1} yields
	\begin{equation}
			\sup_{0\leq t\leq T} (\norm{v(t)}{\Hnorm{2}}^2 + \norm{v_t(t)}{\Lnorm{2}}^2 ) + \int_0^T \biggl( \norm{v}{\Hnorm{3}}^2 + \norm{v_t}{\Hnorm{1}}^2 \biggr)\,dt 
			\leq C_1M_1,
	\end{equation}
	where $ C_1 $ is given by
	\begin{equation}\label{def:C_1}
		(M_0+1)( C_{3,\mu,\lambda,\underline\rho,M_0}e + e) (C_{2,\mu,\lambda,\underline\rho,M_0} +2) + 2 =: C_1.
	\end{equation}
	This concludes the proof.
	\end{proof}
	
	\subsection{The case without gravity and $\gamma > 1$}\label{sec:lin-isen}
	Consider a finite positive time $ T $, which will be determined later. Let $ \mathfrak Y = \mathfrak Y_T $ be the function space defined by
	\begin{equation}\label{fnc-space-isen-02}
		\begin{aligned}
			\mathfrak Y = \mathfrak Y_T := & \lbrace  (\sigma, v) | \sigma \in L^\infty(0,T;H^2(\Omega)), \dt \sigma \in L^\infty (0,T;H^1(\Omega)), \\
			& v \in L^\infty(0,T; H^2(\Omega)) \cap L^2(0,T;H^3(\Omega)), \\
			& \dt v \in L^\infty(0,T;L^2(\Omega)) L^2(0,T;H^1(\Omega)) \rbrace.
		\end{aligned}
	\end{equation}
	For this space, one can make sense of the initial value for $ \sigma_0 $ and $ v_0 $.
	Notice that, thanks to Aubin compactness theorem (see \cite[Theorem 2.1]{temam1977} and \cite{Chen2012,Simon1986}), every bounded subset of $ \mathfrak Y $ is a compact subset of the space
	\begin{equation}\label{embedd-space}
		\mathfrak V = \mathfrak V_T := \lbrace (\sigma,v)| \sigma , v \in L^\infty(0,T;L^2(\Omega)), \nabla v\in L^2(0,T;L^2(\Omega)) \rbrace.
	\end{equation}
	Let $ \mathfrak X = \mathfrak X_T $ be a bounded subset of $ \mathfrak Y $ defined by
	\begin{equation}\label{fnc-space-isen}
		\begin{aligned}
			\mathfrak X = & \mathfrak X_T := \bigl\lbrace (\sigma, v) \in \mathfrak Y | (\sigma, v)|_{t=0} = (\sigma_0, v_0), \dz v|_{z=0,1} = 0, \dz \sigma = 0, \\
			& \sigma^2 \geq \frac{1}{2} \underline{\rho} > 0, \sup_{0\leq t\leq T} \norm{\sigma(t)}{\Hnorm{2}}^2 \leq 2M_0, \sup_{0\leq t\leq T} \norm{\dt \sigma(t)}{\Hnorm{1}}^2 \leq C_2,\\
			& \sup_{0\leq t\leq T} \lbrace \norm{v(t)}{\Hnorm{2}}^2 + \norm{ v_t(t)}{\Lnorm{2}}^2 \rbrace + \int_0^T \biggl( \norm{v}{\Hnorm{3}}^2 + \norm{\dt v}{\Hnorm{1}}^2 \biggr)\,dt\\
			& ~~~~ \leq C_1M_1
			\bigr\rbrace,
		\end{aligned}
	\end{equation}
	where $\underline\rho$ is the positive lower bound of initial density profile as in \eqref{lower-bound-initial},
	for some positive constants $C_1 = C_1(M_0,\mu,\lambda,\underline\rho)$, $ C_2 = C_2(M_0,C_1M_1)$. Notice, for $ (\sigma, v) \in \mathfrak X $,
	\begin{equation*}
	\int_0^1 \dvh (\sigma^2 \widetilde{v}) \,dz =0.
	\end{equation*}
	Let $ (\sigma^o, v^o) \in \mathfrak X $. The linear system corresponding to \eqref{rfeq:isen-CPE} with respect to $ (\sigma^o, v^o ) $ is given in the system:
	\begin{equation}\label{eq:lin-isen}
	\begin{cases}
	\dt \sigma + \overline{v^o} \cdot \nablah \sigma  + \dfrac{1}{2} \sigma\dvh \overline{v^o}  = 0 & \text{in} ~ \Omega,\\	
	\rho^o \dt v + \rho^o v^o \cdot \nablah v^o + \sigma^o \sigma^o w^o \dz v^o + 2\gamma (\sigma^o)^{2\gamma-1}\nablah \sigma^o \\
	~~~~ ~~~~ = \mu \deltah v + \mu \partial_{zz} v + (\mu + \lambda) \nablah \dvh v & \text{in} ~ \Omega,\\
	\dz \sigma = 0 & \text{in} ~ \Omega.
	\end{cases}
	\end{equation}
	where $ \rho^o = (\sigma^o)^2 $ and $ w^o $ is determined, as in \eqref{vertical-isen-02}, by
	\begin{equation}\label{vertical-lin-isen}
	\sigma^o w^o := - \int_0^z \biggl( \sigma^o  \dvh {\widetilde{v^o}} + 2 \widetilde{v^o} \cdot \nablah \sigma^o \biggr) \,dz.
	\end{equation}
	The initial and boundary conditions for the linear equations \eqref{eq:lin-isen} are given by
	\begin{equation}\label{bd-conds-lin}
	(\sigma, v)|_{t=0} = (\sigma_0, v_0)  = (\rho_0^{1/2}, v_0), ~ \dz v|_{z=0,1} = 0.
	\end{equation}
	Here, in addition to the compatible conditions in \eqref{isen-cptbl-conds}, we require $ \rho_0 \geq \underline \rho > 0 $ , for some positive constant $ \underline \rho $ as in \eqref{lower-bound-initial}. Also, we denote by $ V_1 : = h_1 / \rho_0^{1/2} $. Recall that $ h_1 $ is given in \eqref{isen-cptbl-conds}. Then $ V_1 \in L^2 (\Omega) $ and we require $ \norm{\sigma_0}{\Hnorm{2}}^2 \leq M_0, \norm{v_0}{\Hnorm{2}}^2 + \norm{V_1}{\Lnorm{2}}^2 \leq M_1 $. Essentially $ V_1 = v_t|_{t=0} $.
	
	Then the map $ \mathcal{T} $, in the case without gravity, is defined as
	\begin{equation}\label{def:map-isen}
	\mathcal T:(\sigma^o, v^o) \leadsto (\sigma,v),
	\end{equation}
	where $ (\sigma, v) $ is the unique solution to \eqref{eq:lin-isen} for given $ (\sigma^o, v^o) \in \mathfrak X $.
	
	Then similar arguments to those in sections \ref{sec:solving-linear-g} and \ref{sec:apriori-linear-g} show that $ \mathcal{T} $ is well-defined and maps from $ \mathfrak X $ into $ \mathfrak X $. In particular, one has the following:
	\begin{proposition}\label{prop:maps-set2set}
		There is a $ T_v = T_v(M_0,M_1,\mu,\lambda,\underline\rho) >0 $ sufficiently small such that for every $  T \in (0, T_v] $, there is a unique solution $ ( \sigma, v) $ to \eqref{eq:lin-isen} in the set $ \mathfrak X = \mathfrak X_T $. Therefore, for such $ T $, the map $ \mathcal T $ defined in \eqref{def:map-isen} is a well defined map from $ \mathfrak X $ into $ \mathfrak X $.
	\end{proposition}
	\noindent
	The proof is similar as Proposition \ref{prop:exist-lin-isen-g} and Proposition \ref{prop:maps-set2set-g} in sections \ref{sec:solving-linear-g} and \ref{sec:apriori-linear-g}  and therefore is omitted.
		

	\subsection{Existence theory}\label{sec:exist-thy}
	In this subsection, we will establish the existence theory for \eqref{isen-CPE-g} and \eqref{isen-CPE} for given corresponding initial data and boundary conditions.
	
	In particular, in section \ref{sec:exist-g} we will verify that $ \mathcal T  $ defined in \eqref{def:map-isen-g} is continuous in the topology of $ \mathfrak V $ ($ \supset \mathfrak Y \supset \mathfrak X$) (see \eqref{fnc-space-isen-g-02}, \eqref{embedd-space-g}, \eqref{fnc-space-isen-g}). In particular, it is continuous with respect to the norm given in \eqref{norm-bspace-g}. Then the Schauder-Tchonoff fixed point theorem guarantees that there is a fixed point of $ \mathcal T $ in $ \mathfrak X $. This yields the existence of strong solutions to \eqref{isen-CPE-g}.
	
	In section \ref{sec:exist-vacuum}, we first claim that when the initial density $ \rho_0 \geq \underline \rho > 0$, there exists a strong solution to \eqref{isen-CPE}. Then we will show some a priori estimates for \eqref{isen-CPE} which are independent of $ \underline\rho $. In the end of this subsection, we will show the existence of strong solutions for nonnegative initial density.
	
	\subsubsection{The case when there is gravity but no vacuum and $ \gamma = 2 $}\label{sec:exist-g}
	We will apply the Schauder-Tchonoff fixed point theorem to establish the existence of strong solutions to \eqref{isen-CPE-g}. In fact, as we have already shown $ \mathcal T $ is well-defined and from $ \mathfrak X_T $ to $ \mathfrak X_T $ for $ T $ sufficiently small depending on the initial bounds of data as in \eqref{bound-of-initial-data-linear-g} and $ \underline \rho $, thus it is sufficient to verify that $ \mathcal T $, defined by \eqref{def:map-isen-g}, is continuous in $ \mathfrak V = \mathfrak V_T $ given in \eqref{embedd-space-g}
	where the norm is given by
	\begin{equation*}\tag{\ref{norm-bspace-g}}
		\begin{aligned}
		\norm{(\xi,v)}{\mathfrak V}:= & \norm{\xi}{L^\infty(0,T; L^{2}(\Omega))} + \norm{v}{L^\infty(0,T; L^{2}(\Omega))} \\
		& ~~~~ + \norm{v}{L^2(0,T; H^{1}(\Omega))}.
		\end{aligned}
	\end{equation*}
	Recall that $ \mathfrak X_T \subset \mathfrak Y_T $ is compact in $ \mathfrak V_T $.
	
	In order to show this, let $ M_0 = B_{g,1} , M_1 = B_{g,2}  $ in $ \mathfrak X_T $ and $ T \in (0, T_g] $, with $ T_g $ given in Proposition \ref{prop:maps-set2set-g}. Here $ B_{g,1}, B_{g,2} $ are given in \eqref{bound-of-initial-g}. We denote $ (\xi_1^o,v_1^o),  (\xi_2^o,v_2^o) \in \mathfrak X_T $ and
	\begin{gather*}
		(\xi_1,v_1) = \mathcal T  (\xi_1^o,v_1^o),
		(\xi_2,v_2) = \mathcal T (\xi_1^o,v_1^o).
	\end{gather*}
	Then we have for $ i = 1 ,2  $, $ \rho^o_i = \xi^o_i + \frac{1}{2}gz $ and
	\begin{equation*}
	\begin{cases}
		\dt \xi_i + \overline{v^o_i} \cdot \nablah \xi_i + \xi_i \overline{\dvh v^o_i} + \dfrac{g}{2} \overline{z \dvh v^o_i} = 0, \\
		(\xi^o_i + \dfrac{1}{2}gz) ( \dt v_i + v^o_i \cdot\nablah v^o_i + w^o_i \dz v^o_i) + (2\xi^o_i +gz) \nablah \xi^o_i \\
		~~~~ ~~~~ ~~~~ = \mu \deltah v_i + \mu \partial_{zz} v_i + (\mu +\lambda) \nablah \dvh v_i, \\
		\dz \xi_i = 0, \\
		\rho^o_i w^o_i = (\xi^o_i + \dfrac{1}{2}gz) w^o_i = - \int_0^z \dvh (\xi^o_i \widetilde{v^o_i}) + \dfrac{g}{2} \widetilde{z \dvh v^o_i} \,dz.
	\end{cases}
	\end{equation*}
	Denote by $ \xi_{12} : = \xi_1 - \xi_2, v_{12} := v_1 - v_2, \xi_{12}^o : = \xi_1^o - \xi_2^o, v^o_{12} := v_1^o - v_2^o $. Then $ (\xi_{12}, v_{12})|_{t=0} = 0 $. By taking the differences of the above equations, we have
	\begin{equation}\label{eq:differences-g}
		\begin{cases}
			\dt \xi_{12} + \overline{v_1^o} \cdot \nablah \xi_{12} + \xi_{12} \overline{\dvh v_1^o} + \overline{v_{12}^o} \cdot \nablah \xi_2 + \xi_2 \overline{\dvh v^o_{12}}\\
			~~~~ ~~~~ + \dfrac{g}{2} \overline{z \dvh v_{12}^o} = 0, \\
			\rho_1^o \dt v_{12} - \mu \deltah v_{12} - \mu \partial_{zz} v_{12} - (\mu+\lambda) \nablah \dvh v_{12} = - \xi_{12}^o \dt v_2 \\
			~~~~ ~~~~ - \nablah (\xi_{12}^o(\rho_1^o + \rho_2^o)) - \xi_{12}^o v_1^o \cdot \nablah v_1^o - \rho_2^o v_{12}^o \cdot \nablah v_1^o  \\
			~~~~ ~~~~- \rho_2^o v_2^o \cdot\nablah v_{12}^o - (\rho_1^o w_1^o - \rho_2^o w_2^o) \dz v_1^o - \rho_2^o w_2^o \dz v_{12}^o.
		\end{cases}
	\end{equation}
	Now we perform standard $ L^2 $ estimates for \eqref{eq:differences-g}. Multiply \subeqref{eq:differences-g}{1} with $ 2 \xi_{12} $ and take the $ L^2 $-inner product of \subeqref{eq:differences-g}{2} with $ 2 v_{12} $. Integrating the resultants yields
	\begin{align}
		& \dfrac{d}{dt} \norm{\xi_{12}}{\Lnorm{2}}^2 =  - 2 \inth \biggl( (\overline{v_{12}^o} \cdot \nablah \xi_2) \xi_{12} + \xi_2 \overline{\dvh v_{12}^o} \xi_{12} \biggr) \idxh \nonumber \\
		& ~~~~ -  \inth \overline{\dvh v_1^o} \abs{\xi_{12}}{2} \idxh - g \inth \overline{z\dvh v_{12}^o} \xi_{12} \idxh  =: \sum_{i=1}^3 H_i, \label{weakly-conti-00101} \\
		& \dfrac{d}{dt} \norm{\sqrt{\rho_1^o}v_{12}}{\Lnorm{2}}^2 + 2 \bigl( \mu \norm{\nablah v_{12}}{\Lnorm{2}}^2 + \mu \norm{\partial_z v_{12}}{\Lnorm{2}}^2  \nonumber \\
		& ~~~~ ~~~~ + (\mu + \lambda) \norm{\dvh v_{12}}{\Lnorm{2}}^2 \bigr) = \intw \dt \xi_1^o \abs{v_{12}}{2} \idx - 2 \intw \xi_{12}^o \dt v_2 \cdot v_{12} \idx \nonumber \\
		& ~~~~ + 2 \intw  \xi_{12}^o(\rho_1^o + \rho_2^o) \dvh v_{12} \idx - 2 \intw \biggl( ( \xi_{12}^o v_1^o\cdot\nablah v_1^o ) \cdot v_{12} \nonumber \\
		& ~~~~ ~~~~ + ( \rho_2^o v_{12}^o \cdot \nablah v_1^o ) \cdot v_{12} + ( \rho_2^o v_2^o \cdot \nablah v_{12}^o ) \cdot v_{12} \biggr) \idx \nonumber \\
		& ~~~~ - 2 \intw (\rho_1^o w_1^o - \rho_2^o w_2^o ) \dz v_1^o \cdot v_{12} \idx - 2 \intw \rho_2^o w_2^o \dz v_{12}^o \cdot v_{12} \idx=: \sum_{i=4}^{9} H_i. {\label{weakly-conti-001}}
	\end{align}
	Then as before, we will list the estimates of the right-hand side terms of \eqref{weakly-conti-00101} and \eqref{weakly-conti-001}.
	\begin{align*}
		& H_1 \lesssim \hnorm{\overline{v_{12}^o}}{\Lnorm{4}}\hnorm{\nablah \xi_2}{\Lnorm{4}}\hnorm{\xi_{12}}{\Lnorm{2}} + \hnorm{\xi_2}{\Lnorm{\infty}} \hnorm{\overline{\dvh v_{12}^o}}{\Lnorm{2}} \hnorm{\xi_{12}}{\Lnorm{2}}\\
		& ~~~~ \lesssim \norm{\xi_2}{\Hnorm{2}}^2 \norm{\xi_{12}}{\Lnorm{2}}^2 + \norm{\nablah v_{12}^o}{\Lnorm{2}}^2 + \norm{v_{12}^o}{\Lnorm{2}}^2 . \\
		& H_2 \lesssim \norm{v_1^o}{\Hnorm{3}} \norm{\xi_{12}}{\Lnorm{2}}^2.\\
		& H_3 \lesssim \norm{\nablah v_{12}^o}{\Lnorm{2}}^2 + \norm{\xi_{12}}{\Lnorm{2}}^2.\\
		& H_4 \lesssim \norm{\dt \xi_{1}^o}{\Lnorm{2}}\norm{v_{12}}{\Lnorm{3}} \norm{v_{12}}{\Lnorm{6}} \lesssim \delta \norm{\nabla v_{12}}{\Lnorm{2}}^2 \\
		& ~~~~ ~~~~ +  C_{\delta}(\norm{\dt \xi_{1}^o}{\Lnorm{2}}^4 +1) \norm{v_{12}}{\Lnorm{2}}^2. \\
		& H_5 \lesssim \norm{\xi_{12}^o}{\Lnorm{2}} \norm{\dt v_2}{\Lnorm{6}} \norm{v_{12}}{\Lnorm{3}} \lesssim  \delta \norm{\nabla v_{12}}{\Lnorm{2}}^2 \\
		& ~~~~ ~~~~ +  C_\delta  (\norm{\dt v_2}{\Hnorm{1}}^2\norm{\xi_{12}^o}{\Lnorm{2}}^2 +\norm{v_{12}}{\Lnorm{2}}^2 ). \\
		& H_6 \lesssim \delta \norm{\nablah v_{12}}{\Lnorm{2}}^2 + C_\delta ( \norm{\rho_1^o}{\Hnorm{2}}^2 + \norm{\rho_2^o}{\Hnorm{2}}^2 ) \norm{\xi_{12}^o}{\Lnorm{2}}^2 . \\
		& H_7 \lesssim \norm{\xi_{12}^o}{\Lnorm{2}} \norm{v_1^o}{\Hnorm{2}} \norm{\nablah v_1^o}{\Lnorm{6}} \norm{v_{12}}{\Lnorm{3}} \\
		& ~~~~ ~~~~ + \norm{\rho_2^o}{\Hnorm{2}} \norm{v_{12}^o}{\Lnorm{2}} \norm{\nablah v_1^o}{\Lnorm{6}} \norm{v_{12}}{\Lnorm{3}}\\
		& ~~~~ ~~~~ + \norm{\rho_2^o}{\Hnorm{2}} \norm{v_2^o}{\Hnorm{2}} \norm{\nablah v_{12}^o}{\Lnorm{2}} \norm{v_{12}}{\Lnorm{2}}\\
		& ~~~~\lesssim \delta \norm{\nabla v_{12}}{\Lnorm{2}}^2 + \norm{\nabla v_{12}^o}{\Lnorm{2}}^2 + C_\delta ( \norm{\xi_{12}^o}{\Lnorm{2}}^2 +\norm{v_{12}^o}{\Lnorm{2}}^2\\
		& ~~~~ ~~~~ + (\norm{\rho_2^o}{\Hnorm{2}}^2 \norm{v_2^o}{\Hnorm{2}}^2 + \norm{\rho_2^o}{\Hnorm{2}}^4 \norm{v_1^o}{\Hnorm{2}}^4 + \norm{v_1^o}{\Hnorm{2}}^8 )\norm{v_{12}}{\Lnorm{2}}^2 ).
	\end{align*}
	Moreover, by plugging in the definition of $ \rho_i^o w_i^o $, we have
	\begin{align*}
		& H_8 = 2 \intw \biggl\lbrack \int_0^z \bigl\lbrack \dvh (\xi_{12}^o \widetilde{v_1^o}) + \dvh( \xi_2^o \widetilde{v_{12}^o}) + \dfrac{g}{2} \widetilde{z \dvh v_{12}^o} \bigr\rbrack \,dz' \\
		& ~~~~ \times \bigl( \dz v_{1}^o \cdot v_{12} \bigr) \biggr\rbrack \idx = -2 \intw  \biggl\lbrack \int_0^z  \bigl\lbrack \xi_{12}^o \widetilde{v_1^o} + \xi_2^o \widetilde{v_{12}^o} + \dfrac{g}{2} \widetilde{z v_{12}^o} \bigr\rbrack \,dz' \\
		& ~~~~ \cdot \nablah ( \dz v_1^o \cdot v_{12} ) \biggr\rbrack \idx \lesssim \int_0^1 \biggl( \hnorm{\xi_{12}^o}{\Lnorm{2}} \hnorm{\widetilde{v_1^o}}{\Lnorm{\infty}} + \hnorm{\xi_2^o}{\Lnorm{\infty}}\hnorm{\widetilde{v_{12}^o}}{\Lnorm{2}} + \hnorm{\widetilde{v_{12}^o}}{\Lnorm{2}} \biggr)  \,dz' \\
		& ~~~~ \times \int_0^1 \biggl( \hnorm{\nablah \dz v_1^o}{\Lnorm{4}} \hnorm{v_{12}}{\Lnorm{4}} + \hnorm{\dz v_1^o}{\Lnorm{\infty}} \hnorm{\nablah v_{12}}{\Lnorm{2}} \biggr) \,dz \lesssim \delta \norm{\nabla v_{12}}{\Lnorm{2}}^2 \\
		& ~~~~ + C_\delta \norm{v_1^o}{\Hnorm{3}}^2 (\norm{\xi_{12}^o}{\Lnorm{2}}^2 + \norm{v_{12}^o}{\Lnorm{2}}^2 ) + C_\delta \norm{v_1^o}{\Hnorm{2}}^2 ( \norm{v_1^o}{\Hnorm{2}}^4 + \norm{\xi_2^o}{\Hnorm{2}}^4 \\
		& ~~~~ ~~~~ + 1 ) (\norm{\xi_{12}^o}{\Lnorm{2}}^2 + \norm{v_{12}^o}{\Lnorm{2}}^2) + \norm{v_{12}}{\Lnorm{2}}^2, \\
		& H_9 = 2 \intw \biggl\lbrack \int_0^z \bigl( \dvh(\xi_2^o\widetilde{v_2^o}) + \dfrac{g}{2} \widetilde{v\dvh v_2^o}\bigr) \,dz'  \times ( \dz v_{12}^o \cdot v_{12}) \biggr\rbrack \idx \\
		& ~~~~ \lesssim \int_0^1 \biggl( \hnorm{\xi_2^o}{\Lnorm{\infty}} \hnorm{\widetilde{\nablah v_2^o}}{\Lnorm{4}}
		 + \hnorm{\nablah \xi_2^o}{\Lnorm{4}} \hnorm{\widetilde{v_2^o}}{\Lnorm{\infty}} + \hnorm{\widetilde{v_2^o}}{\Lnorm{4}} \biggr) \,dz' \\
		 & ~~~~ ~~~~ \times \int_0^1 \hnorm{\dz v_{12}^o}{\Lnorm{2}} \hnorm{v_{12}}{\Lnorm{4}} \,dz \lesssim \delta \norm{\nabla v_{12}}{\Lnorm{2}}^2
		 + C_\delta \norm{\nabla v_{12}^o}{\Lnorm{2}}^2 \\
		 & ~~~~ ~~~~ + C_\delta (\norm{\xi_2^o}{\Hnorm{2}}^4 + 1) (\norm{v_{2}^o}{\Hnorm{2}}^4 + 1) \norm{v_{12}}{\Lnorm{2}}^2.
	\end{align*}
	Here we have applied inequality \eqref{ineq-supnorm}.
	After summing up the inequalities above, with small enough $ \delta $, \eqref{weakly-conti-00101} and \eqref{weakly-conti-001} yield
	\begin{equation}{\label{weakly-conti-003}}
	\begin{aligned}
		& \dfrac{d}{dt} \norm{\xi_{12}}{\Lnorm{2}}^2 \leq ( \mathcal H(M_0,C_1M_1) + \norm{v_1^o}{\Hnorm{3}}^2 + 1 ) \norm{\xi_{12}}{\Lnorm{2}}^2 \\
		& ~~~~ ~~~~+  C \norm{v_{12}^o}{\Lnorm{2}}^2 + C \norm{\nabla v_{12}^o}{\Lnorm{2}}^2,\\
		& \dfrac{d}{dt} \norm{\sqrt{\rho_1^o}v_{12}}{\Lnorm{2}}^2 + c_{\mu,\lambda} \norm{\nabla v_{12}}{\Lnorm{2}}^2 \leq \underline{\rho}^{-1}\mathcal H(M_0,C_1M_1,C_2)  \norm{\sqrt{\rho_1^o}v_{12}}{\Lnorm{2}}^2 \\
		& ~~~~ ~~~~ + C \norm{\nabla v_{12}^o}{\Lnorm{2}}^2 + C ( \mathcal H(M_0,C_1M_1,C_2) + \norm{\dt v_2}{\Hnorm{1}}^2 + \norm{v_1^o}{\Hnorm{3}}^2) \\
		&~~~~ ~~~~ ~~~~ \times (\norm{\xi_{12}^o}{\Lnorm{2}}^2 + \norm{v_{12}^o}{\Lnorm{2}}^2).
	\end{aligned}
	\end{equation}
	Then the Gr\"onwall's inequality yields
	\begin{equation}\label{weakly-conti-002}
		\begin{aligned}
		& \sup_{0\leq t\leq T} (\norm{\xi_{12}(t)}{\Lnorm{2}}^2 + \norm{v_{12}(t)}{\Lnorm{2}}^2 ) + \int_0^T \norm{ \nabla v_{12}(t)}{\Lnorm{2}}^2 \,dt \\
		& ~~~~ \leq C_{M_0,C_1M_1,C_2,\underline\rho}\bigl( \sup_{0<t<T} (\norm{\xi_{12}^o}{\Lnorm{2}}^2 + \norm{v_{12}^o}{\Lnorm{2}}^2 ) \\
		& ~~~~ ~~~~ + \int_0^T \norm{ \nabla v_{12}^o}{\Lnorm{2}}^2 \,dt\bigr).
		\end{aligned}
	\end{equation}
	This yields the continuity of $ \mathcal T $ in $\mathfrak V$. Therefore, after applying the fixed point theorem mentioned before, we have the following:
	\begin{proposition}\label{prop:exist-isen-g}
		Consider $$ (\rho_0,v_0) = (\xi_0 + \dfrac{1}{2}gz, v_0), $$ given in  \eqref{isen-initial-g} satisfying \eqref{isen-cptbl-conds-g} and \eqref{bound-of-initial-g}. There is a positive constant $ T $ depending on the initial data such that there is a strong solution $ (\rho, v) = (\xi + \frac{1}{2}gz, v) $ to \eqref{isen-CPE-g} (or equivalently \eqref{rfeq:isen-CPE-g}) with the boundary conditions \eqref{bd-cnds} and with $ (\xi, v) \in \mathfrak X_T $.
	\end{proposition}

	\subsubsection{The case when there is vacuum but no gravity and $ \gamma > 1 $}\label{sec:exist-vacuum}
	
	When $ \rho_0 \geq \underline{\rho} > 0 $, the existence of strong solutions to \eqref{isen-CPE} follows from the estimates in section \ref{sec:lin-isen} and  similar arguments to those in section \ref{sec:exist-g}. In fact, taking $ M_0 = B_1, M_1 = B_2 + \underline\rho^{-1} B_2 $,  we have the following:
	\begin{proposition}\label{prop:exist-isen-novacuum}
	Suppose that \eqref{bound-of-initial-energy}, \eqref{isen-cptbl-conds}, \eqref{bound-of-initial} hold for the given initial data \eqref{isen-initial} with $ \rho_0 \geq
		\underline \rho  > 0 $. Then there is a positive constant $ T $, depending on the initial data and $ \underline\rho $, such that there exists a strong solution $(\rho, v) = (\sigma^2, v) $ to \eqref{isen-CPE} (or equivalently \eqref{rfeq:isen-CPE}) satisfying the boundary conditions \eqref{bd-cnds} and that $ (\sigma,v) \in \mathfrak X_T $.
	\end{proposition}
	
	In the following, we shall present some estimates independent of $\underline{\rho} $ and show that for a given non-negative initial density $ \rho_0 \geq 0 $, there are strong solutions to equations \eqref{isen-CPE}. We will use here the notation $ \sigma^2 = \rho $ and the alternative form of equations \eqref{rfeq:isen-CPE}, as well as \eqref{isen-CPE}. Meanwhile, let us assume that
	\begin{equation}\label{initial-bound-001}
	\begin{gathered}
		\norm{\sigma_0}{\Hnorm{2}} = \norm{\rho_0^{1/2}}{\Hnorm{2}} \leq K_1,\\
		\norm{v_0}{\Hnorm{2}} + \norm{h_1}{\Lnorm{2}} \leq K_2,
	\end{gathered}
	\end{equation}
	for given $ K_1, K_2 > 0 $.
	Recall essentially $ h_1 = (\sigma v_t)|_{t=0} $ from \eqref{isen-cptbl-conds}. Also, taking inner product of \subeqref{isen-CPE}{2} with $ v $ yields, after integrating the resultant in the temporal variable, the following conservation of physical energy,
	\begin{equation}
	\begin{aligned}\label{conservation-energy}
		& \dfrac{1}{2}\intw \rho \abs{v}{2} \idx + \dfrac{1}{\gamma-1} \intw \rho^\gamma \idx + \int_0^T \intw \biggl( \mu \abs{\nabla v}{2} + (\mu+\lambda) \abs{\dvh v}{2} \biggr) \idx \\
		& ~~~~ = \dfrac{1}{2}\intw \rho_0 \abs{v_0}{2} \idx + \dfrac{1}{\gamma-1} \intw \rho_0^\gamma \idx < \infty,
	\end{aligned}	
	\end{equation}
	where \subeqref{isen-CPE}{1} is also applied. Also, integrating \subeqref{isen-CPE}{1} in $ \Omega \times (0,T) $ yields the conservation of total mass:
	\begin{equation}\label{conservation-mass}
		0 < \intw \rho \idx = \intw \rho_0 \idx = M < \infty.
	\end{equation}
	These facts are important when applying \eqref{ineq:embedding-weighted} in the following.
	
	{\par\hfill\par\noindent\bf A priori assumptions \par }
	Let $ (\sigma, v) $ be the solution to \eqref{rfeq:isen-CPE} given in Proposition \ref{prop:exist-isen-novacuum}.
	We assume first, for some constants $ C_d \geq K_2^2 $, $ T_d $ (may depend on $ \underline{\rho} $),
	\begin{equation}\label{priori-a-isen}
		\sup_{0\leq t\leq T_d} \bigl( \norm{v(t)}{\Hnorm{2}}^2 + \norm{\sigma v_t(t)}{\Lnorm{2}}^2 \bigr) + \int_0^{T_d}\biggl( \norm{v}{\Hnorm{3}}^2 + \norm{v_t}{\Hnorm{1}}^2 \biggr) \,dt < C_d.
	\end{equation}
	In the following, we will derive some a prior estimates independent of $ \underline{\rho} $. Also, we set $ T \in (0, T_d] $ 
	 to be determined later. We emphasize that the smallness of $ T $ in the following is independent of $ \underline\rho $.
	For the sake of simplicity, we will assume the solution $ (\sigma, v) $ to \eqref{rfeq:isen-CPE} is smooth enough that the manipulations below are allowed. To make the arguments rigorous, one has to perform parallel estimates on the solutions to the linear system \eqref{eq:lin-isen} with, as mentioned before in section \ref{sec:apriori-linear-g}, smooth enough initial data $ (\sigma_0, v_0) $ and inputs $ \sigma^o, v^o $. Then an approximating argument will yield the desired estimates.
	{\par\hfill\par\noindent\bf $ \underline{\rho} $-independent lower bound: non-negativity of $ \rho $  \par}
	We will use the same Stampaccia-like argument as before to derive the lower bound of $ \rho $.  Consider $$ \eta = \eta(x,y,t) := \dfrac{\rho}{\inf_{\vech{x} \in \Omega_h} \rho_0(\vech{x})} - 1  + \int_0^t 2 \hnorm{\dvh \overline v(s)}{\Lnorm{\infty}} \,ds. $$
	Then $ \eta $ satisfies the equation, due to \eqref{conti-isen},
	\begin{align*}
		& \dt \eta + \overline v \cdot \nablah \eta + \eta \dvh \overline{v} =   \bigl( \int_0^t 2 \hnorm{\dvh \overline v(s)}{\Lnorm{\infty}}\,ds - 1 \bigr) \times  \dvh \overline{v} \\
		& ~~~~ + 2 \hnorm{\dvh \overline v(s)}{\Lnorm{\infty}}
		 \geq - 2 \abs{\dvh \overline v}{} + 2 \hnorm{\dvh \overline v}{\Lnorm{\infty}} \geq 0,
	\end{align*}
	for every $ t \in [0,T] $ with $ T \in (0,T_1] $ and $ T_1 $ sufficiently small such that
	\begin{align*}
		& 2 \int_0^t \hnorm{\dvh \overline{ v}(s)}{\Lnorm\infty}\,ds \leq 2 C \int_0^t \norm{v(s)}{\Hnorm{3}}\,ds \\
		& ~~~~ \leq 2 C T^{1/2} \bigl(\int_0^t \norm{v(s)}{\Hnorm{3}}^2 \,ds \bigr)^{1/2} \leq 2 C C_d^{1/2} T_1^{1/2} \leq \dfrac{1}{2}.
	\end{align*}
	
	Denote by $ \eta_-: = - \eta \mathbbm 1_{\lbrace \eta < 0 \rbrace} \geq 0 $. Then multiplying the above equation with $ - \mathbbm 1_{\lbrace \eta < 0 \rbrace} $ and integrating the resultant in the spatial variable yield
	\begin{equation*}
		\dfrac{d}{dt} \inth \eta_- \idxh \leq 0.
	\end{equation*}
	Hence, $ \eta_- = 0 $ in $ \Omega_h \times (0,T] $, since $ \eta_-(0) \equiv 0 $. Therefore, $ \eta \geq 0 $ and
	\begin{equation}\label{lower-bound-conti}
		\begin{aligned}
			& \rho = \inf_{\vech{x}\in \Omega_h} \rho_0 (\vech{x}) \times \bigl( \eta + 1 - \int_0^t 2 \hnorm{\dvh \overline{v}(s)}{\Lnorm{\infty}} \,ds \bigr) \\
			& ~~~~ \geq \inf_{\vech{x}\in \Omega_h} \rho_0(\vech{x}) \times \bigl( 0 + 1 - \dfrac{1}{2} \bigr) = \dfrac{1}{2}  \inf_{\vech{x}\in \Omega_h} \rho_0(\vech x).
		\end{aligned}
	\end{equation}

	{\par\hfill\par\noindent\bf $ \underline{\rho} $-independent estimate: $ H^2(\Omega) $ for $ \sigma = \rho^{1/2} $  \par}
	After applying $ \partial_{hh} $ to \eqref{conti-isen-02}, one has
	\begin{equation}\label{eq:horizontal-conti}
	\begin{aligned}
		& \dt \partial_{hh} \sigma + \overline{v} \cdot \nablah \partial_{hh}\sigma + 2 \overline{\partial_h v} \cdot \nablah \partial_{h} \sigma + \dfrac{1}{2} \partial_{hh}\sigma \overline{\dvh v} + \overline{\partial_{hh} v} \cdot \nablah \sigma \\
		& ~~~~ + \partial_h \sigma \overline{\dvh \partial_h v} + \dfrac{1}{2} \sigma \overline{\dvh \partial_{hh} v} = 0.
	\end{aligned}
	\end{equation}
	Then after performing standard $ L^2 $ estimate of \eqref{eq:horizontal-conti} and similar estimates for lower order derivatives, one has
	\begin{equation*}
	\dfrac{d}{dt} \norm{\sigma}{\Hnorm{2}}^2 \leq C \norm{v}{\Hnorm{3}} \norm{\sigma}{\Hnorm{2}}^2.
	\end{equation*}
	Then the Gr\"onwall's inequality yields
	\begin{equation}\label{H2-conti}
	\sup_{0\leq t\leq T} \norm{\sigma(t)}{\Hnorm{2}}^2 \leq e^{C \int_0^T \norm{v}{\Hnorm{3}} \,dt } \norm{\sigma_0}{\Hnorm{2}}^2 \leq e^{CC_d^{1/2}T^{1/2}} K_1^2 < 2 K_1^2,
	\end{equation}
	for all $ T \in (0,T_2] $, provided
	$ T_2 $ is sufficiently small.

	{\par\hfill\par\noindent\bf $ \underline{\rho} $-independent estimate: $ H^1(\Omega) $ for $ \sigma_t $  \par}
	Applying $ \partial_h $ to \eqref{conti-isen-02} yields
	\begin{equation*}
		\dt \partial_h \sigma = - \overline{v} \cdot \nablah \partial_h \sigma - \overline{\partial_h v} \cdot \nablah \sigma - \dfrac{1}{2} \partial_h \sigma \overline{\dvh v} - \dfrac{1}{2} \sigma\overline{\dvh \partial_h v}.
	\end{equation*}
	Therefore, one has
	\begin{equation*}
		\norm{\dt \partial_h\sigma}{\Lnorm{2}} \lesssim \norm{v}{\Hnorm{2}}\norm{\sigma}{\Hnorm{2}}.
	\end{equation*}
	Similar estimate also holds for $ \norm{\dt\sigma}{\Lnorm{2}} $. Hence, we have together with \eqref{H2-conti},
	\begin{equation}\label{H1-t-conti}
		\norm{\dt \sigma(t)}{\Hnorm{1}} \leq C \norm{v(t)}{\Hnorm{2}}\norm{\sigma(t)}{\Hnorm{2}} \leq \sqrt{2} C C_d^{1/2} K_1 =: K_3',
	\end{equation}
	for all $ t \in [0,T] $.

	{\par\hfill\par\noindent\bf $ \underline{\rho} $-independent estimate: $ L^2(\Omega) $ for $ v_t $  \par}
	Taking the time derivative of \subeqref{rfeq:isen-CPE}{2} yields
	\begin{equation}\label{eq:temporal-v}
		\begin{aligned}
		& \sigma^2 \dt v_t - \mu \deltah v_t - \mu \partial_{zz} v_t - (\mu+\lambda) \nablah \dvh v_t =- 2 \sigma \dt \sigma \dt v \\
		& ~~~~ ~~~~ - \dt(\sigma^2 v \cdot\nablah v) - \dt (\sigma^2 w \dz v) - \dt \nablah \sigma^{2\gamma}.
		\end{aligned}
	\end{equation}
	Taking the $L^2$-inner product of \eqref{eq:temporal-v} with $ \dt v $ gives
	\begin{align}\label{001}
		& \dfrac{1}{2} \dfrac{d}{dt} \norm{\sigma v_t}{\Lnorm{2}}^2 + \mu \norm{\nablah v_t}{\Lnorm{2}}^2 + \mu \norm{\dz v_t}{\Lnorm{2}}^2+ (\mu+\lambda) \norm{\dvh v_t}{\Lnorm{2}}^2 \nonumber \\
		& ~~~~ = - \intw  \sigma \dt \sigma \abs{v_t}{2}\idx -  \intw \dt(\sigma^2 v \cdot\nablah v) \cdot v_t \idx \nonumber \\
		& ~~~~ ~~~~ - \intw \dt (\sigma^2 w \dz v) \cdot v_t \idx + \intw \dt \sigma^{2\gamma} \dvh v_t \idx =: \sum_{i=1}^{4} L_i.
	\end{align}
	Then one has the following estimates to the terms in the right-hand side of \eqref{001}.
	\begin{align*}
	& L_1 \lesssim \norm{\dt \sigma}{\Lnorm{2}} \norm{\sigma v_t}{\Lnorm{3}} \norm{v_t}{\Lnorm{6}} \lesssim \norm{\dt \sigma}{\Lnorm{2}}\norm{\sigma v_t}{\Lnorm{2}}^{1/2} \norm{\sigma}{\Lnorm{\infty}}^{1/2}\norm{v_t}{\Lnorm{6}}^{3/2} \\
	& ~~~~  \lesssim \delta \norm{\nabla v_t}{\Lnorm{2}}^2 + C_\delta C_d (K_1^2 K_3'^4 + 1) . \\
	& L_2 \lesssim \norm{\dt \sigma}{\Lnorm{2}} \norm{v}{\Lnorm{\infty}} \norm{\nablah v}{\Lnorm{6}} \norm{\sigma v_t}{\Lnorm{3}} + \norm{\sigma v_t}{\Lnorm{3}} \norm{\nablah v}{\Lnorm{6}} \norm{\sigma v_t}{\Lnorm{2}}\\
	& ~~~~ ~~~~ + \norm{\sigma}{\Lnorm{\infty}}\norm{v}{\Lnorm{\infty}} \norm{\nablah v_t}{\Lnorm{2}} \norm{\sigma v_t}{\Lnorm{2}} \lesssim \delta \norm{\nabla v_t}{\Lnorm{2}}^2 \\
	& ~~~~ ~~~~ + C_\delta ( C_d + 1 )( K_1^2 C_d + K_1 K_3' C_d^2 + 1) . \\
	& L_4 \lesssim \norm{\sigma}{\Lnorm{\infty}}^{2\gamma-1} \norm{\dt \sigma}{\Lnorm{2}}\norm{\nabla v_t}{\Lnorm{2}} \lesssim \delta \norm{\nabla v_t}{\Lnorm{2}}^2 + C_\delta K_1^{4\gamma-2}K_3'^2.
	\end{align*}
	We have applied above the H\"older inequality and \eqref{ineq:embedding-weighted}, i.e.,
	\begin{gather*}
	\norm{\sigma v_t}{\Lnorm{3}} \lesssim \norm{\sigma v_t}{\Lnorm{2}}^{1/2} \norm{\sigma v_t}{\Lnorm{6}}^{1/2} \lesssim \norm{\sigma}{\Lnorm{\infty}}^{1/2} \norm{\sigma v_t}{\Lnorm{2}}^{1/2} \norm{v_t}{\Lnorm{6}}^{1/2}, \\
	\norm{v_t}{\Lnorm{6}} \leq C \norm{\nabla v_t}{\Lnorm{2}} + C \norm{\sigma v_t}{\Lnorm{2}},
	\end{gather*}
	for some constant $ C = C( \intw \rho\idx , \intw \rho^\gamma \idx) $.
	Notice that $ \sigma = \rho^{1/2} $ and that the conservations of energy and mass \eqref{conservation-energy}, \eqref{conservation-mass} hold. Therefore, the constant $ C $ in the above inequality depends only on the initial energy and total mass.
	In order to estimate $ L_3 $ term, we substitute \eqref{vertical-isen} and integrate by parts. Then
	\begin{align*}
	& L_3 = - \intw \dt (\rho w) \dz v \cdot v_t \idx - \intw \rho w \dz v_t \cdot v_t \idx \\
	& ~~~~ = - \int_0^1 \inth \biggl\lbrack \biggl( \int_0^z (\sigma^2 \widetilde{v})_t \,dz' \biggr) \cdot \nablah (\dz v \cdot v_t) \biggr\rbrack \idxh \,dz \\
	& ~~~~ ~~~~ + \int_0^1 \inth \biggl\lbrack \biggl( \int_0^z \dvh (\sigma^2 \widetilde{v}) \,dz'  \biggr) \times \bigl( \dz v_t \cdot v_t \bigr) \biggr\rbrack \ \idxh \,dz =: L_3' + L_3''.
	\end{align*}
	Now we use \eqref{ineq-supnorm}, the Minkowski's and the H\"older inequalities,
	\begin{align*}
	& L_3' = - \int_0^1 \inth \biggl\lbrack \biggl( \int_0^z \bigl( \sigma \widetilde{v_t} + 2 \sigma_t \widetilde{v} \bigr) \,dz' \biggr) \cdot (\nablah \dz v \cdot \sigma v_t + \sigma \nablah v_t \cdot \dz v ) \biggr\rbrack \idxh \,dz \\
	& ~~~~ \lesssim \int_0^1 \bigl( \hnorm{\sigma \widetilde{v_t}}{\Lnorm{2}} + \hnorm{\sigma_t}{\Lnorm{2}}\hnorm{\widetilde v}{\Lnorm{\infty}}\biggr) \,dz' \times \int_0^1 \biggl( \hnorm{\sigma}{\Lnorm{\infty}} \hnorm{\nablah \dz v}{\Lnorm{4}} \hnorm{v_t}{\Lnorm{4}} \\
	& ~~~~ ~~~~ + \hnorm{\sigma}{\Lnorm{\infty}} \hnorm{\nablah v_t}{\Lnorm{2}} \hnorm{\dz v}{\Lnorm{\infty}} \biggr) \,dz \lesssim \int_0^1 \biggl( \hnorm{\sigma \widetilde{v_t}}{\Lnorm{2}} + \hnorm{\sigma_t}{\Lnorm{2}} \hnorm{\widetilde v}{\Hnorm{2}} \biggr) \,dz' \\
	& ~~~~ \times \int_0^1 \hnorm{\sigma}{\Hnorm{2}} \hnorm{\dz v}{\Hnorm{1}}^{1/2} \hnorm{\dz v}{\Hnorm{2}}^{1/2} \hnorm{v_t}{\Hnorm{1}} \,dz \lesssim \norm{\sigma}{\Hnorm{2}}(\norm{\sigma v_t}{\Lnorm{2}}\\
	& ~~~~ ~~~~ + \norm{\sigma_t}{\Lnorm{2}} \norm{v}{\Hnorm{2}})  \norm{v}{\Hnorm{2}}^{1/2} \norm{v}{\Hnorm{3}}^{1/2} ( \norm{\sigma v_t}{\Lnorm{2}} + \norm{\nabla v_t}{\Lnorm{2}} )\\
	& ~~~~ \lesssim \delta \norm{\nabla v_t}{\Lnorm{2}}^2 + \omega \norm{v}{\Hnorm{3}}^2 + C_{\delta,\omega} C_d ( K_1^4(K_3'^4 + 1) C_d^2 + 1),\\
	& L_3'' \lesssim \int_0^1 \biggl( \hnorm{\nablah \sigma}{\Lnorm{4}} \hnorm{\widetilde v}{\Lnorm{\infty}} + \hnorm{\sigma}{\Lnorm{\infty}} \hnorm{\widetilde{\nablah v}}{\Lnorm{4}} \biggr) \,dz' \times \int_0^1 \hnorm{\dz v_t}{\Lnorm{2}} \hnorm{\sigma v_t}{\Lnorm{4}} \,dz \\
	& ~~~~ \lesssim \norm{\sigma}{\Hnorm{2}}^{5/3} \norm{v}{\Hnorm{2}} \norm{\nabla v_t}{\Lnorm{2}}\norm{\sigma v_t}{\Lnorm{2}}^{1/3}(\norm{\sigma v_t}{\Lnorm{2}}^{2/3} + \norm{\nabla v_t}{\Lnorm{2}}^{2/3})\\
	& ~~~~ \lesssim \delta \norm{\nabla v_t}{\Lnorm{2}}^2 + C_\delta (K_1^{10} C_d^{4} + K_1^{10/3} C_d^2).
	\end{align*}
	Here we have applied the facts that $ \hnorm{\sigma}{\Hnorm{2}} = \norm{\sigma}{\Hnorm{2}} $ and that
	\begin{gather*}
		\norm{v_t}{\Hnorm{1}} \lesssim \norm{\sigma v_t}{\Lnorm{2}} + \norm{\nabla v_t}{\Lnorm{2}},\\
		\hnorm{\sigma v_t}{\Lnorm{4}} \lesssim \hnorm{\sigma v_t}{\Lnorm{2}}^{1/3} \hnorm{\sigma v_t}{\Lnorm{8}}^{2/3} \lesssim \hnorm{\sigma}{\Lnorm{\infty}}^{2/3} \hnorm{\sigma v_t}{\Lnorm{2}}^{1/3} \hnorm{v_t}{\Hnorm{1}}^{2/3},
	\end{gather*}
	where the first inequality results from \eqref{ineq:embedding-weighted}.
	After summing the above inequalities, \eqref{001} then implies
	\begin{equation}\label{Temp-v}
		\dfrac{d}{dt} \norm{\sigma v_t}{\Lnorm{2}}^2 + c_{\mu,\lambda} \norm{\nabla v_t}{\Lnorm{2}}^2 \leq \omega \norm{v}{\Hnorm{3}}^2 + C_\omega \mathcal H(K_1,K_3',C_d),
	\end{equation}
	where as before, $ \mathcal H $ denotes a polynomial quantity of its arguments.
	
	{\par\hfill\par\noindent\bf $ \underline{\rho} $-independent estimate: spatial derivatives of $ v $  \par}
	Now we are able to derive the estimates on the spatial derivatives of $ v $. Standard $ L^2 $ estimate of \subeqref{rfeq:isen-CPE}{2} yields the following
	\begin{equation}\label{L2-v}
		\begin{aligned}
			& \dfrac{d}{dt} \norm{\sigma v}{\Lnorm{2}}^2 + c_{\mu,\lambda} \norm{\nabla v}{\Lnorm{2}}^2 \leq C \norm{\sigma^{2\gamma}}{\Lnorm{2}}^2 \leq C \mathcal H(K_1).
		\end{aligned}
	\end{equation}
	Furthermore, taking the $ L^2 $-inner product of \subeqref{rfeq:isen-CPE}{2} with $ v_t $ yields
	\begin{align}\label{002}
		& \dfrac{1}{2} \dfrac{d}{dt} (\mu\norm{\nablah v}{\Lnorm{2}}^2 + \mu \norm{\dz v}{\Lnorm{2}}^2 + (\mu+\lambda) \norm{\dvh v}{\Lnorm{2}}^2 ) + \norm{\sigma v_t}{\Lnorm{2}}^2 \nonumber \\
		& ~~~~ = - \intw ( \sigma v \cdot \nablah v ) \cdot \sigma v_t \idx - \intw \sigma w \dz v \cdot \sigma v_t \idx \nonumber \\
		& ~~~~ ~~~~  - 2\gamma \intw \sigma^{2\gamma-2} \nablah \sigma \cdot \sigma v_t \idx =: \sum_{i=5}^{7} L_{i}.
	\end{align}
	As before, one has
	\begin{align*}
		& L_5 \lesssim \norm{\sigma}{\Hnorm{2}} \norm{v}{\Hnorm{2}} \norm{v}{\Hnorm{1}} \norm{\sigma v_t}{\Lnorm{2}} \lesssim \delta \norm{\sigma v_t}{\Lnorm{2}}^2 + C_\delta K_1^2 C_d^2, \\
		& L_7 \lesssim \delta \norm{\sigma v_t}{\Lnorm{2}}^2 + C_\delta K_1^{4\gamma-2}.
	\end{align*}
	After plugging in \eqref{vertical-isen-02},
	\begin{equation}\label{Impt-H2-sigma}
	\begin{aligned}
		& L_6 = \int_0^1 \inth \biggl\lbrack \int_0^z \bigl( \sigma  \widetilde{\dvh v} + 2 \widetilde v \cdot\nablah \sigma \bigr)  \,dz' \times \bigl( \dz v \cdot \sigma v_t \bigr) \biggr\rbrack \idxh \,dz \\
		& ~~~~ \lesssim \norm{\sigma}{\Hnorm{2}} \norm{v}{\Hnorm{2}}^2 \norm{\sigma v_t}{\Lnorm{2}} \lesssim \delta \norm{\sigma v_t}{\Lnorm{2}}^2 + C_\delta K_1^2 C_d^2.
	\end{aligned}
	\end{equation}
	Thus from \eqref{002}, one has
	\begin{equation}\label{H1-v}
		\begin{aligned}
			& 	\dfrac{d}{dt} (\mu\norm{\nablah v}{\Lnorm{2}}^2 + \mu \norm{\dz v}{\Lnorm{2}}^2 + (\mu+\lambda) \norm{\dvh v}{\Lnorm{2}}^2 ) + \norm{\sigma v_t}{\Lnorm{2}}^2 \\
			& ~~~~ \leq \mathcal H(K_1,C_d).
		\end{aligned}
	\end{equation}
	Next we estimate the second order spatial derivatives. Taking the $ L^2 $-inner product of \subeqref{rfeq:isen-CPE}{2} with $ \partial_{zz} v_t $ yields
	\begin{align}\label{003}
		& \dfrac{1}{2} \dfrac{d}{dt} ( \mu \norm{\nablah \partial_z v}{\Lnorm{2}}^2  + \mu \norm{\partial_{zz}v}{\Lnorm{2}}^2  + (\mu+\lambda) \norm{\dvh \partial_z v}{\Lnorm{2}}^2 ) + \norm{\sigma \partial_z v_t}{\Lnorm{2}}^2 \nonumber\\
		& ~~ = - \intw \dz ( \sigma^2 v \cdot \nablah v) \cdot \dz v_t \idx - \intw \dz( \sigma^2 w \dz v) \cdot \dz v_t \idx =: L_8 + L_9.
	\end{align}
	At the same time, taking the $ L^2 $-inner product of \subeqref{rfeq:isen-CPE}{2} with $ \deltah v_t $ yields
	\begin{align}\label{004}
	& \dfrac{1}{2} \dfrac{d}{dt} ( \mu \norm{\nablah^2 v}{\Lnorm{2}}^2 + \mu \norm{\nablah \dz v}{\Lnorm{2}}^2 + (\mu+\lambda) \norm{\nablah \dvh v}{\Lnorm{2}}^2) \nonumber \\
	& ~~~~ + \norm{\sigma \nablah v_t}{\Lnorm{2}}^2 = - 2 \intw
	\bigl( \sigma \nablah \sigma \cdot \nablah v_t \bigr) \cdot v_t \idx \nonumber \\
	& ~~~~ - \intw \nablah (\sigma^2 v \cdot \nablah v) : \nablah  v_t  \idx - \intw \nablah (\sigma^2 w \dz v) : \nablah v_t \idx \nonumber \\
	& ~~~~ - \intw \nablah^2 \sigma^{2\gamma} : \nablah v_t\idx =: \sum_{i=10}^{13} L_{i}.
	\end{align}
	Now we list the corresponding estimates for the terms in the right-hand side of \eqref{003} and \eqref{004}.
	\begin{align*}
		& L_8 \lesssim \norm{\sigma}{\Hnorm{2}} \norm{v}{\Hnorm{2}}^2 \norm{\sigma \dz v_t}{\Lnorm{2}} \lesssim \delta \norm{\sigma \dz v_t}{\Lnorm{2}}^2 + C_\delta K_1^2 C_d^2, \\
		& L_{10} \lesssim \norm{\nablah \sigma}{\Lnorm{6}} \norm{\nablah v_t}{\Lnorm{2}} \norm{\sigma v_t}{\Lnorm{3}} \lesssim \norm{\sigma}{\Hnorm{2}}^{3/2} \norm{\nablah v_t}{\Lnorm{2}} \norm{\sigma v_t}{\Lnorm{2}}^{1/2} \\
		& ~~~~  ~~~~ \times (\norm{\sigma v_t}{\Lnorm{2}}^{1/2} + \norm{\nabla v_t}{\Lnorm{2}}^{1/2} ) \lesssim \omega \norm{\nabla v_t}{\Lnorm{2}}^2 + C_\omega C_d (K_1^6 + 1) , \\
		& L_{11} \lesssim \norm{\sigma}{\Hnorm{2}} \norm{v}{\Hnorm{2}}^2 \norm{\sigma \nablah v_t}{\Lnorm{2}} \lesssim \delta \norm{\sigma \nablah v_t}{\Lnorm{2}}^2 + C_\delta K_1^2 C_d^2, \\
		& L_{13} \lesssim \omega\norm{\nabla v_t}{\Lnorm{2}}^2 + C_\omega K_1^{4\gamma}.
	\end{align*}
	To estimate $ L_9, L_{12} $, we plug in \eqref{vertical-isen} in the corresponding expressions. Indeed
	\begin{align*}
		& L_9 = - \intw \dz (\sigma^2 w) \dz v \cdot \dz v_t \idx  - \intw \sigma^2 w \partial_{zz} v \cdot \dz v_t \idx \\
		& ~~~~ = \intw ( \sigma \widetilde{\dvh v}+ 2 \widetilde v \cdot \nablah \sigma ) \times ( \dz v \cdot \sigma \dz v_t ) \idx \\
		& ~~~~ ~~~~ + \int_0^1 \inth \biggl\lbrack \int_0^z \bigl( \sigma \widetilde{\dvh v} + 2 \widetilde v \cdot \nablah \sigma \bigr) \,dz' \times ( \partial_{zz} v \cdot \sigma \dz v_t ) \biggr\rbrack \idxh \,dz\\
		& ~~~~ \lesssim  \norm{\sigma}{\Hnorm{2}} \norm{v}{\Hnorm{2}}^2 \norm{\sigma \dz v_t}{\Lnorm{2}} + \norm{\sigma}{\Hnorm{2}}\norm{v}{\Hnorm{2}}^{3/2} \norm{v}{\Hnorm{3}}^{1/2} \norm{\sigma \dz v_t}{\Lnorm{2}}\\
		& ~~~~ \lesssim \delta \norm{\sigma \dz v_t}{\Lnorm{2}}^2 + \omega \norm{v}{\Hnorm{3}}^2 + C_{\delta,\omega} (C_d^2 K_1^2 + C_d^{3} K_1^4), \\
		& L_{12} \lesssim \int_0^1 \inth \biggl\lbrack \biggl\lbrack \biggl( \int_0^z \bigl( \sigma^2  \widetilde{\nablah  \dvh v} + 2 \sigma \widetilde{\dvh v} \nablah\sigma + 2 \widetilde v \cdot \nablah \sigma \nablah \sigma \\
		& ~~~~ ~~~~ ~~~~ + 2 \sigma \widetilde v \cdot\nablah \nablah \sigma + 2 \sigma  \widetilde{\nablah  v} \cdot \nablah \sigma \bigr) \,dz' \biggr) \otimes \dz v \biggr\rbrack : \nablah v_t  \biggr\rbrack \idxh \,dz \\
		& ~~~~ ~~~~ + \int_0^1 \inth \biggl\lbrack \int_0^z \bigl( \sigma \widetilde{\dvh v} + 2 \widetilde v \cdot \nablah \sigma \bigr) \,dz' \times \bigl( \nablah \partial_{z} v : \sigma \nablah v_t \bigr) \biggr\rbrack \idxh \,dz\\
		& ~~~~ \lesssim \norm{\sigma}{\Hnorm{2}}^2 \norm{v}{\Hnorm{2}}^{3/2} \norm{v}{\Hnorm{3}}^{1/2} \norm{\nablah v_t}{\Lnorm{2}} \lesssim \omega \norm{\nablah v_t}{\Lnorm{2}}^2 + \omega \norm{v}{\Hnorm{3}}^2 \\
		& ~~~~ ~~~~ + C_{\omega} C_d^3 K_1^8.
	\end{align*}
	Therefore \eqref{003} and \eqref{004} yield, after summing the above inequalities,
	\begin{equation}\label{H2-v}
		\begin{aligned}
			& \dfrac{d}{dt} (\mu \norm{\nablah^2 v}{\Lnorm{2}}^2 + 2\mu \norm{\nablah \dz v}{\Lnorm{2}}^2 + \mu \norm{\partial_{zz} v}{\Lnorm{2}}^2\\
			& ~~~~ ~~~~ + (\mu + \lambda) \norm{\nabla \dvh v}{\Lnorm{2}}^2)  + \norm{\sigma \nabla v_t}{\Lnorm{2}}^2 \\
			& ~~~~ \leq \omega (\norm{\nabla v_t}{\Lnorm{2}}^2 + \norm{v}{\Hnorm{3}}^2) + C_\omega \mathcal H(K_1,C_d).
		\end{aligned}
	\end{equation}
	Finally, we provide estimates for the third spatial derivative of $ v $. Applying $ \partial \in \lbrace \partial_x, \partial_y,\partial_z \rbrace $ to \subeqref{rfeq:isen-CPE}{2} yields
	\begin{equation}\label{eq:3-spatial-deirvative}
		\begin{aligned}
			& \mu \deltah \partial v + \mu \partial_{zz} \partial v + (\mu+ \lambda) \nablah \dvh \partial v = \partial (\sigma^2 v_t) + \partial (\sigma^2 v\cdot \nablah v) \\
			& ~~~~ + \partial (\sigma^2 w \dz v) + \partial \nablah \sigma^{2\gamma}.
		\end{aligned}
	\end{equation}
	We first consider the case when $ \partial = \partial_h \in \lbrace \partial_x, \partial_y \rbrace $ followed by the case when $ \partial = \partial_z  $. Taking the $ L^2 $-inner product of  \eqref{eq:3-spatial-deirvative}, when $ \partial = \partial_h $, with  $ \deltah \partial_h v $ and integrating by parts give
	\begin{align*}
	&  \mu \norm{\nablah^2 \partial_h v}{\Lnorm{2}}^2 +\mu \norm{\nablah \partial_{hz} v}{\Lnorm{2}}^2  + (\mu+\lambda)  \norm{\nablah \dvh \partial_h v}{\Lnorm{2}}^2   \\
	& =  \intw (\partial_h(\sigma^2 v_t) + \partial_h (\sigma^2 v\cdot\nablah v) + \partial_h(\sigma^2 w \dz v) + \partial_h \nablah \sigma^{2\gamma} ) \cdot \deltah \partial_h v \idx.
	\end{align*}
	Similarly, taking the $ L^2 $-inner product of \eqref{eq:3-spatial-deirvative}, when $ \partial = \partial_z $, with $ \partial_{zzz} v $ yields
	\begin{align*}
		& \mu \norm{\nablah \partial_{zz} v}{\Lnorm{2}}^2 + \mu \norm{\partial_{zzz} v}{\Lnorm{2}}^2 + (\mu+\lambda) \norm{\dvh \partial_{zz} v}{\Lnorm{2}}^2 \\
		& = \intw  (\partial_z(\sigma^2 v_t) + \partial_z (\sigma^2 v\cdot\nablah v) + \partial_z (\sigma^2 w \dz v) ) \cdot  \partial_{zzz} v \idx.
	\end{align*}
	Therefore, one obtains, after applying the Cauchy-Schwarz inequality,
	\begin{align}\label{005}
		& \norm{\nabla^3 v}{\Lnorm{2}}^2 \lesssim \norm{\nabla(\sigma^2 v_t)}{\Lnorm{2}}^2 + \norm{\nabla (\sigma^2 v\cdot \nablah v)}{\Lnorm{2}}^2 + \norm{\nabla (\sigma^2 w \dz v)}{\Lnorm{2}}^2 \nonumber \\
		& ~~~~ ~~~~ + \norm{\nablah ^2 \sigma^{2\gamma}}{\Lnorm{2}}^2.
	\end{align}
	Then, as before,
	\begin{align*}
		& \norm{\nabla(\sigma^2 v_t)}{\Lnorm{2}}^2 \lesssim \norm{\nabla \sigma}{\Lnorm{6}}^2 \norm{\sigma v_t}{\Lnorm{3}}^2 + \norm{\sigma}{\Hnorm{2}}^4 \norm{\nabla v_t}{\Lnorm{2}}^2  \\
		& ~~~~ \lesssim \norm{\sigma}{\Hnorm{2}}^3 \norm{\sigma v_t}{\Lnorm{2}} ( \norm{\nabla v_t}{\Lnorm{2}} + \norm{\sigma v_t}{\Lnorm{2}}) + \norm{\sigma}{\Hnorm{2}}^4 \norm{\nabla v_t}{\Lnorm{2}}^2\\
		& ~~~~ \lesssim K_1^4 \norm{\nabla v_t}{\Lnorm{2}}^2 + C_d (K_1^2 + K_1^3), \\
		& \norm{\nabla (\sigma^2 v\cdot \nablah v)}{\Lnorm{2}}^2 \lesssim \norm{\sigma}{\Hnorm{2}}^4 \norm{v}{\Hnorm{2}}^4 \lesssim C_d^2 K_1^4, \\
		&  \norm{\nablah ^2 \sigma^{2\gamma}}{\Lnorm{2}}^2 \lesssim K_1^{4\gamma}.
	\end{align*}
	On the other hand,
	\begin{align*}
		& \partial_z (\sigma^2 w \dz v) = - \dvh (\sigma^2 \widetilde v) \dz v - \biggl( \int_0^z \dvh(\sigma^2 \widetilde v) \,dz' \biggr) \times  \partial_{zz} v,\\
		& \partial_h (\sigma^2 w \dz v) = - \biggl( \int_0^z \partial_h \dvh (\sigma^2 \widetilde v) \,dz' \biggr)\times \dz v  - \biggl( \int_0^z \dvh (\sigma^2 \widetilde v) \,dz' \biggr) \times \partial_{hz} v.
	\end{align*}
	Then by applying \eqref{ineq-supnorm}, the H\"older and Minkowski's inequalities, it yields
	\begin{align*}
		& \norm{\nabla (\sigma^2 w \dz v)}{\Lnorm{2}}^2 \lesssim \norm{\sigma}{\Hnorm{2}}^4 \norm{v}{\Hnorm{2}}^4 + \biggl(\int_0^1 \bigl( \hnorm{\nablah \sigma}{\Lnorm{4}} \hnorm{\sigma}{\Lnorm{\infty}} \hnorm{\widetilde v}{\Lnorm{\infty}}\\
		& ~~~~ ~~~~ + \hnorm{\sigma}{\Lnorm{\infty}}^2 \hnorm{\widetilde{\nablah v}}{\Lnorm{4}}\bigr)\,dz'\biggr)^2 \times \int_0^1 \hnorm{\nabla \dz v}{\Lnorm{4}}^2\,dz + \biggl(\int_0^1 \hnorm{\nablah^2 (\sigma^2 \widetilde v)}{\Lnorm{2}} \,dz' \biggr)^2\\
		& ~~~~ \times \int_0^1 \hnorm{\dz v}{\Hnorm{1}} \hnorm{\dz v}{\Hnorm{2}}\,dz \lesssim \delta \norm{v}{\Hnorm{3}}^2 + C_\delta (C_d^2 K_1^4 + C_d^3 K_1^8).
	\end{align*}
	Notice that $ \norm{v}{\Hnorm{3}}^2 \lesssim \norm{\nabla^3 v}{\Lnorm{2}}^2 + C_d $, where $ C_d $ is as in \eqref{priori-a-isen}.
	Hence, after summing the above estimates with an appropriate $ \delta $, \eqref{005} yields
	\begin{equation}\label{H3-v-diss}
		\norm{v}{\Hnorm{3}}^2 \leq \norm{\nabla^3 v}{\Lnorm{2}}^2 + C_d \leq C K_1^4 \norm{\nabla v_t}{\Lnorm{2}}^2 + \mathcal H(K_1,C_d).
	\end{equation}
	
	We summarize the estimates obtained, so far, in this section in the following:
	\begin{proposition}\label{prop:lwbd-inde-ests}
			Consider the solution $(\sigma,v) = (\rho^{1/2},v) $  to \eqref{isen-CPE} with the bound \eqref{priori-a-isen} and initial data satisfying \eqref{bound-of-initial-energy}, \eqref{initial-bound-001}. There is a positive constant $ T^*=T^*(C_d,K_1,K_2) $, sufficiently small, such that $ (\sigma,v) $ admits the following bounds, for $ T = \min \lbrace T^*, T_d \rbrace $,
			\begin{gather*}
				\inf_{(\vec{x},t)\in \Omega \times [0,T]}\rho(\vec{x},t) \geq \dfrac{1}{2} \inf_{\vec{x}\in \Omega}\rho_0 > 0 , ~
				\sup_{0\leq t\leq T} \norm{\sigma(t)}{\Hnorm{2}} \leq 2 K_1, \\
				\sup_{0\leq t\leq T} \norm{\dt \sigma(t)}{\Hnorm{1}} \leq K_3, \\
				\sup_{0\leq t\leq T} \bigl( \norm{v(t)}{\Hnorm{2}}^2 + \norm{(\sigma v_t)(t)}{\Lnorm{2}}^2 \bigr) + \int_0^T \biggl( \norm{v(t)}{\Hnorm{3}}^2 + \norm{v_t(t)}{\Hnorm{1}}^2 \biggr) \,dt \leq K_4^2,
			\end{gather*}
		where $ K_4 = \sqrt{2C_{\mu,\lambda}} K_2, K_3 = CK_1K_4 $ are given in \eqref{Total-Bd-v} and \eqref{H1-t-conti-final}. Notably, the bounds in these estimates depend only on the initial bounds $ K_1, K_2 $ and do not depend on the lower bound of density. Also, the smallness of $ T^* $ does not 
		depend on $ \underline\rho $, even though $ T_d $ may depend on $ \underline\rho $.
	\end{proposition}
	\begin{proof}{}
		Denote by
		\begin{equation}\label{total-energy}
			\begin{aligned}
				& \mathcal E(t) := \norm{\sigma v_t}{\Lnorm{2}}^2 + \norm{\sigma v}{\Lnorm{2}}^2 + \mu \norm{\nabla v}{\Lnorm{2}}^2 + (\mu+\lambda) \norm{\dvh v}{\Lnorm{2}}^2  \\
				& ~~~~ + \mu \norm{\nablah^2 v}{\Lnorm{2}}^2 + 2 \mu \norm{\nablah \dz v}{\Lnorm{2}}^2 + \mu \norm{\partial_{zz} v}{\Lnorm{2}}^2 \\
				& ~~~~ + (\mu+\lambda)\norm{\nabla \dvh v}{\Lnorm{2}}^2.
			\end{aligned}
		\end{equation}
		Then from \eqref{Temp-v}, \eqref{L2-v}, \eqref{H1-v} and \eqref{H2-v}, we have
		\begin{align*}
			& \dfrac{d}{dt} \mathcal E(t) + c_{\mu,\lambda} \bigl( \norm{\nabla v_t}{\Lnorm{2}}^2 + \norm{\nabla v}{\Lnorm{2}}^2 + \norm{\sigma v_t}{\Lnorm{2}}^2 + \norm{\sigma \nabla v_t}{\Lnorm{2}}^2 \bigr)\\
			& ~~~~ \leq \omega \bigl( \norm{v}{\Hnorm{3}}^2 + \norm{\nabla v_t}{\Lnorm{2}}^2 \bigr) + C_\omega \mathcal H (K_1,K_3',C_d).
		\end{align*}
		Then integrating the above inequality yields, for $ T \in (0, T_d] $, where $ T_d $ is as in \eqref{priori-a-isen},
			\begin{align*}
				& \sup_{0\leq  t\leq  T} \mathcal E(t) + c_{\mu,\lambda} \int_0^T \biggl( \norm{\nabla v_t}{\Lnorm{2}}^2 + \norm{\nabla v}{\Lnorm{2}}^2 + \norm{\sigma v_t}{\Lnorm{2}}^2 \\
				& ~~~~ + \norm{\sigma \nabla v_t}{\Lnorm{2}}^2 \biggr) \,dt  \leq \mathcal E(0) + \omega C_d + C_\omega T \mathcal H(K_1,K_3',C_d).
			\end{align*}
		Then together with \eqref{H3-v-diss}, we have, after choosing $ \omega $ small enough and then $ T $ sufficiently small,
		\begin{equation}\label{Total-Bd-v}
			\begin{aligned}
				& \sup_{0\leq t\leq T} \bigl( \norm{v(t)}{\Hnorm{2}}^2 + \norm{\sigma v_t(t)}{\Lnorm{2}}^2 \bigr) + \int_0^T \biggl( \norm{v}{\Hnorm{3}}^2 + \norm{v_t}{\Hnorm{1}}^2 \biggr) \,dt \\
				& ~~~~ \leq C_{\mu,\lambda}K_2^2 +  \omega C_d C_{\mu,\lambda}(K_1^4+1) + C_\omega T \mathcal H(K_1,K_3',C_d) \\
				& ~~~~ \leq 2 C_{\mu,\lambda} K_2^2 =: K_4^2,
			\end{aligned}
		\end{equation}
		where we have employed inequality \eqref{ineq:embedding-weighted} and the fact for some positive constant $ C_{\mu,\lambda} > 0 $ we have
		$$ C^{-1}_{\mu,\lambda}(\norm{v}{\Hnorm{2}}^2 + \norm{\sigma v_t}{\Lnorm{2}}^2) \leq  \mathcal E(t) \leq C_{\mu,\lambda}(\norm{v}{\Hnorm{2}}^2 + \norm{\sigma v_t}{\Lnorm{2}}^2).  $$
		Then plugging in \eqref{Total-Bd-v} back into \eqref{H1-t-conti} implies
		\begin{equation}\label{H1-t-conti-final}
			\norm{\dt \sigma}{\Hnorm{1}} \leq C K_4 K_1 =: K_3.
		\end{equation}
		Thus the conclusion is drawn from \eqref{lower-bound-conti}, \eqref{H2-conti}, \eqref{Total-Bd-v} and \eqref{H1-t-conti-final}.
	\end{proof}
		
	{\par\hfill\par\noindent\bf Existence of strong solutions with vacuum but no gravity and $ \gamma > 1 $ \par}
	Now we are in the place to remove the strict positivity (of the initial density profile) assumption in Proposition \ref{prop:exist-isen-novacuum}. In order to do so, we introduce a sequence of approximating initial data $ (\rho_{0,n}, v_{0,n} ) $ satisfying in addition to \eqref{bound-of-initial-energy}, \eqref{isen-cptbl-conds}, \eqref{bound-of-initial}, $$ \rho_{0,n} \geq \dfrac{1}{n} > 0, $$ such that $$ \rho_{0,n}^{1/2} \rightarrow \rho_0^{1/2}, v_{0,n} \rightarrow v_0 $$ in $ H^2(\Omega) $, as $ n \rightarrow \infty $, where $ (\rho_0, v_0 ) $ (or equivalently $ (\sigma_0, v_0) $) is given in \eqref{isen-initial} satisfying \eqref{isen-cptbl-conds} and \eqref{bound-of-initial}.
	We require that the initial physical energy and total mass given in \eqref{bound-of-initial-energy} with $ \rho_0, v_0 $ replaced by $ \rho_{0,n}, v_{0,n} $ satisfy
	\begin{gather*}
		0 < \intw \rho_{0,n} \idx  = M < \infty, \\
		0 < \intw \rho_{0,n} \abs{v_{0,n}}{2} \idx + \dfrac{1}{\gamma-1} \intw \rho_{0,n}^\gamma \idx \leq E_0 + 1 < \infty,
	\end{gather*}
	uniformly in $ n $, so that when we apply the inequality \eqref{ineq:embedding-weighted}, the constant for the inequality is independent of $ n $. 
	
	Now we apply Proposition \ref{prop:exist-isen-novacuum} with the initial data $ (\rho_{0,n},v_{0,n}) $. Indeed, consider $ M_0 = B_1 $, $ M_1 = B_2 + n B_2 $. Then Proposition \ref{prop:exist-isen-novacuum} guarantees that there is a $ T_1 = T_1(n,B_1,B_2) $ such that \eqref{isen-CPE}  admits a strong solution $ (\rho_n, v_n) = (\sigma_n^2 , v_n) $ satisfying
	\begin{align*}
		& \sup_{0\leq t\leq T_1} \norm{\sigma_{n}(t)}{\Hnorm{2}}^2 \leq 2B_1, ~~ \sup_{0\leq t\leq T_1} \norm{\dt \sigma_n(t)}{\Hnorm{1}}^2 \leq \mathcal C_1(B_1,B_2,n), \\
		& \sup_{0\leq t\leq T_1} \bigl( \norm{v_n(t)}{\Hnorm{2}}^2 + \norm{\dt v_{n}(t)}{\Lnorm{2}}^2 \bigr) + \int_0^{T_1}\bigl( \norm{v_n(t)}{\Hnorm{3}}^2 + \norm{\dt v_n(t)}{\Hnorm{1}}^2 \bigr) \,dt \\
		& ~~~~ ~~~~ \leq \mathcal C_2(B_1,B_2,n), ~~ \text{and} ~ \rho_n = \sigma_n^2 \geq \dfrac{1}{2n} .
	\end{align*}
	Sequently, we apply Proposition \ref{prop:lwbd-inde-ests} with $ K_1 = B_1^{1/2}, K_2 = 2B_2^{1/2} $, $ C_d = (1 + 2 n) \mathcal C_2(B_1,B_2,n)  $ and $ T_d = T_1 $. It yields that there is a $ T_2 = T_2(B_1,B_2,n) \leq T_1 $ such that the following bounds are satisfied
	\begin{equation}\label{uniform-bound-vcm}
		\begin{gathered}
			\sup_{0\leq t\leq T_2} \norm{\sigma_n(t)}{\Hnorm{2}}\leq 2 B_1^{1/2}, ~~ \sup_{0\leq t\leq T_2} \norm{\dt \sigma_n(t)}{\Hnorm{1}} \leq \mathcal C_3 (B_1,B_2),\\
			\sup_{0\leq t\leq T_2} \bigl( \norm{v_n(t)}{\Hnorm{2}}^2 + \norm{(\sigma_n v_{n,t})(t)}{\Lnorm{2}}^2 \bigr) \\
			 ~~~~ + \int_0^{T_2} \bigl( \norm{v_n(t)}{\Hnorm{3}}^2 + \norm{v_{n,t}(t)}{\Hnorm{1}}^2 \bigr) \,dt
			 \leq \mathcal C_{4}(B_1,B_2), \\
			\text{and} ~ \inf_{(\vec{x},t) \in \Omega\times [0,T_2]} \rho_n \geq \dfrac{1}{2n}.
		\end{gathered}
	\end{equation}
	
	Next, let $ (\sigma_n,v_n)|_{t=T_2} $ as a new initial data for \eqref{isen-CPE}. The same arguments as above yield the bound \eqref{uniform-bound-vcm} with lower bound of $ \rho_n $ replaced by $ \frac{1}{4n} $, $ B_1 $ replaced by $ 4 B_1 $ and $ B_2 $ replaced by $ \mathcal C_4(B_1,B_2) $. That is, for some $ \delta T =\delta T(B_1,B_2,n) > 0 $,
	\begin{align*}
		& \sup_{T_2<t<T_2+\delta T} \norm{\sigma_n}{\Hnorm{2}}\leq 4 B_1^{1/2},  \sup_{T_2<t<T_2+\delta T} \norm{\dt \sigma_n}{\Hnorm{1}} \leq \mathcal C_3 (4B_1,\mathcal C_{4}(B_1,B_2)),\\
			& \sup_{T_2<t<T_2+\delta T} \bigl( \norm{v_n}{\Hnorm{2}}^2 + \norm{\sigma_n v_{n,t}}{\Lnorm{2}}^2 \bigr) + \int_{T_2}^{T_2+\delta T} \norm{v_n}{\Hnorm{3}}^2 + \norm{v_{n,t}}{\Hnorm{1}}^2 \,dt \\
			& ~~~~ \leq \mathcal C_{4}(4B_1,\mathcal C_{4}(B_1,B_2)), ~~ \inf_{(\vec{x},t) \in \Omega\times (T_2,T_2+\delta T)} \rho_n \geq \dfrac{1}{4n}.
	\end{align*}
	Now we apply Proposition \ref{prop:lwbd-inde-ests} with $ T_d = T_2 + \delta T $ and $ C_d = \mathcal C_{4}(4B_1,\mathcal C_{4}(B_1,B_2)) $ in the time interval $ (0,T_d) $. This will yield that there is a $ T^* = T^*(B_1,B_2) $ and $ T_3 := \min \lbrace T^*, T_2 + \delta T \rbrace $, the bounds in \eqref{uniform-bound-vcm} hold with $ T_2 $ replaced by $ T_3 $.
	
	If $ T_3 = T^* $, we have got an existence time independent of $ n $ and this finishes the job. Otherwise, let $ (\sigma_n,v_n) |_{t= T_3} $ as a new initial data and repeat the arguments above to get the bounds in \eqref{uniform-bound-vcm} with $ T_2 $ replaced by $ T_4:= \min\lbrace T^*, T_3 + \delta T\rbrace = \min\lbrace T^*, T_2 + 2 \delta T\rbrace $. Keep repeating this process, one will eventually get that there is a $ m \in \mathbb Z^+ $ sufficiently large that $ T_m := \min\lbrace T^*, T_2 + (m-2) \delta T \rbrace = T^* $.
	
	Therefore, we have got a sequence of approximating solutions $ (\rho_n, v_n) = (\sigma_n^2, v_n) $ with a uniform existence time $ T^* $ independent of $ n $ for the approximating initial data $ (\rho_{0,n}, v_{0,n}) $ constructed above. In particular, $ (\sigma_n, v_n) $ satisfies the bounds in \eqref{uniform-bound-vcm} with $ T_2 $ replaced by $ T^* $. Thus by taking $ n\rightarrow \infty $, it is straightforward to check that we have got a strong solution $ (\rho, v) = (\sigma^2, v) $ to \eqref{isen-CPE}.
	In fact, we have the following:
	
	\begin{proposition}\label{prop:exist-isen-vacuum}
		Consider the initial data $ (\rho_0, v_0) $ (or equivalently $ (\sigma_0, v_0) $) given in \eqref{isen-initial} satisfying \eqref{bound-of-initial-energy}, \eqref{isen-cptbl-conds} and \eqref{bound-of-initial}. There is a constant $ T^* > 0 $ such that there exists a solution $ (\rho, v) = (\sigma^{2}, v) $ to equation \eqref{isen-CPE} satisfying
	\begin{equation}\label{space-isen-vacuum}
	\begin{gathered}
		\sigma \in L^\infty(0,T^*;H^2(\Omega)), ~~ \dt \sigma \in L^\infty(0,T^*; H^1(\Omega)), \\
		v \in L^\infty(0,T^*;H^2(\Omega))\cap L^2(0,T^*;H^3(\Omega)), ~~ \dt v \in L^2(0,T^*;H^1(\Omega))\\
		\sigma \dt v \in L^\infty(0,T^*;L^2(\Omega)),
	\end{gathered}
	\end{equation}
	and
	\begin{equation}\label{bound-isen-vacuum}
	\begin{aligned}
		& \sup_{0\leq t\leq T^*} \norm{\sigma(t)}{\Hnorm{2}}\leq 2 B_1^{1/2}, ~~ \sup_{0\leq t\leq T^*} \norm{\dt \sigma(t)}{\Hnorm{1}} \leq \mathcal C_3 (B_1,B_2),\\
			& \sup_{0\leq t\leq T^*} \bigl( \norm{v(t)}{\Hnorm{2}}^2 + \norm{(\sigma v_{t})(t)}{\Lnorm{2}}^2 \bigr) + \int_0^{T^*} \bigl(\norm{v}{\Hnorm{3}}^2 + \norm{v_{t}}{\Hnorm{1}}^2 \bigr) \,dt \\
			& ~~~~ \leq \mathcal C_{4}(B_1,B_2), ~~ \text{and} ~ \inf_{(\vec{x},t) \in \Omega\times [0,T^*]} \rho \geq 0,
	\end{aligned}
	\end{equation}
	for some constant $ \mathcal C_3 = \mathcal C_3(B_1,B_2),  \mathcal C_4 = \mathcal C_4(B_1,B_2) $.
	\end{proposition}
	
\section{Continuous dependence on initial data and uniqueness}\label{sec:cn-ddy-uni}
	In this section, we will show the continuous dependence of the solutions of  \eqref{isen-CPE-g} and \eqref{isen-CPE} on the initial data. This will also imply the uniqueness of strong solutions constructed in Proposition \ref{prop:exist-isen-g} and Proposition \ref{prop:exist-isen-vacuum}.
	\subsection{The case when there is gravity but no vacuum and $ \gamma = 2 $}
	Consider two sets of initial data $ (\rho_{i,0},v_{i,0}) =  (\xi_{i,0}+\frac{1}{2}gz ,v_{i,0}),  i = 1,2, $ in \eqref{isen-initial-g} for \eqref{isen-CPE-g} satisfying \eqref{isen-cptbl-conds-g}, \eqref{bound-of-initial-g}.
	Denote $ (\rho_i,v_i) = (\xi_i+\frac 1 2 gz, v_i), i = 1,2, $ as the corresponding strong solutions constructed in Proposition \ref{prop:exist-isen-g} in the interval $ [0,T] $ for some $ T > 0 $. Then we have $(\xi_i, v_i) \in \mathfrak X_T, i=1,2 $. Throughout this section we will denote the constant $ C > 0 $ which may be different from line to line and depends on $ \mu, \lambda, B_{g,1}, B_{g,2},\underline\rho, T $.
	Also, we will use the notations
	\begin{gather*}
		\xi_{12} := \xi_1 - \xi_2, ~ v_{12} := v_1- v_2, \\
		\xi_{12,0}  := \xi_{1,0} - \xi_{2,0}, ~ v_{12,0} := v_{1,0}- v_{2,0}.
	\end{gather*}
	Taking the difference of the equations satisfied by $ (\xi_i, v_i), i=1,2 $,  as in \eqref{eq:differences-g}, then $ (\xi_{12}, v_{12}) $ satisfies
	\begin{equation*}
		\begin{cases}
			\dt \xi_{12} + \overline{v_1} \cdot \nablah \xi_{12} + \xi_{12} \overline{\dvh v_1} + \overline{v_{12}} \cdot \nablah \xi_2 + \xi_2 \overline{\dvh v_{12}}\\
			~~~~ ~~~~ + \dfrac{g}{2} \overline{z \dvh v_{12}} = 0, \\
			\rho_1 \dt v_{12} - \mu \deltah v_{12} - \mu \partial_{zz} v_{12} - (\mu+\lambda) \nablah \dvh v_{12} = - \xi_{12} \dt v_2 \\
			~~~~ ~~~~ - \nablah (\xi_{12}(\rho_1 + \rho_2)) - \xi_{12} v_1 \cdot \nablah v_1 - \rho_2 v_{12} \cdot \nablah v_1  \\
			~~~~ ~~~~- \rho_2 v_2 \cdot\nablah v_{12} - (\rho_1 w_1 - \rho_2 w_2) \dz v_1 - \rho_2 w_2 \dz v_{12}.
		\end{cases}
	\end{equation*}
	As in \eqref{weakly-conti-003}, we will have the following inequalities
	\begin{align*}
		& \dfrac{d}{dt} \norm{\xi_{12}}{\Lnorm{2}}^2 \leq ( \norm{v_1}{\Hnorm{3}}^2 + C ) \norm{\xi_{12}}{\Lnorm{2}}^2  + C \norm{\sqrt{\rho_1}v_{12}}{\Lnorm{2}}^2\\
		& ~~~~ ~~~~ + C \norm{\nabla v_{12}}{\Lnorm{2}}^2,\\
		& \dfrac{d}{dt} \norm{\sqrt{\rho_1}v_{12}}{\Lnorm{2}}^2 + c_{\mu,\lambda} \norm{\nabla v_{12}}{\Lnorm{2}}^2 \leq C  \norm{\sqrt{\rho_1}v_{12}}{\Lnorm{2}}^2 \\
		& ~~~~ + C (\norm{\dt v_2}{\Hnorm{1}}^2 + \norm{v_1}{\Hnorm{3}}^2 + 1)  (\norm{\xi_{12}}{\Lnorm{2}}^2 + \norm{\sqrt{\rho_1}v_{12}}{\Lnorm{2}}^2).
	\end{align*}
	After taking a suitable linear combination of the above inequalities, we have
	\begin{equation*}
		\begin{aligned}
			& \dfrac{d}{dt} \bigl( \norm{\xi_{12}}{\Lnorm{2}}^2 + C_{\mu,\lambda}' \norm{\sqrt{\rho_1}v_{12}}{\Lnorm{2}}^2 \bigr) + C_{\mu,\lambda}'' \norm{\nabla v_{12}}{\Lnorm{2}}^2 \\
			& ~~~~ \leq C(\norm{v_1}{\Hnorm{3}}^2 + \norm{\dt v_2}{\Hnorm{1}}^2 + 1 )( \norm{\xi_{12}}{\Lnorm{2}}^2 + C_{\mu,\lambda}' \norm{\sqrt{\rho_1}v_{12}}{\Lnorm{2}}^2),
		\end{aligned}
	\end{equation*}
	for some positive constants $ C_{\mu, \lambda}', C_{\mu, \lambda}'' $ depending on $ \mu, \lambda $.
	Now we apply the Gr\"onwall's inequality to obtain
	\begin{align*}
		&\sup_{0<\leq t\leq T} (\norm{\xi_{12}(t)}{\Lnorm{2}}^2 + C_{\mu,\lambda}' \norm{(\sqrt{\rho_1}v_{12})(t)}{\Lnorm{2}}^2) + \int_0^T \norm{\nabla v_{12}}{\Lnorm{2}}^2 \,dt \\
		& ~~~~ \leq C e^{C \int_0^T ( \norm{v_1}{\Hnorm{3}}^2 + \norm{\dt v_2}{\Hnorm{1}}^2 + 1 ) \,dt }  \times (\norm{\xi_{12,0}}{\Lnorm{2}}^2 + C_{\mu,\lambda}' \norm{\sqrt{\rho_{1,0}}v_{12,0}}{\Lnorm{2}}^2).
	\end{align*}
	Thus we have shown the following:
	\begin{proposition}\label{prop:uniqueness-g}
	Given two sets of initial data $ (\rho_{i,0},v_{i,0}) =  (\xi_{i,0}+\frac{1}{2}gz ,v_{i,0}),  i = 1,2 $, satisfying \eqref{isen-cptbl-conds-g} and \eqref{bound-of-initial-g}, the corresponding strong solutions $ (\rho_i,v_i) = (\xi_i+\frac 1 2 gz, v_i), i = 1,2 $, of \eqref{isen-CPE-g} constructed in Proposition \ref{prop:exist-isen-g} in the interval $ [0,T] $, for some $ T > 0 $, satisfy
	\begin{equation*}
	\begin{aligned}
		& \norm{\rho_1 - \rho_2}{L^\infty(0,T;L^2(\Omega))} + \norm{v_1- v_2}{L^\infty(0,T;L^2(\Omega))}\\
		& ~~~~ + \norm{\nabla (v_1 - v_2)}{L^2(0,T;L^2(\Omega))} \leq C_{\mu, \lambda, B_{g,1}, B_{g,2},\underline\rho, T} \\
		& ~~~~ ~~~~ \times (\norm{\rho_{1,0}-\rho_{2,0}}{L^2(\Omega))} + \norm{v_{1,0}-v_{2,0}}{L^2(\Omega))}).
	\end{aligned}
	\end{equation*}
	In particular, if $ \rho_{1,0} = \rho_{2,0}, v_{1,0} = v_{2,0} $, we have $ \rho_1 =\rho_2, v_1 = v_2 $ in $ [0,T] $.
	\end{proposition}
	
	\subsection{The case when there is vacuum but no gravity and $ \gamma > 1 $}\label{sec:continuous-dependence}
	First, we claim that any solution $ (\rho, v) = (\sigma^2, v) $ to \eqref{isen-CPE} satisfying \eqref{space-isen-vacuum} with the bounds in \eqref{bound-isen-vacuum} will also satisfy the following equations
	\begin{equation*}{\tag{\ref{isen-CPE}'}}\label{rf:isen-CPE}
		\begin{cases}
			\dt \sigma + \overline{v} \cdot \nablah \sigma + \dfrac{1}{2} \sigma\overline{\dvh v} = 0 & \text{in} ~ \Omega,\\
			\sigma w = - \int_0^z \sigma \widetilde{\dvh v} + 2 \widetilde{v} \cdot \nablah \sigma \,dz & \text{in} ~ \Omega,\\
			\sigma^2 \dt v +\sigma^2 v\cdot\nablah v + \sigma \sigma w \dz v + \nablah \sigma^{2\gamma} \\
			~~~~ ~~~~ = \mu \deltah v + \mu \partial_{zz} v + (\mu+\lambda) \nablah \dvh v & \text{in} ~ \Omega, \\
			\dz \sigma = 0 & \text{in} ~ \Omega.
		\end{cases}
	\end{equation*}
	
	To show this claim, we first consider the non-degenerate variable $ \rho + \varepsilon = \sigma^2 + \varepsilon $, for some constant $ \varepsilon > 0 $. From \eqref{conti-isen}, one has
	\begin{equation*}
		\dt (\rho + \varepsilon) + \overline{v} \cdot\nablah (\rho + \varepsilon) + (\rho + \varepsilon) \overline{\dvh v} - \varepsilon \overline{\dvh v} = 0.
	\end{equation*}
	Then after dividing $ (\rho + \varepsilon)^{1/2} $, one has
	\begin{equation}\label{eqeq:001}
		\begin{aligned}
			& 2 \dt (\rho + \varepsilon)^{1/2} + 2 \overline{v} \cdot \nablah (\rho + \varepsilon)^{1/2} + (\rho + \varepsilon)^{1/2} \overline{\dvh v} \\
			& ~~~~ ~~~~ ~~~~ - \dfrac{\varepsilon}{(\rho + \varepsilon)^{1/2}} \overline{\dvh v} = 0.
		\end{aligned}
	\end{equation}
	Now it is easy to verify that \eqref{eqeq:001} will converge to \subeqref{rf:isen-CPE}{1} in the sense of distribution as $ \varepsilon \rightarrow 0 $.
	On the other hand, from \eqref{vertical-isen}, one has
	\begin{equation*}
		\sigma^2 w = - \sigma \int_0^z \bigl( \sigma \widetilde{\dvh v} + 2 \widetilde{v} \cdot\nablah \sigma \bigr) \,dz.
	\end{equation*}
	We define
	\begin{equation*}
		\sigma w_\sigma : = - \int_0^z \sigma \widetilde{\dvh v} + 2 \widetilde{v} \cdot\nablah \sigma \,dz.
	\end{equation*}
	Then $ \sigma \sigma w_\sigma = \rho w $ and we will use hereafter the notation $ \sigma w = \sigma w_\sigma $. As before it is easy to verify that \subeqref{rf:isen-CPE}{3} is equivalent to \subeqref{isen-CPE}{2} in the sense of distribution. Summing up the facts above, we have shown that the solutions to \eqref{isen-CPE} satisfying the \eqref{space-isen-vacuum} regularity with the bounds in \eqref{bound-isen-vacuum} are also solutions to \eqref{rf:isen-CPE}.
	
	Consider two sets of initial data $ ( \rho_{i,0},v_{i,0}) = (\sigma_{i,0}^2, v_{i,0}), i = 1,2 $, in \eqref{isen-initial} for \eqref{isen-CPE} satisfying \eqref{isen-cptbl-conds} and \eqref{bound-of-initial}. Denote $ (\rho_i,v_i) = (\sigma_i^2, v_i), i = 1,2 $, as the corresponding strong solutions constructed in Proposition \ref{prop:exist-isen-vacuum} in the interval $ [0,T^*] $, for some $ T^* > 0 $. Then we have $ (\sigma_i, v_i), i=1,2 $, satisfying the bounds in \eqref{bound-isen-vacuum}. Also $ (\sigma_i, v_i), i=1,2 $, are solutions to \eqref{rf:isen-CPE}. Throughout this section, we will denote the constant $ C > 0 $ which may be different from line to line and depends on $ \mu, \lambda, B_1,B_2, T^* $. Also, we will use the notations
	\begin{gather*}
		\sigma_{12} := \sigma_1 - \sigma_2, ~ v_{12} := v_1- v_2, \\
		\sigma_{12,0}  := \sigma_{1,0} - \sigma_{2,0}, ~ v_{12,0} := v_{1,0}- v_{2,0}.
	\end{gather*}	
	Taking the difference of the equations satisfied by $ (\sigma_i, v_i), i=1,2 $, we have
	\begin{equation}\label{eq:differences-vacuum}
		\begin{cases}
			\dt \sigma_{12} + \overline{v_1} \cdot\nablah \sigma_{12} + \dfrac{1}{2} \sigma_{12} \overline{\dvh v_1} + \overline{v_{12}} \cdot \nablah \sigma_2 \\
			~~~~ ~~~~ + \dfrac{1}{2} \sigma_2 \overline{\dvh v_{12}} = 0,\\
			\sigma_{1}^2 \dt v_{12} - \mu \deltah v_{12} - \mu \partial_{zz} v_{12} - (\mu+\lambda)\nablah \dvh v_{12} \\
			~~~~  = - \sigma_{12} ( \sigma_{1}+\sigma_2)\dt v_{2}
			 - \nablah \bigl( \sigma_{12}\dfrac{\sigma_1^{2\gamma} - \sigma_2^{2\gamma}}{\sigma_1 - \sigma_2}\bigr) \\
			~~~~ ~~~~  - \sigma_{12}(\sigma_1 + \sigma_2) v_2 \cdot \nablah v_2 - \sigma_1^2 v_{12} \cdot \nablah v_2\\
			~~~~ ~~~~  - \sigma_1^2 v_1 \cdot \nablah v_{12} - \sigma_{12} \sigma_2 w_2 \dz v_2 -\sigma_1 (\sigma_1 w_1 - \sigma_2 w_2) \dz v_2 \\
			~~~~ ~~~~  - \sigma_1 \sigma_1 w_1 \dz v_{12},\\
			\sigma_i w_i = - \int_0^z \bigl( \sigma_i \widetilde{\dvh v_i} + 2 \widetilde{v_i} \cdot \nablah \sigma_i \bigr) \,dz, ~~~~~ i = 1,2.
		\end{cases}
	\end{equation}
	Next, multiply \subeqref{eq:differences-vacuum}{1} with $ 2 \sigma_{12} $ and take the $ L^2 $-inner product of \subeqref{eq:differences-vacuum}{2} with $ 2 v_{12} $. Integrating the resultants yields
	\begin{align}
		& \dfrac{d}{dt} \norm{\sigma_{12}}{\Lnorm{2}}^2 = - \inth \bigl( 2 \overline{v_{12}} \cdot\nablah \sigma_2 \sigma_{12} + \sigma_2 \overline{\dvh v_{12}} \sigma_{12} \bigr) \idxh =: G_1, \label{cn-dpcy-vacuum-001} \\
		& \dfrac{d}{dt} \norm{\sigma_1 v_{12}}{\Lnorm{2}}^2 + 2 \bigl( \mu \norm{\nablah v_{12}}{\Lnorm{2}}^2 + \mu \norm{\dz v_{12}}{\Lnorm{2}}^2 + (\mu+\lambda) \norm{\dvh v_{12}}{\Lnorm{2}}^2 \bigr) \nonumber \\
		& ~~ =   2 \intw  \sigma_1 \dt \sigma_1 \abs{v_{12}}{2} \idx - 2 \intw \biggl( \sigma_{12}(\sigma_1+\sigma_2) \dt v_2 \cdot v_{12} \biggr) \idx \nonumber \\
		& ~~~~  + 2 \intw \sigma_{12} \dfrac{\sigma_1^{2\gamma} - \sigma_2^{2\gamma}}{\sigma_1-\sigma_2} \dvh v_{12} \idx - 2 \intw \biggl( (\sigma_{12}(\sigma_1 + \sigma_2) v_2 \cdot \nablah v_2 ) \cdot v_{12} \nonumber \\
		& ~~~~  + ( \sigma_1^2 v_{12} \cdot \nablah v_2 ) \cdot v_{12} + ( \sigma_1^2 v_1 \cdot \nablah v_{12} ) \cdot v_{12} \biggr) \idx - 2 \intw \biggl( \sigma_{12} \sigma_2 w_2 \dz v_2 \cdot v_{12} \nonumber \\
		& ~~~~  + \sigma_1\sigma_1 w_1 \dz v_{12} \cdot v_{12} \biggr) \idx - 2 \intw \sigma_1 (\sigma_1 w_1 -\sigma_2 w_2) \dz v_2 \cdot v_{12}  \idx =: \sum_{i=2}^{7} G_i. \label{cn-dpcy-vacuum-002}
	\end{align}
	As before, we will list the estimates for $ G_i $ terms above, in the following.
	\begin{align*}
		& G_1 \lesssim \hnorm{\overline{v_{12}}}{\Lnorm{4}} \hnorm{\nablah \sigma_2}{\Lnorm{4}} \hnorm{\sigma_{12}}{\Lnorm{2}} + \hnorm{\sigma_2}{\Lnorm{\infty}} \hnorm{\overline{\nablah v_{12}}}{\Lnorm{2}} \hnorm{\sigma_{12}}{\Lnorm{2}} \\
		& ~~~~ \lesssim  \norm{\nabla v_{12}}{\Lnorm{2}}^2 + \norm{\sigma_2}{\Hnorm{2}}^2  \norm{\sigma_{12}}{\Lnorm{2}}^2 + \norm{\sigma_1 v_{12}}{\Lnorm{2}}^2. \\
		& G_2 \lesssim \norm{\dt \sigma_1}{\Lnorm{2}} \norm{\sigma_1 v_{12}}{\Lnorm{3}} \norm{v_{12}}{\Lnorm{6}} \lesssim \norm{\sigma_{1}}{\Lnorm{\infty}}^{1/2} \norm{\dt \sigma_1}{\Lnorm{2}} \norm{\sigma_1 v_{12}}{\Lnorm{2}}^{1/2} \\
		& ~~~~ ~~~~ \times \norm{v_{12}}{\Lnorm{6}}^{3/2} \lesssim \delta \norm{\nabla v_{12}}{\Lnorm{2}}^2 + C_\delta (\norm{\sigma_1}{\Hnorm{2}}^2 \norm{\dt \sigma_1}{\Lnorm{2}}^4 + 1) \\
		& ~~~~ ~~~~ \times \norm{\sigma_1 v_{12}}{\Lnorm{2}}^2. \\
		& G_{3} \lesssim \norm{\sigma_{12}}{\Lnorm{2}}(\norm{\sigma_1}{\Lnorm{\infty}} + \norm{\sigma_2}{\Lnorm{\infty}}) \norm{\dt v_{2}}{\Lnorm{3}} \norm{v_{12}}{\Lnorm{6}} \lesssim \delta\norm{\nabla v_{12}}{\Lnorm{2}}^2\\
		& ~~~~ + C_\delta \norm{\sigma_1v_{12}}{\Lnorm{2}}^2 + C_\delta ( \norm{\sigma_1}{\Hnorm{2}}^2 + \norm{\sigma_2}{\Hnorm{2}}^2 ) \norm{\dt v_{2}}{\Hnorm{1}}^2 \norm{\sigma_{12}}{\Lnorm{2}}^2. \\
		& G_4 \lesssim \norm{\sigma_{12}}{\Lnorm{2}} (\norm{\sigma_1}{\Lnorm{\infty}}^{2\gamma-1} + \norm{\sigma_2}{\Lnorm{\infty}}^{2\gamma-1}) \norm{\nabla v_{12}}{\Lnorm{2}} \lesssim \delta \norm{\nabla v_{12}}{\Lnorm{2}}^2 \\
		& ~~~~  + C_\delta (\norm{\sigma_1}{\Hnorm{2}}^{4\gamma-2} + \norm{\sigma_2}{\Hnorm{2}}^{4\gamma-2}) \norm{\sigma_{12}}{\Lnorm{2}}^2.\\
		& G_5 \lesssim \norm{\sigma_{12}}{\Lnorm{2}}(\norm{\sigma_1}{\Lnorm{\infty}} + \norm{\sigma_2}{\Lnorm{\infty}}) \norm{v_2}{\Lnorm{\infty}} \norm{\nabla v_2}{\Lnorm{3}} \norm{v_{12}}{\Lnorm{6}} \\
		& ~~~~ ~~~~ + \norm{\sigma_1}{\Lnorm{\infty}} \norm{\sigma_1 v_{12}}{\Lnorm{2}} \norm{\nabla v_2}{\Lnorm{3}} \norm{v_{12}}{\Lnorm{6}} \\
		& ~~~~ ~~~~ + \norm{\sigma_1}{\Lnorm{\infty}} \norm{v_1}{\Lnorm{\infty}} \norm{\nabla v_{12}}{\Lnorm{2}}\norm{\sigma_1 v_{12}}{\Lnorm{2}}\\
		& ~~~~ \lesssim \delta \norm{\nabla v_{12}}{\Lnorm{2}}^2 + C_{\delta}( \norm{\sigma_1}{\Hnorm{2}}^2 + \norm{\sigma_2}{\Hnorm{2}}^2) \norm{v_2}{\Hnorm{2}}^4 \norm{\sigma_{12}}{\Lnorm{2}}^2\\
		& ~~~~ ~~~~ + C_\delta ( \norm{\sigma_1}{\Hnorm{2}}^2 \norm{v_{1}}{\Hnorm{2}}^2 + \norm{\sigma_1}{\Hnorm{2}}^2 \norm{v_2}{\Hnorm{2}}^2 +1)\norm{\sigma_1v_{12}}{\Lnorm{2}}^2.
	\end{align*}
	We have applied in the above estimates \eqref{ineq:embedding-weighted}
	\begin{equation*}
		\norm{v_{12}}{\Lnorm{q}} \lesssim \norm{\nabla  v_{12}}{\Lnorm{2}} + \norm{\sigma_1 v_{12}}{\Lnorm{2}},
	\end{equation*}
	for $ q \in [2,6] $. Meanwhile, after plugging in \subeqref{eq:differences-vacuum}{3},
	we have
	\begin{align*}
		& G_{6} =  \intw 2 \biggl( \sigma_{12} \int_0^z ( \sigma_2 \widetilde{\dvh v_2} + 2\widetilde{v_2} \cdot \nablah \sigma_2 ) \,dz' \times ( \dz v_2 \cdot v_{12} ) \\
		& ~~~~ ~~~~ - \sigma_1 (\sigma_1 \widetilde{\dvh v_1} + 2 \widetilde{v_1} \cdot \nablah \sigma_1) \abs{v_{12}}{2} \biggr) \idx \\
		& ~~~~ \lesssim \int_0^1 \bigl( \hnorm{\sigma_2}{\Lnorm{\infty}} \hnorm{\widetilde{\nablah v_2}}{\Lnorm{8}} + \hnorm{\widetilde{v_2}}{\Lnorm{\infty}} \hnorm{\nablah \sigma_2}{\Lnorm{8}} \bigr) \,dz'\\
		& ~~~~ ~~~~ \times \int_0^1 \hnorm{\sigma_{12}}{\Lnorm{2}} \hnorm{\dz v_{2}}{\Lnorm{8}} \hnorm{v_{12}}{\Lnorm{4}} \,dz \\
		& ~~~~  + \norm{\sigma_1 v_{12}}{\Lnorm{2}} \norm{v_{12}}{\Lnorm{6}}
		( \norm{\sigma_1}{\Lnorm{\infty}} \norm{\nabla v_1}{\Lnorm{3}} + \norm{v_1}{\Lnorm{\infty}} \norm{\nabla \sigma_1}{\Lnorm{3}}) \\
		&  \lesssim \delta \norm{\nabla v_{12}}{\Lnorm{2}}^2
		+ C_\delta \norm{\sigma_2}{\Hnorm{2}}^2 \norm{v_2}{\Hnorm{2}}^4 \norm{\sigma_{12}}{\Lnorm{2}}^2 \\
		& ~~~~ + C_\delta ( \norm{\sigma_1}{\Hnorm{2}}^2 \norm{v_1}{\Hnorm{2}}^2 +1)
		\norm{\sigma_1 v_{12}}{\Lnorm{2}}^2.
	\end{align*}
	We have applied in the above estimate \eqref{ineq:embedding-weighted} and the fact that $ \hnorm{\sigma_{12}}{\Lnorm{2}} = \norm{\sigma_{12}}{\Lnorm{2}} $. Also, after integrating by parts,
	\begin{align*}
		& G_7 = 2 \intw \biggl( \sigma_1 \int_0^z \bigl( \sigma_{12} \widetilde{\dvh v_1} + \sigma_2 \widetilde{\dvh v_{12}} + 2 \widetilde{v_1} \cdot \nablah \sigma_{12} + 2 \widetilde{v_{12}} \cdot\nablah \sigma_2 \bigr) \,dz' \\
		& ~~~~ \times ( \dz v_2 \cdot v_{12} ) \biggr) \idx =  - 4 \intw \biggl\lbrack \sigma_1 \biggl( \int_0^z \bigl( \sigma_{12} \widetilde{v_1} + \sigma_2 \widetilde{v_{12}} \bigr) \,dz' \biggr) \cdot \nablah (\dz v_2 \cdot v_{12} ) \biggr\rbrack \idx \\
		& ~~~~ - 2  \intw \biggl( \sigma_1 \int_0^z \bigl( \sigma_{12} \widetilde{\dvh v_1} + \sigma_2 \widetilde{\dvh v_{12}}\bigr) \,dz' \times (\dz v_2 \cdot v_{12})\biggr)\idx = : G_7' + G_7''.
		\end{align*}
		Notice that
		\begin{align*}
		& G_7' \lesssim \int_0^1 \biggl( \hnorm{\sigma_1}{\Lnorm{\infty}}\hnorm{\sigma_{12}}{\Lnorm{2}} \hnorm{\widetilde{v_{1}}}{\Lnorm{\infty}} + \hnorm{\sigma_2}{\Lnorm{\infty}} \hnorm{\widetilde{\sigma_1 v_{12}}}{\Lnorm{2}} \biggr) \,dz'\\
		& ~~~~ \times \int_0^1 \biggl( \hnorm{\nablah \dz v_2}{\Lnorm{4}}  \hnorm{v_{12}}{\Lnorm{4}} + \hnorm{\dz v_2}{\Lnorm{\infty}} \hnorm{\nablah v_{12}}{\Lnorm{2}} \biggr)   \,dz\\
		& ~~~~ \lesssim ( \norm{\sigma_1}{\Hnorm{2}}\norm{v_1}{\Hnorm{2}} \norm{\sigma_{12}}{\Lnorm{2}} + \norm{\sigma_2}{\Hnorm{2}} \norm{\sigma_1 v_{12}}{\Lnorm{2}})\\
		& ~~~~ ~~~~ \times  \norm{v_2}{\Hnorm{2}}^{1/2} \norm{v_2}{\Hnorm{3}}^{1/2} (\norm{\sigma_1 v_{12}}{\Lnorm{2}} + \norm{\nabla v_{12}}{\Lnorm{2}})\\
		& ~~~~ \lesssim \delta \norm{\nabla v_{12}}{\Lnorm{2}}^2 + C_\delta (\norm{v_2}{\Hnorm{3}}^2 + \norm{v_2}{\Hnorm{2}}^2 \norm{\sigma_1}{\Hnorm{2}}^4 \norm{v_1}{\Hnorm{2}}^{4} \\
		& ~~~~ ~~~~ + \norm{v_2}{\Hnorm{2}}^2\norm{\sigma_2}{\Hnorm{2}}^4 + 1 ) (\norm{\sigma_{12}}{\Lnorm{2}}^2 + \norm{\sigma_1 v_{12}}{\Lnorm{2}}^2),\\
		& \text{and} \\
		& G_7'' \lesssim \int_0^1 \biggl( \hnorm{\sigma_{12}}{\Lnorm{2}} \hnorm{\widetilde{\nablah v_1}}{\Lnorm{8}} + \hnorm{\sigma_2}{\Lnorm{8}} \hnorm{\widetilde{\nablah v_{12}}}{\Lnorm{2}} \biggr) \,dz' \\
		& ~~~~ ~~~~ \times  \int_0^1 \hnorm{\dz v_2}{\Lnorm{8}} \hnorm{\sigma_1 v_{12}}{\Lnorm{2}}^{1/4} \hnorm{\sigma_1v_{12}}{\Lnorm{6}}^{3/4} \,dz\\
		& ~~~~ \lesssim \int_0^1 \biggl( \hnorm{\sigma_{12}}{\Lnorm{2}} \hnorm{\widetilde{v_1}}{\Hnorm{2}} + \hnorm{\sigma_2}{\Hnorm{1}} \hnorm{\widetilde{\nablah v_{12}}}{\Lnorm{2}} \biggr) \,dz' \\
		& ~~~~ ~~~~ \times  \int_0^1 \hnorm{\dz v_2}{\Hnorm{1}} \hnorm{\sigma_1 v_{12}}{\Lnorm{2}}^{1/4} \hnorm{\sigma_1v_{12}}{\Lnorm{6}}^{3/4} \,dz \\
		& ~~~~ \lesssim (\norm{\sigma_{12}}{\Lnorm{2}} \norm{v_1}{\Hnorm{2}} + \norm{\sigma_2}{\Hnorm{2}} \norm{\nabla v_{12}}{\Lnorm{2}}) \norm{v_2}{\Hnorm{2}} \norm{\sigma_1 v_{12}}{\Lnorm{2}}^{1/4}\\
		& ~~~~ ~~~~ \times \norm{\sigma_1}{\Lnorm{\infty}}^{3/4} \norm{v_{12}}{\Lnorm{6}}^{3/4} \lesssim \delta \norm{\nabla v_{12}}{\Lnorm{2}}^2 + C_\delta \norm{v_{1}}{\Hnorm{2}}^2 \norm{\sigma_{12}}{\Lnorm{2}}^2 \\
		& ~~~~ + C_\delta (( \norm{\sigma_2}{\Hnorm{2}}^8 +1)\norm{v_2}{\Hnorm{2}}^8 \norm{\sigma_1}{\Hnorm{2}}^{6} + 1) \norm{\sigma_1 v_{12}}{\Lnorm{2}}^2.
	\end{align*}
	We denote $ C $ to be a positive constant depending on $ \mu, \lambda, B_1, B_2, T^* $ which may be different from line to line. Then after summing the estimates above, from \eqref{cn-dpcy-vacuum-001} and \eqref{cn-dpcy-vacuum-002} we have
	\begin{align*}
		& \dfrac{d}{dt} \bigl( \norm{\sigma_{12}}{\Lnorm{2}}^2 + C_{\mu,\lambda}'\norm{\sigma_1 v_{12}}{\Lnorm{2}}^2 \bigr) + c_{\mu,\lambda} \norm{\nabla v_{12}}{\Lnorm{2}}^2\\
		& ~~~~ \leq C (\norm{\dt v_2}{\Hnorm{1}}^2 + \norm{v_2}{\Hnorm{3}}^2 + 1  )\bigl(\norm{\sigma_{12}}{\Lnorm{2}}^2 + C_{\mu,\lambda}'\norm{\sigma_1 v_{12}}{\Lnorm{2}}^2\bigr),
	\end{align*}
	for some positive constants $ C_{\mu,\lambda}', c_{\mu,\lambda} $.
	Then after applying the Gr\"onwall's inequality, one has
	\begin{align*}
		& \sup_{0\leq t\leq T^*} \bigl( \norm{\sigma_{12}(t)}{\Lnorm{2}}^2 + C_{\mu,\lambda}' \norm{(\sigma_1v_{12})(t)}{\Lnorm{2}}^2 \bigr) + \int_0^{T^*}\norm{\nabla v_{12}}{\Lnorm{2}}^2 \,dt \\
		& ~~~~ \leq C \bigl( \norm{\sigma_{12,0}}{\Lnorm{2}}^2 + C_{\mu,\lambda}' \norm{\sigma_{1,0}v_{12,0}}{\Lnorm{2}}^2 \bigr) \leq C \bigl( \norm{\sigma_{12,0}}{\Lnorm{2}}^2 + \norm{v_{12,0}}{\Lnorm{2}}^2 \bigr).
	\end{align*}
	Therefore, after employing \eqref{ineq:embedding-weighted} and noticing the fact that we can interchange $ (\sigma_1, v_1), (\sigma_2,v_2) $ in the previous arguments, we will have the following:
	\begin{proposition}\label{prop:uniqueness-vacuum}
		Given two sets of initial data $ (\rho_{i,0},v_{i,0}) =  (\sigma_{i,0}^2 ,v_{i,0}),  i = 1,2 $, for \eqref{isen-CPE} satisfying \eqref{bound-of-initial-energy}, \eqref{isen-cptbl-conds} and \eqref{bound-of-initial}, the corresponding strong solutions $ (\rho_i,v_i) = (\sigma_i^2, v_i), i = 1,2 $, constructed in Proposition \ref{prop:exist-isen-vacuum} in the interval $[0,T^*] $, for some $ T^* > 0 $, satisfy
	\begin{equation*}
	\begin{aligned}
		& \norm{\sigma_1 - \sigma_2}{L^\infty(0,T^*;L^2(\Omega))} + \norm{\sigma_1(v_1- v_2)}{L^\infty(0,T^*;L^2(\Omega))}\\
		& ~~~~ + \norm{\sigma_2(v_1-v_2)}{L^\infty(0,T^*;L^2(\Omega))} + \norm{v_1 - v_2}{L^2(0,T^*;L^2(\Omega))} \\
		& ~~~~ + \norm{\nabla (v_1 - v_2)}{L^2(0,T^*;L^2(\Omega))}\\
		& ~~ \leq C_{\mu, \lambda, B_{1}, B_{2}, T^*}\Big(\norm{\sigma_{1,0}-\sigma_{2,0}}{L^2(\Omega))} + \norm{v_{1,0}-v_{2,0}}{L^2(\Omega))}\Big).
	\end{aligned}
	\end{equation*}
	In particular, if $ \rho_{1,0} = \rho_{2,0}, v_{1,0} = v_{2,0} $, we have $ \rho_1 =\rho_2, v_1 = v_2 $ in $ [0,T^*] $.
	\end{proposition}

\subsection{Proofs of the main theorems}
Theorem \ref{thm:gravity} follows from Proposition \ref{prop:exist-isen-g} and Proposition \ref{prop:uniqueness-g}. Theorem \ref{thm:vacuum} follows from Proposition \ref{prop:exist-isen-vacuum} and Proposition \ref{prop:uniqueness-vacuum}.

\section{A formulation of the free boundary problem for the atmospheric  dynamics}\label{sec:fm-fb}

In this section, we formulate the free boundary problem of \eqref{isen-CPE-g}. For simplicity, only the inviscid case is considered here. We write down the following form of inviscid compressible primitive equations: 
\begin{equation*}\tag{FBCPE}\label{isen-CPE-fb}
	\begin{cases}
		\dt \rho + \dvh (\rho v) + \dz (\rho w) = 0 & \text{in} ~ \Omega(t), \\
	\dt (\rho v) + \dvh (\rho v \otimes v) + \dz (\rho w v) + \nablah P = 0
	& \text{in} ~ \Omega(t),\\
	\dz P + \rho g = 0 & \text{in} ~ \Omega(t),
	\end{cases}
\end{equation*}
where 
 $ P = \rho^\gamma, \gamma > 1 $. Here $ \Omega(t) = \mathbb T^2 \times (0, Z(\vech{x},t)) $ is evolving domain, with $ z = 0 $ being the ground and $ z = Z(\vech{x},t) $ being the interface between the atmosphere and the vacuum universe. We impose the following boundary conditions
\begin{equation}\label{bc-fb}
	\begin{gathered}
	- \dt Z - v|_{z = Z(\vech{x},t)} \cdot \nablah Z + w|_{z = Z(\vech{x},t)}=0 ,\\
	 P|_{z = Z(\vech{x},t)} = 0,
	w|_{z = 0} = 0.
	\end{gathered}
\end{equation}
We remark here that the first boundary condition \subeqref{bc-fb}{1} represents the kinematic boundary condition on the moving boundary, while \subeqref{bc-fb}{2} represent that the pressure in continuous on the interface between the atmosphere and the vacuum universe, where the later vanishes; and the normal velocity vanishes on the ground (i.e., no penetration).

Similarly as before, the hydrostatic balance equation \subeqref{isen-CPE-fb}{3} , thanks to \subeqref{bc-fb}{2}, yields
\begin{align*}
	\rho^{\gamma-1}(\vech{x},z,t) & = \dfrac{\gamma-1}{\gamma} g(- z + Z(\vech{x},t)) + \rho^{\gamma-1}(\vech{x}, Z(\vech{x},t),t)\\
	& = \dfrac{\gamma-1}{\gamma} g(- z + Z(\vech{x},t)), ~~ \text{in} ~ \Omega(t).
\end{align*}
On the other hand, from \subeqref{isen-CPE-fb}{1}, we have
\begin{equation}\label{CPE-fb-001}
	\begin{aligned}
		& \dt \rho^{\gamma-1} + v \cdot \nablah \rho^{\gamma-1} + w \dz  \rho^{\gamma-1} \\
		& ~~~~ + (\gamma-1) \rho^{\gamma-1}(\dvh v + \dz w) = 0, ~~ \text{in} ~ \Omega(t).
	\end{aligned}
\end{equation}
Now we define the new coordinates as
\begin{equation}\label{new-coordinates}
	\begin{gathered}
		 \vech{x}' = (x',y') := \vech{x} = (x,y), ~ t' = t, \\
		 \eta := \dfrac{\rho^{\gamma-1}}{\rho^{\gamma-1}|_{z=0}} =  \dfrac{-z + Z(\vech{x},t)}{Z(\vech{x},t)}= 1 - \dfrac{z}{Z(\vech{x},t)} .
	 \end{gathered}
\end{equation}
In these new coordinates, $ \eta = 0 , 1 $ correspond to the upper atmosphere and the ground, respectively.
It is easy to verify $ \nabla_{\vech{x}'} Z = \nabla_{\vech{x}} Z, \dt Z = \partial_{t'} Z $.
Also we have the following change of variables for the differential operators
\begin{align*}
	\partial_{x} & = \partial_{x'} + \biggl( \dfrac{z}{Z^2} \partial_{x'} Z \biggr)  \partial_{\eta} = \partial_{x'} - \biggl( \dfrac{\eta-1}{Z} \partial_{x'} Z \biggr) \partial_{\eta}, \\
	\partial_{y} & = \partial_{y'} + \biggl( \dfrac{z}{Z^2} \partial_{y'} Z \biggr)  \partial_{\eta} = \partial_{y'} - \biggl( \dfrac{\eta-1}{Z} \partial_{y'} Z \biggr) \partial_{\eta}, \\
	\partial_z & = - \dfrac{1}{Z} \partial_\eta ,\\
	\partial_t & = \partial_{t'} + \biggl( \dfrac{z}{Z^2} \partial_{t'} Z \biggr)  \partial_{\eta} = \partial_{t'} - \biggl( \dfrac{\eta-1}{Z} \partial_{t'} Z \biggr) \partial_{\eta}, \\
	(\vech{x}', \eta) & \in \Omega_h \times(0,1) =: \Omega'.
\end{align*}
Now we rewrite equations \eqref{isen-CPE-fb}, \eqref{CPE-fb-001} in the new coordinates $ (\vech{x}',\eta,t') $:
\begin{equation*}
	\begin{cases}
		\partial_{t'} Z + v \cdot \nabla_{\vech{x}'} Z - w + (\gamma-1) \eta \bigl( Z \dv_{\vech{x}'} v \\
		~~~~ ~~~~ ~~~~ - (\eta-1) \nabla_{\vech{x}'} Z \cdot \partial_\eta v - \partial_{\eta} w\bigr) = 0 & \text{in} ~ \Omega',\\
		\rho \bigl(\partial_{t'} v + v \cdot \nabla_{\vech{x}'} v - \dfrac{1}{Z}( (\eta-1) \partial_{t'}Z + (\eta-1) v \cdot \nabla_{\vech{x}'} Z\\
		~~~~ ~~~~ ~~~~ + w )\partial_{\eta} v  + g \nabla_{\vech{x}'} Z  \bigr) = 0 
		& \text{in} ~ \Omega',\\
		\partial_\eta Z = 0 & \text{in} ~ \Omega',
	\end{cases}
\end{equation*}
where \begin{equation}\label{def:density} \rho = \biggl( \dfrac{\gamma-1}{\gamma} g \eta Z \biggr)^{1/(\gamma-1)}.\end{equation}
Moreover, we define the new unknown
\begin{equation}\label{def:new-vertical-v}
	W := - \dfrac{(\eta-1) \partial_{t'}Z + (\eta-1) v \cdot \nabla_{\vech{x}'} Z + w }{Z}.
\end{equation}
Then after dropping the prime sign, we end up with the following equations
\begin{equation}\label{rfeq:isen-CPE-fb}
	\begin{cases}
		\gamma \eta(\dt Z + v \cdot \nablah Z) + ( \gamma-1) \eta Z ( \dvh v + \partial_\eta W) + Z W = 0 & \text{in} ~ \Omega,\\
		\rho (\dt v + v \cdot \nablah v + W \partial_\eta v + g \nablah Z) = 0 
		& \text{in} ~ \Omega,\\
		\partial_\eta Z = 0 & \text{in} ~ \Omega,\\
		\rho = \biggl( \dfrac{\gamma-1}{\gamma} g \eta Z \biggr)^{1/(\gamma-1)} & \text{in} ~ \Omega,
	\end{cases}
\end{equation}
with $ \Omega = \Omega_h \times (0,1) $. We now find the corresponding form of the boundary conditions in \eqref{bc-fb}. Indeed, we have, recalling the definitions in \eqref{new-coordinates} and \eqref{def:new-vertical-v},
\begin{equation}\label{new-bc-01}
	W|_{\eta = 0 , 1}=0.
\end{equation}
Notice, the free boundary problem \eqref{isen-CPE-fb} now has been formulated as a fixed boundary problem \eqref{isen-CPE-fb} with the new unknowns $ (Z, v, W) $. We recall that $ Z $ is the graph of the free boundary and is also in proportion to the square of sound speed on the ground (recall that the square of sound speed on the ground is $ \partial_\rho P|_{\eta=1} = \gamma \rho^{\gamma-1}|_{\eta=1} = (\gamma-1) g Z $),
 from which the pressure on the ground can be recovered as $$ P_s = \bigl( \dfrac{\gamma-1}{\gamma} g Z\bigr)^{\frac{\gamma}{\gamma-1}}. $$ Observe that $ W $ serves as if it is the new `vertical velocity' in this formulation.

In the following, we shall perform some formal analysis on the reformulated equations \eqref{rfeq:isen-CPE-fb}. Multiply \subeqref{rfeq:isen-CPE-fb}{1} with $  \eta^{1/(\gamma-1) - 1} $ to obtain
\begin{equation*}
	\gamma \eta^{\frac{1}{\gamma-1}}(\dt Z + v \cdot \nablah Z) + (\gamma-1) \eta^{\frac{1}{\gamma-1}} \dvh v Z + (\gamma-1) Z \partial_\eta(\eta^{\frac{1}{\gamma-1}} W) = 0.
\end{equation*}
Noticing that $ \partial_\eta Z = 0 $, integrating the above equation along the vertical variable $ \eta $ yields
\begin{equation}\label{new-interface}
	(\gamma-1) \dt Z + \gamma \overline{\eta^{\frac{1}{\gamma-1}} v} \cdot \nablah Z + (\gamma-1) \overline{ \eta^{\frac{1}{\gamma-1}} \dvh v} Z =0,
\end{equation}
where we use, again, the notations
\begin{equation*}
	\overline{f} := \int_0^1 f \,d\eta.
\end{equation*}
Then after eliminating $ \dt Z $ from the above two equations, we have
\begin{align*}
	& (\gamma-1) Z \partial_\eta (\eta^{\frac{1}{\gamma-1}} W) = \dfrac{\gamma}{\gamma-1} \eta^{\frac{1}{\gamma-1}} \bigl( \gamma \overline{\eta^{\frac{1}{\gamma-1}} v} \cdot\nablah Z + ( \gamma-1) \overline{\eta^{\frac{1}{\gamma-1}} \dvh v} Z \bigr) \\
	& ~~~~ - \gamma \eta^{\frac{1}{\gamma-1}} v \cdot\nablah Z - (\gamma-1) \eta^{\frac{1}{\gamma-1}} \dvh v Z.
\end{align*}
Therefore, we have the following representation of $ W $ after using the fundamental theorem of calculus,
\begin{equation}\label{new-verticle-v}
	\begin{aligned}
		& (\gamma-1) \eta^{\frac{1}{\gamma-1}} Z W= \eta^{\frac{\gamma}{\gamma-1}}\bigl( \gamma \overline{\eta^{\frac{1}{\gamma-1}} v} \cdot\nablah Z + ( \gamma-1) \overline{\eta^{\frac{1}{\gamma-1}} \dvh v} Z \bigr) \\
		& ~~ - \gamma \int_0^\eta \zeta^{\frac{1}{\gamma-1}}v(\vech{x},\zeta,t) \,d\zeta \cdot\nablah Z - (\gamma-1) Z \int_0^\eta \zeta^{\frac{1}{\gamma-1}}\dvh v(\vech{x},\zeta,t) \,d\zeta.
	\end{aligned}
\end{equation}
This recovers the `vertical velocity'. We leave the study of \eqref{rfeq:isen-CPE-fb} for future work. We make a final comment on the deviation of \eqref{rfeq:isen-CPE-fb}. In \cite[Chapter 3]{Washington2005}, the $ \sigma $-coordinate system is introduced to reformulate \eqref{isen-CPE-fb}. Our choice of coordinates \eqref{new-coordinates} shares the same philosophy as the $ \sigma $-coordinates.

\subsection*{Acknowledgement}

The authors would like to thank the \'{E}cole Polytechnique for its kind hospitality, where this work was completed, and the \'{E}cole Polytechnique Foundation for its partial financial support through the 2017-2018 ``Gaspard Monge Visiting Professor" Program. This work is supported in part by the NSF grant number DMS-1516866 and by the ONR grant N00014-15-1-2333. The work of E.S.T. was also supported in part by the Einstein Stiftung/Foundation - Berlin, through the Einstein Visiting Fellow Program.

\bibliographystyle{plain}

\end{document}